\documentclass{amsart}

\usepackage{epsfig,epsf,fancybox}
\usepackage{amsmath}
\usepackage{mathrsfs,bbm}
\usepackage{amssymb}
\usepackage{graphicx}
\usepackage{color}
\usepackage{multirow}
\usepackage{paralist}
\usepackage{verbatim}
\usepackage{galois}
\usepackage{algorithm}
\usepackage{algorithmic}
\usepackage{boxedminipage}
\usepackage{booktabs}
\usepackage{accents}
\usepackage{stmaryrd}
\usepackage{subfig}
\usepackage{appendix}
\usepackage{graphicx}
\usepackage{epstopdf}
\usepackage{appendix}
\usepackage{cases}

\newtheorem{theorem}{Theorem}[section]
\newtheorem{lemma}[theorem]{Lemma}
\newtheorem{proposition}[theorem]{Proposition}
\newtheorem{corollary}[theorem]{Corollary}

\theoremstyle{definition}
\newtheorem{definition}[theorem]{Definition}
\newtheorem{example}[theorem]{Example}

\theoremstyle{remark}
\newtheorem{remark}[theorem]{Remark}

\numberwithin{equation}{section}

\newcommand{\f}{\mathscr{F}}
\newcommand{\lr}{\mathcal{L}}
\newcommand{\e}{\mathbb{E}}
\newcommand{\br}{\mathbb{R}}
\newcommand{\pr}{\mathcal{P}}

\newcommand{\dd}{\partial}

\newcommand{\tm}{\tilde{m}}
\newcommand{\hm}{\hat{m}}
\newcommand{\vp}{\varphi}
\newcommand{\brd}{\mathbb{R}^d}

\allowdisplaybreaks[2]

\begin{document}
\title[State-Density Flows of Mean-Field SDEs and Associated PDEs]{State-Density Flows of Non-Degenerate Density-Dependent Mean Field SDEs and Associated PDEs}

\author[Z. Huang]{Ziyu Huang}
\address{School of Mathematical Sciences, Fudan University, Shanghai 200433, China}
\email{zyhuang19@fudan.edu.cn}

\author[S. Tang]{Shanjian Tang}
\address{Department of Finance and Control Sciences, School of Mathematical Sciences, Fudan University, Shanghai 200433, China}
\email{sjtang@fudan.edu.cn}
\thanks{Research partially supported by National Natural Science Foundation of China (Grants No. 11631004 and No. 12031009).}

\subjclass[2020]{Primary 60H30, 60H10; Secondary 35K55}


\keywords{Mean Field SDE, PDE, McKean-Vlasov SDEs, Fokker-Planck equation}

\begin{abstract}
In this paper, we study a combined system of a Fokker-Planck (FP) equation for $m^{t,\mu}$ with initial $(t,\mu)\in[0,T]\times L^2(\mathbb{R}^d)$, and a stochastic differential equation  for $X^{t,x,\mu}$ with initial $(t,x)\in[0,T]\times \mathbb{R}^d$, whose coefficients depend on the solution of FP equation. We develop a combined probabilistic and analytical method to explore the regularity of the functional $V(t,x,\mu)=\mathbb{E}[\Phi(X^{t,x,\mu}_T,m^{t,\mu}(T,\cdot))]$. Our main result states that, under a non-degenerate condition and appropriate regularity assumptions on the coefficients, the function $V$ is the unique classical solution of a nonlocal partial differential equation of mean-field type. The proof depends heavily on the differential properties of the flow $\mu\mapsto (m^{t,\mu}, X^{t,x,\mu})$ over $\mu\in L^2(\mathbb{R}^d)$. We also give an example to illustrate the role of our main result. Finally, we give a discussion on the $L^1$ choice case in the initial $\mu$.
\end{abstract}

\maketitle

\section{Introduction}\label{sec1}
Let $(\Omega,\f,\{\f_t,\ 0\le t\le T\},\mathbb{P})$ denote a complete filtered probability space augmented by all the $\mathbb{P}$-null sets on which a $d$-dimensional Brownian motion $B$ is defined. In this paper, we consider the functional-dependent Fokker-Planck (FP) equation
\begin{equation}\label{pde_m}
	\left\{
	\begin{aligned}
		&\frac{\dd m^{t,\mu}}{\dd s}(s,x)-\sum_{i,j=1}^d\frac{\dd^2}{\dd x_i\dd x_j}[a_{ij}(x,m^{t,\mu}(s,\cdot))m^{t,\mu}(s,x)]\\
		&\qquad+\sum_{i=1}^d\frac{\dd}{\dd x_i}[b_i(x,m^{t,\mu}(s,\cdot))m^{t,\mu}(s,x)]=0,\qquad (s,x)\in(t,T]\times\brd;\\
		&m^{t,\mu}(t,x)=\mu(x),\quad x\in\brd,
	\end{aligned}
	\right.
\end{equation}
where $\mu\in L^2(\brd)$ is the initial probability density, and the dependent stochastic differential equation (SDE)
\begin{equation}\label{sde}
	\left\{
	\begin{aligned}
		&dX^{t,x,\mu}_s=b(X^{t,x,\mu}_s,m^{t,\mu}(s,\cdot))ds+\sigma(X^{t,x,\mu}_s,m^{t,\mu}(s,\cdot))dB_s,\quad s\in(t,T];\\
		&X^{t,x,\mu}_t=x\in\brd,
	\end{aligned}
	\right.
\end{equation}
where
\begin{align*}
	&b:\brd\times L^2(\brd)\to\brd,\quad \sigma:\brd\times L^2(\brd)\to \br^{d\times d},\\
	&b=(b_1,\dots,b_d)^*;\quad\sigma_{ij}=(\sigma)_{ij},\quad a_{ij}=\frac{1}{2}(\sigma\sigma^*)_{ij},\quad 1\le i,j\le d.
\end{align*}
FP equation \eqref{pde_m} can describe the probability density flow of a McKean-Vlasov SDE with the initial probability density $\mu$. For a functional defined on $\brd\times L^2(\brd)$, we develop a combined probabilistic and analytical method to study the differentiability of $(m^{t,\mu},X^{t,x,\mu})$ in the initial $(x,\mu)\in\brd\times L^2(\brd)$, and to explore the regularity of the function $V:[0,T]\times\brd\times L^2(\brd)\to\br$ defined by
\begin{equation*}\label{intr_v}
	V(t,x,\mu)=\e[\Phi(X^{t,x,\mu}_T,m^{t,\mu}(T,\cdot))],
\end{equation*}
where $m^{t,\mu}$ is a solution of FP equation \eqref{pde_m} and $X^{t,x,\mu}$ is a solution of SDE \eqref{sde}. Our main result states that, under a non-degenerate condition and appropriate regularity assumptions, the function $V$ defined above is the unique classical solution of the following partial differential equation (PDE) of mean-field type:
\begin{equation}\label{pde_v}
	\left\{
	\begin{aligned}
		&\frac{\dd V}{\dd t}(t,x,\mu)+\sum_{i=1}^{d}\frac{\dd V}{\dd x_i}(t,x,\mu)b_i(x,\mu)+\sum_{i,j=1}^{d}\frac{\dd^2V}{\dd x_i\dd x_j}(t,x,\mu)a_{ij}(x,\mu)\\
		&\qquad+\int_{\brd}\frac{\dd V}{\dd \mu}(t,x,\mu)(\xi)\Big\{\sum_{i,j=1}^d\frac{\dd^2}{\dd \xi_i\dd \xi_j}[a_{ij}(\xi,\mu(\cdot))\mu(\xi)]\\
		&\qquad-\sum_{i=1}^d\frac{\dd}{\dd \xi_i}[b_i(\xi,\mu(\cdot))\mu(\xi)]\Big\}d\xi=0,\\
		&\qquad(t,x,\mu)\in[0,T)\times\brd\times W^{2,2}(\brd);\\
		&V(T,x,\mu)=\Phi(x,\mu),\quad (x,\mu)\in\brd\times W^{2,2}(\brd).
	\end{aligned}
	\right.
\end{equation}

Our work is based on the derivative in $m\in L^2(\brd)$. Recall that a function $f:L^2(\brd)\to\br$ is differentiable, if there exists a function $L^2(\brd)\times\brd\ni(m,x)\mapsto \frac{\dd f}{\dd m}(m)(x)$ such that $\frac{\dd f}{\dd m}(m)(\cdot)\in L^2(\brd)$ for all $m\in L^2(\brd)$, and for any $\tilde{m}\in L^2(\brd)$,
\begin{equation*}\label{der}
	\frac{d}{d \theta}f(m+\theta \tilde{m})\Big|_{\theta=0}=\int_{\brd} \frac{\dd f}{\dd m}(m)(x)\tilde{m}(x)dx.
\end{equation*}
There are  various studies on the differentiability of functionals of probability distribution, namely, the $L$-derivative introduced by Lions \cite{PL} for functions defined on the space $\pr_2(\brd)$ and the linear functional derivative introduced by Carmona and Delarue \cite{CAR}, where $\pr_2(\br^d)$ denotes the space of probability measures over $\brd$ with finite second-order moments. We refer to \cite[Chapter 5]{CAR} for a discussion of relations between the $L$-derivative and the linear functional derivative. We also refer the reader to Bensoussan et al. \cite{AB4} for a discussion of relations between the differentiability of functions of probability distribution in $\pr_2(\brd)$ and of probability density in $L^2(\brd)$.

Based on the $L$-derivative, PDEs of the form \eqref{pde_v} have been already studied with different approaches and under various settings. We refer to \cite{AB4,BR,CHA,CD}. Buckdahn et al. \cite{BR} use a probabilistic method to investigate PDE \eqref{pde_v} of the probabilistic form
\begin{equation}\label{pde_v0}
	\left\{
	\begin{aligned}
		&\frac{\dd V}{\dd t}(t,x,P_{\xi})+\sum_{i=1}^d\frac{\dd V}{\dd x_i}(t,x,P_\xi)b_i(x,P_\xi)\\
		&\qquad+\frac{1}{2}\sum_{i,j,k=1}^d\frac{\dd^2V}{\dd x_i\dd x_j}(t,x,P_\xi)(\sigma_{ik}\sigma_{jk})(x,P_\xi)\\
		&\qquad+\e\Big[\frac{1}{2}\sum_{i,j,k=1}^d\frac{\dd }{\dd y_i}(\frac{\dd V}{\dd \mu})_j(t,x,P_\xi)(\xi)(\sigma_{ik}\sigma_{jk})(\xi,P_\xi)\\
		&\qquad+\sum_{i=1}^d(\frac{\dd V}{\dd \mu})_i(t,x,P_\xi)(\xi)b_i(\xi,P_\xi)\Big]=0,\\
		&\quad\qquad\qquad(t,x,P_\xi)\in[0,T)\times\brd\times \pr_2(\brd),\\
		&V(T,x,P_\xi)=\Phi(x,P_\xi),\quad (x,P_\xi)\in\brd\times \pr_2(\brd).
	\end{aligned}
	\right.
\end{equation}
They show that, if the coefficients $(b,\sigma)$ and the terminal condition $\Phi$ are twice differentiable in $(x,\mu)$ with bounded Lipschitz derivatives of first- and second-order, there is a unique classical solution of PDE \eqref{pde_v0}. Chassagneux et al. \cite{CHA} analyze a non-linear version of PDE \eqref{pde_v0} and show that such a PDE admits a classical solution on sufficiently small time intervals. They also give general examples such that the solution can be extended to arbitrary large intervals. Crisan and McMurray \cite{CD} prove the existence and uniqueness of the classical solutions of PDE \eqref{pde_v0} when the terminal condition $\Phi$ is not differentiable. Nonetheless, a non-degenerate condition is necessary and the third-order differentiability of the coefficients $(b,\sigma)$ in the distribution variable is required. Based on the linear functional derivative, de Raynal et al. \cite{de} give the existence and uniqueness result of the solutions of PDE \eqref{pde_v0} under the non-degenerate condition when the second-order linear functional derivatives of the coefficients $(b,\sigma)$ in the distribution variable are uniformly H\"older-continuous.

The $L$-derivative and the linear functional derivative can deal with the probability measure which does not admit a density. However, in our setting, the regularity assumption of coefficients is reduced: only the first-order differentiability in $m\in L^2(\brd)$ of coefficients $(b,\sigma)$ is required. We would like to emphasize that, our results still hold when $\mu\in L^2(\brd)$ is not necessarily a probability density. Moreover, a probability distribution with a square-integrable density might go beyond the space of probability measures of finite second-order moments (see Example~\ref{eg1}), and thus our measure variable space $W^{2,2}(\brd)$ in PDE \eqref{pde_v} is not necessarily included in $\pr_2(\brd)$ as it used to be impressed.

The formulation of taking the density as the primal variable in the augmented system of McKean-Vlasov SDEs can be found in Bensoussan et al. \cite{AB0,AB2} and Barbu and R\"ockner \cite{BAR}. The former deals with Mean-Field control problems and Mean-Field games, and the latter concerns the existence of weak solutions of the McKean-Vlasov SDEs for the case of Nemytskii-type coefficients. The density form of the FP equation \eqref{pde_m} on nonlinear Markov processes is justified by modeling natures in many occasions, and also by mathematical conveniences when the diffusion is non-degenerate (which usually allows a probability measure to admit a density) and the FP equation only depends on some characteristics of the probability measure.

In Sections~\ref{PDE1} and \ref{PDE2}, we give some estimates of $m^{t,\mu}$ and study the differentiability of $m^{t,\mu}$ with respect to the initial $\mu\in L^2(\brd)$ in the sense of Fr\'echet differentiability. The derivative is characterized as a  weak solution of a generalized PDE (see \eqref{k}), and the analysis relies on the non-degenerate condition, which is crucial to the regularity analysis of parabolic equations. The regularity of nonlinear PF equations with respect to the measure component is studied by Cardaliaguet et al. \cite{PC1} and, for higher order regularity, in the work of Tse \cite{TSE} for the case where $\sigma$ is a positive constant by using the linear functional derivative. We also refer to recent results on the smoothing properties of McKean-Vlasov SDEs. See \cite{BAN,CD,de1}, and the related result of Chassagneux et al. \cite{CHA1} on the weak error for propagation of chaos.

In Section~\ref{SDE}, we study the differentiability of the flow $(x,\mu)\mapsto X^{t,x,\mu}$ in the Hilbert space $\brd\times L^2(\brd)$. The derivative is characterized as a solution of a SDE which is associated with the derivative of $m^{t,\mu}$ in $\mu\in L^2(\brd)$ (see \eqref{sde_U}). We also refer to Buckdahn et al. \cite{BR} for the study of first- and second-order derivatives of $X^{t,x,\mu}$ with respect to the probability law when $\mu\in\pr_2(\brd)$ is a probability distribution.

The results in Sections~\ref{PDE1}-\ref{SDE} are used in Section~\ref{regularity} to prove the regularity of the value function $V(t,x,\mu)$. Section~\ref{PDE_V} is devoted to an It\^o's formula associated with mean-field problems, and it gives our main result, Theorem~\ref{main} and Corollary~\ref{cor}, stating that the value function $V$ is the unique classical solution of PDE \eqref{pde_v}. In Section~\ref{sec:exp}, we give an example to illustrate the role of our Theorem~\ref{main}. Finally, in Section~\ref{l1}, we give a discussion on the $L^1$ choice case in the initial $\mu$.

\subsection{Notations}
Let $(\Omega,\f,\{\f_t,\ 0\le t\le T\},\mathbb{P})$ denote a complete filtered probability space augmented by all the $\mathbb{P}$-null sets on which a $d$-dimensional Brownian motion $B=(B^1,\dots,B^d)=\{B_t,\ {0\le t\le T}\}$ is defined. Let $L^2(\f_t;\brd)$ denote the set of all $\f_t$-measurable square-integrable $\brd$-valued random variables. Let $\lr^2_{\f}(0,T)$ denote the set of all $\f_t$-progressively-measurable $\brd$-valued processes $\alpha=\{\alpha_t,\ 0\le t\le T\}$ such that $\e\big[\int_0^T |\alpha_t|^2dt\big]<+\infty$. Let $\mathcal{S}^2_{\f}[0,T]$ denote the set of all $\f_t$-progressively-measurable $\brd$-valued processes $\beta=\{\beta_t,\ 0\le t\le T\}$ such that $\e[\sup_{0\le t\le T} |\beta_t|^2]<+\infty$. In this paper, the notation $C(h_1,h_2,\dots,h_n)$ stands for a constant depending only on $(h_1,h_2,\dots,h_n)$.

We introduce the Sobolev spaces \cite{ARA}. Let $p$ be a positive real number. We denote by $L^p(\brd)$ the class of all measurable functions $u$ defined on $\brd$ for which
\begin{equation*}
	\|u\|_{L^p(\brd)}:=\Big(\int_{\brd}|u(x)|^pdx\Big)^{\frac{1}{p}}<+\infty.
\end{equation*}
The space $(L^p(\brd),\|\cdot\|_{L^p(\brd)})$ is a Banach space  \cite[p.29]{ARA} and $C_0^{\infty}(\brd)$ is dense in $L^p(\brd)$ if $1\le p<+\infty$  \cite[p.38]{ARA}. The space $(L^2(\brd),(\cdot,\cdot))$ is a Hilbert space \cite[p.31]{ARA} with respect to the inner product
\begin{equation*}
	(u,v):=\int_{\brd} u(x)v(x)dx,\quad u,v\in L^2(\brd).
\end{equation*}
For any positive integer $m$ and $1\le p<+\infty$, we denote by
\begin{equation*}
	W^{m,p}(\brd):=\{u\in L^p(\brd):D^\alpha u\in L^p(\brd),\  |\alpha|\le m\},
\end{equation*}
where $D^\alpha u$ is the generalized partial derivative. The norm of $u\in W^{m,p}(\brd)$ is defined by
\begin{equation*}
	\|u\|_{W^{m,p}(\brd)}:=\Big(\sum_{0\le |\alpha|\le m}\|D^\alpha u\|_{L^p(\brd)}^p\Big)^{\frac{1}{p}}.
\end{equation*}
The space $(W^{m,p}(\brd),\|\cdot\|_{W^{m,p}(\brd)})$ is a Banach space \cite[p.60]{ARA}.

\section{The Fokker-Planck equation}\label{PDE1}
In this section, we analyze the FP equation \eqref{pde_m} under the following assumptions. For notational convenience, we use the same constant $L>0$ for all the conditions below.

\textbf{(H1)} There exists $\gamma>0$, such that
\begin{align*}
	\sum_{i,j=1}^da_{ij}(x,m)\xi_i\xi_j\geq\gamma|\xi|^2,\quad \forall\xi\in\brd,\quad (x,m)\in\brd\times L^2(\brd).
\end{align*}

\textbf{(H2)} The function $b(\cdot,m):\brd\to\brd$ is differentiable and the function $\sigma(\cdot,m):\brd\to\br^{d\times d}$ is twice differentiable for all $m\in L^2(\brd)$. The functionals $f(x,\cdot):=b_i,\frac{\dd b_i}{\dd x_j},\sigma_{ij},\frac{\dd \sigma_{ij}}{\dd x_k},\frac{\dd^2\sigma_{ij}}{\dd x_k\dd x_l}(x,\cdot):L^2(\brd)\to \br$ have derivatives $\frac{\dd f}{\dd m}(x,m)(\cdot)\in W^{1,2}(\brd)$ for all $(x,m)\in\brd\times L^2(\brd)$ and $1\le i,j,k,l\le d$. The functionals and derivatives are uniformly bounded by $L$. That is,
\begin{align*}
	|f(x,m)|+\|\frac{\dd f}{\dd m}(x,m)(\cdot)\|_{W^{1,2}(\brd)}\le L,\qquad (x,m)\in \brd\times L^2(\brd).
\end{align*}
Moreover, the functionals $f(\cdot,\cdot):\brd\times L^2(\brd)\to \br$ and the derivatives $\frac{\dd f}{\dd m}(\cdot,\cdot):\brd\times L^2(\brd)\to L^2(\brd)$ are $L$-Lipschitz continuous. That is,
\begin{align*}
	&|f(x',m')-f(x,m)|+\|\frac{\dd f}{\dd m}(x',m')(\cdot)-\frac{\dd f}{\dd m}(x,m)(\cdot)\|_{L^2(\brd)}\\
	&\quad\le L\Big(|x'-x|+\|m'-m\|_{L^2(\brd)}\Big),\quad (x,m),\ (x',m') \in \brd\times L^2(\brd).
\end{align*}

The following example shows that a probability distribution admitting a density in $L^2(\brd)$ might be neither in $\pr_2(\brd)$ nor in $\pr_1(\brd)$.
\begin{example}\label{eg1}
	Let $d=1$ and the density be defined as
	\begin{align*}
		\rho(x)=\sum_{n=1}^{\infty}\frac{\pi^2}{6n}{\mathbbm 1}_{[n,n+\frac{1}{n})}(x),\quad x\in\br.
	\end{align*}
	It is easy to check that $\int_\br \rho(x)dx=1$ and $\rho\in L^2(\br)$. However,
	\begin{align*}
		&\int_\br |x|\rho(x)dx\geq \frac{\pi^2}{6}\sum_{n=1}^{\infty}\frac{1}{n}=+\infty,\qquad\int_\br |x|^2\rho(x)dx\geq \frac{\pi^2}{6}\sum_{n=1}^{\infty}1=+\infty.
	\end{align*}
\end{example}

For a differentiable function $F:L^2(\brd)\to\br$ (see Section~\ref{sec1} for the definition), we give the following example.
\begin{example}
	Let $h\in W^{1,2}(\brd)$. We define $F:L^2(\brd)\to \br$ as
	\begin{align*}
		F(m):=\exp(-f(m)^2), \quad 	f(m):=\int_{\brd} h(x)m(x)dx, \quad m\in L^2(\brd).
	\end{align*}
	Note that $|F|\le 1$ and for any $m,\tilde{m}\in L^2(\brd)$,
	\begin{align*}
		\lim_{\theta\to 0}\frac{1}{\theta}[F(m+\theta\tm)-F(m)]=-2f(m)\exp(-f(m)^2)\int_{\brd}h(x)\tm(x)dx,
	\end{align*}
	therefore,
	\begin{align*}
		\frac{\dd F}{\dd m}(m)(\cdot)=-2f(m)\exp(-f(m)^2)h(\cdot), \qquad \forall m\in L^2(\brd).
	\end{align*}
	Moreover, for any $m\in L^2(\brd)$,
	\begin{align*}
		&\|\frac{\dd F}{\dd m}(m)(\cdot)\|_{L^{2}(\brd)}\le \sqrt{2}\exp(-\frac{1}{2})\|h\|_{L^{2}(\brd)}<+\infty,\\
		&\|\frac{\dd F}{\dd m}(m)(\cdot)\|_{W^{1,2}(\brd)}\le \sqrt{2}\exp(-\frac{1}{2})\|h\|_{W^{1,2}(\brd)}<+\infty,
	\end{align*}
	and for any $m,m'\in L^2(\brd)$,
	\begin{align*}
		&|F(m')-F(m)|\le \sqrt{2}\exp(-\frac{1}{2})\|h\|_{L^{2}(\brd)}\|m'-m\|_{L^2(\brd)},\\
		&\|\frac{\dd F}{\dd m}(m')(\cdot)-\frac{\dd F}{\dd m}(m)(\cdot)\|_{L^{2}(\brd)}\le 2\|h\|^2_{L^2(\brd)}\|m'-m\|_{L^2(\brd)}.
	\end{align*}
\end{example}

The following lemma is proved in Appendix~\ref{pf_thm1}.
\begin{lemma}\label{thm1}
	Let Assumptions (H1) and (H2) be satisfied and $\mu\in L^{2}(\brd)$. Then, equation \eqref{pde_m} has a unique solution $m^{t,\mu}\in L^2([t,T];W^{1,2}(\brd))\cap L^{\infty}([t,T];L^{2}(\brd))$. Moreover, if $\mu\in W^{1,2}(\brd)$, then, $m^{t,\mu}\in L^2([t,T];W^{2,2}(\brd))\cap L^{\infty}([t,T];W^{1,2}(\brd))$ and $m^{t,\mu}_s\in L^2([t,T];L^2(\brd))$, and we have the estimates
	\begin{equation}\label{estimate_m}
		\begin{split}
			&\sup_{t\le s\le T}\|m^{t,\mu}(s,\cdot)\|_{L^{2}(\brd)}+\|m^{t,\mu}\|_{L^2([t,T];W^{1,2}(\brd))}
			\le C\big(\gamma,L,T,\|\mu\|_{L^{2}(\brd)}\big);\\
			&\sup_{t\le s\le T}\|m^{t,\mu}(s,\cdot)\|_{W^{1,2}(\brd)}+\|m^{t,\mu}\|_{L^2([t,T];W^{2,2}(\brd))}+\|m^{t,\mu}_{s}\|_{L^2([t,T];L^2(\brd))}\\
			&+\sup_{s\ne s'}\frac{\|m^{t,\mu}(s,\cdot)-m^{t,\mu}(s',\cdot)\|_{L^2(\brd)}}{|s-s'|^{\frac{1}{2}}} \le C\big(\gamma,L,T,\|\mu\|_{W^{1,2}(\brd)}\big).
		\end{split}
	\end{equation}
\end{lemma}

Now we discuss the continuous dependence of $m^{t,\mu}$ on $\mu$. The following proposition is proved in Appendix~\ref{pf_lemma6}.
\begin{proposition}\label{lemma6}
	Let Assumptions (H1) and (H2) be satisfied, and $\mu,\mu'\in W^{1,2}(\brd)$. Let $m^{t,\mu}$ and $m^{t,\mu'}$ be solutions of equation \eqref{pde_m} with initials $\mu$ and $\mu'$, respectively, such that \eqref{estimate_m} holds true. Then,
	\begin{equation}\label{lem6}
		\begin{split}
			&\sup_{t\le s\le T}\|m^{t,\mu'}(s,\cdot)-m^{t,\mu}(s,\cdot)\|_{L^{2}(\brd)}+\|m^{t,\mu'}-m^{t,\mu}\|_{L^2([t,T];W^{1,2}(\brd))}\\
			&\qquad\le C\big(\gamma,L,T,\|\mu\|_{W^{1,2}(\brd)}\big)\|\mu'-\mu\|_{L^{2}(\brd)};\\
			&\sup_{t\le s\le T}\|m^{t,\mu'}(s,\cdot)-m^{t,\mu}(s,\cdot)\|_{W^{1,2}(\brd)}+\|m^{t,\mu'}-m^{t,\mu}\|_{L^2([t,T];W^{2,2}(\brd))}\\
			&\qquad \le C\big(\gamma,L,T,\|\mu\|_{W^{1,2}(\brd)}\big)\|\mu'-\mu\|_{W^{1,2}(\brd)}.
		\end{split}
	\end{equation}
\end{proposition}

In the rest of this section, we discuss the higher regularity of the solution $m^{t,\mu}$ of equation \eqref{pde_m} under some higher regularity assumptions. The following assumption is the regularity-enhanced version of Assumption (H2).

\textbf{(H2')} The functions $b$ and $\sigma$ satisfy (H2). Moreover, $b(\cdot,m):\brd\to\brd$ is twice differentiable and $a(\cdot,m):\brd\to\br^{d\times d}$ is three times differentiable for all $m\in L^2(\brd)$, with the derivatives being bounded by $L$. That is,
\begin{equation*}
	\begin{split}
		&|\frac{\partial^2 b}{\partial x_i\dd x_j}(x,m)|+ |\frac{\partial^3 a}{\partial x_i\dd x_j\dd x_k}(x,m)| \le L, \quad (x,m)\in \brd\times L^2(\brd),\quad 1\le i,j,k\le d.
	\end{split}
\end{equation*}

The following proposition is proved in Appendix~\ref{pf_thm1'}.
\begin{proposition}\label{thm1'}
	Let Assumptions (H1) and (H2) be satisfied, and $\mu,\mu'\in W^{2,2}(\brd)$. Let $m^{t,\mu}$ and $m^{t,\mu'}$ be solutions of equation \eqref{pde_m} with initials $\mu$ and $\mu'$, respectively, such that \eqref{estimate_m} holds true. Then,
	\begin{align}
		&\sup_{t\le s\le T}\|m^{t,\mu}_s(s,\cdot)\|_{L^2(\brd)}+\sup_{t\le s\le T}\| m^{t,\mu}(s,\cdot)\|_{W^{2,2}(\brd)}\label{thm1'_1}\\
		&+\|m^{t,\mu}_{s}\|_{L^2([t,T];W^{1,2}(\brd))}\le C\big(\gamma,L,T,\|\mu\|_{W^{2,2}(\brd)}\big);\notag\\
		&\sup_{t\le s\le T}\|m_s^{t,\mu'}(s,\cdot)-m_s^{t,\mu}(s,\cdot)\|_{L^{2}(\brd)}+\|m_{s}^{t,\mu'}-m^{t,\mu}_{s}\|_{L^2([t,T];W^{1,2}(\brd))}\label{thm1'_2}\\
		&+\sup_{t\le s\le T}\|m^{t,\mu'}(s,\cdot)-m^{t,\mu}(s,\cdot)\|_{W^{2,2}(\brd)}\notag\\
		&\qquad\le C\big(\gamma,L,T,\|\mu\|_{W^{2,2}(\brd)},\|\mu'\|_{W^{2,2}(\brd)}\big)\|\mu'-\mu\|^2_{W^{2,2}(\brd)}.\notag
	\end{align}
	Further, if Assumption (H2') is satisfied, then,
	\begin{equation}\label{thm1''_1}
		\begin{split}
			&\|m^{t,\mu}\|_{L^2([t,T];W^{3,2}(\brd))}+\sup_{s\neq s'}\frac{\|m^{t,\mu}(s',\cdot)-m^{t,\mu}(s,\cdot)\|_{W^{1,2}(\brd)}}{|s'-s|^{\frac{1}{2}}}\\
			&\qquad\le C(\gamma,L,T,\|\mu\|_{W^{2,2}(\brd)});
		\end{split}
	\end{equation}
	if moreover $\mu\in W^{3,2}(\brd)$, then,
	\begin{align}
		&\sup_{t\le s\le T}\|m^{t,\mu}_{s}(s,\cdot)\|_{W^{1,2}(\brd)}+\sup_{t\le s\le T}\|m^{t,\mu}(s,\cdot)\|_{W^{3,2}(\brd)}\label{cor2_1}\\
		&+\|m^{t,\mu}_{ss}\|_{L^2([t,T];L^2(\brd))}+\|m^{t,\mu}_{s}\|_{L^2([t,T];W^{2,2}(\brd))}\notag\\
		&\qquad\le C\big(\gamma,L,T,\|\mu\|_{W^{3,2}(\brd)}\big);\notag\\
		&\sup_{s\neq s'}\frac{\|m^{t,\mu}_s(s',\cdot)-m^{t,\mu}_s(s,\cdot)\|_{L^2(\brd)}}{|s'-s|^{\frac{1}{2}}}+\sup_{s\neq s'}\frac{\|m^{t,\mu}(s',\cdot)-m^{t,\mu}(s,\cdot)\|_{W^{2,2}(\brd)}}{|s'-s|^{\frac{1}{2}}}\label{thm1'''_1}\\
		&\qquad\le C(\gamma,L,T,\|\mu\|_{W^{3,2}(\brd)}).\notag
	\end{align}
\end{proposition}

\begin{remark}
	Rich results on (local) FP equations which are second-order quasilinear parabolic PDEs of a divergent-form are available in Ladyženskaja et al. \cite{LO}. FP equation \eqref{pde_m} is nonlocal for  the coefficients $\{a_{ij},b_i,\ 1\le i,j\le n\}$ depend on the unknown in a functional way.
\end{remark}

\section{Derivative of $m^{t,\mu}$}\label{PDE2}
In this section, we compute the derivative of $m^{t,\mu}$ with respect to $\mu$. Let $\tilde{\mu}\in L^2(\brd)$, which is considered to be a "direction", and define
\begin{equation*}\label{mu_h:def}
	\mu_h:=\mu+h\tilde{\mu},\qquad h\in[0,1].
\end{equation*}
Let $m^{t,\mu_h}$ be the solution of equation \eqref{pde_m} corresponding to initial $\mu_h$, and set
\begin{equation}\label{m_h:def}
	\tilde{m}_h^{t,\mu}(\tilde{\mu}):=\frac{m^{t,\mu_h}-m^{t,\mu}}{h},\qquad h\in(0,1].
\end{equation}
Then, $\tilde{m}^{t,\mu}_h(\tilde{\mu})(t,x)=\tilde{\mu}(x)$ for $x\in\brd$, and for $(s,x)\in(t,T]\times\brd$,
\begin{align*}
	&\frac{\dd \tilde{m}^{t,\mu}_h(\tilde{\mu})}{\dd s}(s,x)-\sum_{i,j=1}^d\frac{\dd^2}{\dd x_i\dd x_j}\Big[a_{ij}(x,m^{t,\mu_h}(s,\cdot))\tilde{m}^{t,\mu}_h(\tilde{\mu})(s,x)\\
	&\qquad+\frac{1}{h}[a_{ij}(x,m^{t,\mu_h}(s,\cdot))-a_{ij}(x,m^{t,\mu}(s,\cdot))]m^{t,\mu}(s,x)\Big]\\
	&\qquad+\sum_{i=1}^d\frac{\dd}{\dd x_i}\Big[b_i(x,m^{t,\mu_h}(s,\cdot))\tilde{m}^{t,\mu}_h(\tilde{\mu})(s,x)\\
	&\qquad+\frac{1}{h}[b_i(x,m^{t,\mu_h}(s,\cdot))-b_i(x,m^{t,\mu}(s,\cdot))]m^{t,\mu}(s,x)\Big]=0.
\end{align*}
From Assumptions (H1) and (H2) and standard arguments for second-order parabolic equations, for $\mu\in W^{1,2}(\brd)$, we have the following estimate:
\begin{equation}\label{lem3_1}
	\begin{split}
		&\sup_{t\le s\le T}\|\tilde{m}^{t,\mu}_h(\tilde{\mu})(s,\cdot)\|_{L^{2}(\brd)}+\|\tilde{m}^{t,\mu}_h(\tilde{\mu})\|_{L^2([t,T];W^{1,2}(\brd))}\\
		&\qquad\le C\big(\gamma,L,T,\|\mu\|_{W^{1,2}(\brd)}\big)\|\tilde{\mu}\|_{L^2(\brd)}.
	\end{split}
\end{equation}
Along the direction $\tilde{\mu}$, the directional derivative of $m^{t,\mu}$ with respect to $\mu$ can be formulated as the solution of the following PDE:
\begin{equation}\label{tm}
	\left\{
	\begin{aligned}
		&\frac{\dd \tilde{m}^{t,\mu}(\tilde{\mu})}{\dd s}(s,x)-\sum_{i,j=1}^d\frac{\dd^2}{\dd x_i\dd x_j}\Big[a_{ij}(x,m^{t,\mu}(s,\cdot))\tilde{m}^{t,\mu}(\tilde{\mu})(s,x)\\
		&\qquad+\int_{\brd}\frac{\dd a_{ij}}{\dd m}(x,m^{t,\mu}(s,\cdot))(\xi)\tm^{t,\mu}(\tilde{\mu})(s,\xi)d\xi \cdot m^{t,\mu}(s,x)\Big]\\
		&\qquad+\sum_{i=1}^d\frac{\dd}{\dd x_i}\Big[b_i(x,m^{t,\mu}(s,\cdot))\tilde{m}^{t,\mu}(\tilde{\mu})(s,x)\\
		&\qquad+\int_{\brd}\frac{\dd b_i}{\dd m}(x,m^{t,\mu}(s,\cdot))(\xi)\tm^{t,\mu}(\tilde{\mu})(s,\xi)d\xi \cdot m^{t,\mu}(s,x)\Big]=0,\\
		&\qquad\qquad\qquad\qquad\qquad\qquad\qquad\qquad(s,x)\in(t,T]\times\brd,\\
		&\tilde{m}^{t,\mu}(\tilde{\mu})(t,x)=\tilde{\mu}(x),\quad x\in\brd.
	\end{aligned}
	\right.
\end{equation}
The following lemma is proved in Appendix~\ref{pf_thm4}.
\begin{lemma}\label{thm4}
	Let Assumptions (H1) and (H2) be satisfied, $\mu\in W^{1,2}(\brd)$ and $\tilde{\mu}\in L^2(\brd)$. Let $m^{t,\mu}$ be the solution of equation \eqref{pde_m} such that \eqref{estimate_m} holds true. Then, equation \eqref{tm} has a unique solution $\tm^{t,\mu}(\tilde{\mu})\in C^0([t,T];L^2(\brd))$, such that
	\begin{equation}\label{estimate_tm1}
		\begin{split}
			&\sup_{t\le s\le T}\|\tm^{t,\mu}(\tilde{\mu})(s,\cdot)\|_{L^{2}(\brd)}+\|\tm^{t,\mu}(\tilde{\mu})\|_{L^2([t,T];W^{1,2}(\brd))}\\
			&\qquad\le C\big(\gamma,L,T,\|\mu\|_{W^{1,2}(\brd)}\big)\|\tilde{\mu}\|_{L^2(\brd)}.
		\end{split}
	\end{equation}
	For any $\mu,\mu'\in W^{1,2}(\brd)$, we have the estimate
	\begin{equation}\label{estimate_tm2}
		\begin{split}
			&\sup_{t\le s\le T}\|\tm^{t,\mu'}(\tilde{\mu})(s,\cdot)-\tm^{t,\mu}(\tilde{\mu})(s,\cdot)\|_{L^{2}(\brd)}\\
			&\qquad\le C\big(\gamma,L,T,\|\mu\|_{W^{1,2}(\brd)},\|\mu'\|_{W^{1,2}(\brd)}\big)\|\tilde{\mu}\|_{L^{2}(\brd)}\|\mu'-\mu\|_{L^2(\brd)}.
		\end{split}
	\end{equation}
\end{lemma}

We have the following convergence of $\tm_h^{t,\mu}(\tilde{\mu})$ to $\tilde{m}^{t,\mu}(\tilde{\mu})$ as $h$ goes to $0$, whose proof is given in Appendix~\ref{pf_lem5}.
\begin{lemma}\label{lem5}
	Let Assumptions (H1) and (H2) be satisfied, $\mu\in W^{1,2}(\brd)$ and $\tilde{\mu}\in L^2(\brd)$. Let $\tm^{t,\mu}(\tilde{\mu})\in C^0([t,T];L^2(\brd))$ be the solution of equation \eqref{tm}. Then,
	\begin{align*}
		&\sup_{t\le s\le T}\|\frac{m^{t,\mu+h\tilde{\mu}}-m^{t,\mu}}{h}(s,\cdot)-\tm^{t,\mu}(\tilde{\mu})(s,\cdot)\|_{L^2(\brd)}\\
		&\qquad\le C\big(\gamma,L,T,\|\mu\|_{W^{1,2}(\brd)},\|\tilde{\mu}\|_{L^{2}(\brd)}\big)h.
	\end{align*}
	That is, $\tm^{t,\mu}(\tilde{\mu})(s,\cdot)$ is the derivative of $m^{t,\mu}(s,\cdot)$ along the direction $\tilde{\mu}\in L^2(\brd)$.
\end{lemma}

Lemma~\ref{lem5} and estimates \eqref{estimate_tm1} and \eqref{estimate_tm2} show that the mapping
\begin{align*}
	L^2(\brd)\ni\mu\mapsto m^{t,\mu}(s,\cdot)\in L^2(\brd)
\end{align*}
is Fr\'echet differentiable for $\mu\in W^{1,2}(\brd)$. Its Fr\'echet derivative $D_\mu m^{t,\mu}(s,\cdot)$ is
\begin{align*}
	D_\mu m^{t,\mu}(s,\cdot)(\tilde{\mu})=\tm^{t,\mu}(\tilde{\mu})(s,\cdot),\qquad\tilde{\mu}\in L^2(\brd).
\end{align*}
To represent the Fr\'echet derivative $D_\mu m^{t,\mu}(s,\cdot)$, we consider the following PDE:
\begin{equation}\label{k}
	\left\{
	\begin{aligned}
		&\frac{\dd k^{t,\mu}}{\dd s}(s,x,y)-\sum_{i,j=1}^d\frac{\dd^2}{\dd x_i\dd x_j}\Big[a_{ij}(x,m^{t,\mu}(s,\cdot))k^{t,\mu}(s,x,y)\\
		&\qquad+\int_{\brd}\frac{\dd a_{ij}}{\dd m}(x,m^{t,\mu}(s,\cdot))(\xi)k^{t,\mu}(s,\xi,y)d\xi \cdot m^{t,\mu}(s,x)\Big]\\
		&\qquad+\sum_{i=1}^d\frac{\dd}{\dd x_i}\Big[b_i(x,m^{t,\mu}(s,\cdot))k^{t,\mu}(s,x,y)\\
		&\qquad+\int_{\brd}\frac{\dd b_i}{\dd m}(x,m^{t,\mu}(s,\cdot))(\xi)k^{t,\mu}(s,\xi,y)d\xi \cdot m^{t,\mu}(s,x)\Big]=0,\\
		&\qquad\qquad\qquad\qquad\qquad\qquad(s,x,y)\in(t,T]\times\brd\times\brd,\\
		&k^{t,\mu}(t,x,y)=\delta(x-y),\quad (x,y)\in\brd\times\brd.
	\end{aligned}
	\right.
\end{equation}
The following proposition is proved in Appendix~\ref{pf_lem7}.
\begin{proposition}\label{lem7}
	Let Assumptions (H1) and (H2) be satisfied and $\mu\in W^{1,2}(\brd)$. Let $m^{t,\mu}$ be the solution of equation \eqref{pde_m} such that \eqref{estimate_m} holds true and let $k^{t,\mu}$ be a weak solution of equation \eqref{k}. Then, for any $\tilde{\mu}\in L^2(\brd)$, the directional derivative $\tm^{t,\mu}(\tilde{\mu})(s,\cdot)$ of $m^{t,\mu}(s,\cdot)$ satisfies
	\begin{align*}
		\tm^{t,\mu}(\tilde{\mu})(s,x)=\int_{\brd} k^{t,\mu}(s,x,y)\tilde{\mu}(y)dy,\quad (s,x)\in [t,T]\times\brd.
	\end{align*}
	That is, the Fr\'echet derivative $D_\mu m^{t,\mu}(s,\cdot)$ of $m^{t,\mu}(s,\cdot)$ satisfies
	\begin{align*}
		D_\mu m^{t,\mu}(s,\cdot)(\tilde{\mu})=\int_{\brd} k^{t,\mu}(s,\cdot,y)\tilde{\mu}(y)dy,\qquad\tilde{\mu}\in L^2(\brd).
	\end{align*}
\end{proposition}

In the rest of this section, we give the existence, the boundedness and the continuity of solutions of   equation \eqref{k}. We solve equation \eqref{k} via solutions of the following two PDEs: one for $f^{t,\mu}$ defined on $[t,T]\times\brd\times\brd$ as
\begin{equation}\label{f}
	\left\{
	\begin{aligned}
		&\frac{\dd f^{t,\mu}}{\dd s}(s,x,y)-\sum_{i,j=1}^d\frac{\dd^2}{\dd x_i\dd x_j}[a_{ij}(x,m^{t,\mu}(s,\cdot))f^{t,\mu}(s,x,y)]\\
		&\qquad+\sum_{i=1}^d\frac{\dd}{\dd x_i}[b_i(x,m^{t,\mu}(s,\cdot))f^{t,\mu}(s,x,y)]=0,\\
		&\qquad\qquad\qquad\qquad(s,x,y)\in(t,T]\times\brd\times\brd,\\
		&f^{t,\mu}(t,x,y)=\delta(x-y),\quad (x,y)\in\brd\times\brd;
	\end{aligned}
	\right.
\end{equation}
and the other one for $g^{t,\mu}$ defined on $[t,T]\times\brd\times\brd$ as
\begin{equation}\label{g}
	\left\{
	\begin{aligned}
		&\frac{\dd g^{t,\mu}}{\dd s}(s,x,y)-\sum_{i,j=1}^d\frac{\dd^2}{\dd x_i\dd x_j}\Big[a_{ij}(x,m^{t,\mu}(s,\cdot))g^{t,\mu}(s,x,y)\\
		&\qquad+\int_{\brd}\frac{\dd a_{ij}}{\dd m}(x,m^{t,\mu}(s,\cdot))(\xi)g^{t,\mu}(s,\xi,y)d\xi \cdot m^{t,\mu}(s,x)\Big]\\
		&\qquad+\sum_{i=1}^d\frac{\dd}{\dd x_i}\Big[b_i(x,m^{t,\mu}(s,\cdot))g^{t,\mu}(s,x,y)\\
		&\qquad+\int_{\brd}\frac{\dd b_i}{\dd m}(x,m^{t,\mu}(s,\cdot))(\xi)g^{t,\mu}(s,\xi,y)d\xi \cdot m^{t,\mu}(s,x)\Big]\\
		&=\sum_{i,j=1}^d\frac{\dd^2}{\dd x_i\dd x_j}\Big[\int_{\brd}\frac{\dd a_{ij}}{\dd m}(x,m^{t,\mu}(s,\cdot))(\xi)f^{t,\mu}(s,\xi,y)d\xi \cdot m^{t,\mu}(s,x)\Big]\\
		&\qquad-\sum_{i=1}^d\frac{\dd}{\dd x_i}\Big[\int_{\brd}\frac{\dd b_i}{\dd m}(x,m^{t,\mu}(s,\cdot))(\xi)f^{t,\mu}(s,\xi,y)d\xi \cdot m^{t,\mu}(s,x)\Big],\\
		&\qquad\qquad\qquad\qquad\qquad\qquad\qquad\qquad(s,x,y)\in(t,T]\times\brd\times\brd,\\
		&g^{t,\mu}(t,x,y)=0,\quad (x,y)\in\brd\times\brd.
	\end{aligned}
	\right.
\end{equation}
Obviously, if $f^{t,\mu}$ and $g^{t,\mu}$ are weak solutions of equations \eqref{f} and \eqref{g}, respectively, then,
\begin{align*}
	k^{t,\mu}(s,x,y):=f^{t,\mu}(s,x,y)+g^{t,\mu}(s,x,y),\quad (s,x,y)\in[t,T]\times\brd\times\brd,
\end{align*}
is a weak solution of equation \eqref{k}. To estimate the solution of equation \eqref{k}, we need to estimate solutions of equations \eqref{f} first and then \eqref{g}. We now describe our Assumption (H3).

\textbf{(H3)} For $(t,\mu)\in [0,T]\times W^{1,2}(\brd)$, equation \eqref{f} has a weak solution $f^{t,\mu}$, such that for any $\varphi\in L^{2}(\brd)$, the function $\phi^{t,\mu}$ defined as
\begin{align}\label{phi}
	\phi^{t,\mu}(s,y):=\int_{\brd} f^{t,\mu}(s,x,y)\varphi(x)dx,\quad (s,y)\in[t,T]\times\brd,
\end{align}
satisfies
\begin{equation}\label{lem8_1}
	\begin{split}
		&\sup_{t\le s\le T}\|\phi^{t,\mu}(s,\cdot)\|_{L^{2}(\brd)}+\|\phi^{t,\mu}\|_{L^2([t,T];W^{1,2}(\brd))} \le C(\gamma,L,T)\|\varphi\|_{L^{2}(\brd)}.
	\end{split}
\end{equation}
If moreover $\varphi\in W^{1,2}(\brd)$, then,
\begin{equation}\label{lem8_1'}
	\begin{split}
		&\|\phi_s^{t,\mu}\|_{L^{2}([t,T];L^2(\brd))} \le C(\gamma,L,T)\|\varphi\|_{W^{1,2}(\brd)}.
	\end{split}
\end{equation}
For $\mu,\mu'\in W^{1,2}(\brd)$, functions $\phi^{t,\mu},\phi^{t,\mu'}$ satisfy
\begin{equation}\label{lem10_1}
	\begin{split}
		&\sup_{t\le s\le T}\|\phi^{t,\mu'}(s,\cdot)-\phi^{t,\mu}(s,\cdot)\|_{L^{2}(\brd)}\\
		&\qquad\le C(\gamma,L,T,\|\mu\|_{W^{1,2}(\brd)}){\|\varphi\|_{W^{1,2}(\brd)}}\|\mu'-\mu\|_{L^2(\brd)}.
	\end{split}
\end{equation}

\begin{remark}\label{rm1}
	Let Assumptions (H1) and (H2) be satisfied. If coefficients $(b,\sigma)$ are independent of $x$, then Assumption (H3) holds true, which is proved in Lemma~\ref{exp1}.
\end{remark}

We now consider the existence, boundedness, and the continuous dependence of $g^{t,\mu}$ on $\mu$. We first claim that the right-hand side of equation \eqref{g} belongs to the class $L^2([t,T];L^2(\brd\times\brd))$. Actually, from \eqref{lem8_1} and Assumption (H2), we have for any $t\le s\le T$ and $1\le i,j\le d$,
\begin{equation*}\label{bf}
	\begin{split}
		& \int_{\brd}\int_{\brd} \Big|\int_{\brd}\frac{\dd a_{ij}}{\dd m}(x,m^{t,\mu}(s,\cdot))(\xi)f^{t,\mu}(s,\xi,y)d\xi \cdot m_{x_ix_j}^{t,\mu}(s,x)\Big|^2 dydx\\
		&\qquad= \int_{\brd}\int_{\brd}\Big|\int_{\brd}\frac{\dd a_{ij}}{\dd m}(x,m^{t,\mu}(s,\cdot))(\xi)f^{t,\mu}(s,\xi,y)d\xi \Big|^2dy\cdot|m_{x_ix_j}^{t,\mu}(s,x)|^2dx\\
		&\qquad\le C(\gamma,L,T)\int_{\brd} \|\frac{\dd a}{\dd m}(x,m^{t,\mu}(s,\cdot))(\cdot)\|^2_{L^2(\brd)}\cdot |m_{x_ix_j}^{t,\mu}(s,x)|^2dx\\
		&\qquad\le C(\gamma,L,T)\|m_{x_ix_j}^{t,\mu}(s,\cdot)\|^2_{L^2(\brd)}.
	\end{split}
\end{equation*}
So from Lemma~\ref{thm1}, we have
\begin{align*}
	&\int_t^T\int_{\brd}\int_{\brd} \Big|\int_{\brd}\frac{\dd a_{ij}}{\dd m}(x,m^{t,\mu}(s,\cdot))(\xi)f^{t,\mu}(s,\xi,y)d\xi \cdot m_{x_ix_j}^{t,\mu}(s,x)\Big|^2 dydxds\\
	&\qquad\le C(\gamma,L,T)\|m_{x_ix_j}^{t,\mu}\|^2_{L^2([t,T];L^2(\brd))}\le C(\gamma,L,T,\|\mu\|_{W^{1,2}(\brd)}).
\end{align*}
In a similar way, we have the last claim. From \cite[Definition and Remark, p.374; Theorems 3 and 4, p.378]{ELC}, similar to the proof of Lemma~\ref{thm4}, we know that equation \eqref{g} has a unique solution in $C^0([t,T];L^2(\brd\times\brd))$. Moreover, we have the following estimate, which is proved in Appendix~\ref{pf_lem9}.

\begin{lemma}\label{lem9}
	Let Assumptions (H1)-(H3) be satisfied. Let $\mu,\mu'\in W^{1,2}(\brd)$ and let $g^{t,\mu},g^{t,\mu'}$ be solutions of equation \eqref{g} corresponding to $\mu$ and $\mu'$, respectively. Then,
	\begin{align}
		&\sup_{s\in[t,T]}\|g^{t,\mu}(s,\cdot,\cdot)\|_{W^{1,2}(\brd\times\brd)}+\|g^{t,\mu}_s\|_{L^2([t,T];L^{2}(\brd\times\brd))}\label{lem9_1}\\
		&\qquad\le C\big(\gamma,L,T,\|\mu\|_{W^{1,2}(\brd)}\big);\notag\\
		&\sup_{t\le s\le T}\|g^{t,\mu'}(s,\cdot,\cdot)-g^{t,\mu}(s,\cdot,\cdot)\|_{L^{2}(\brd\times\brd)}\label{lem11_1}\\
		&\qquad\le C(\gamma,L,T,\|\mu\|_{W^{1,2}(\brd)},\|\mu'\|_{W^{1,2}(\brd)})\|\mu'-\mu\|_{W^{1,2}(\brd)}.\notag
	\end{align}
	Here, for $s\in[t,T]$,
	\begin{equation*}
		\|g^{t,\mu}(s,\cdot,\cdot)\|_{W^{1,2}(\brd\times\brd)}:=\|g^{t,\mu}(s,\cdot,\cdot)\|_{L^2(\brd\times\brd)}+\sum_{i=1}^d\|g^{t,\mu}_{x_i}(s,\cdot,\cdot)\|_{L^2(\brd\times\brd)}.
	\end{equation*}
\end{lemma}

As a consequence of Assumption (H3) and Lemma~\ref{lem9}, we have the following proposition that shows the boundedness and continuous dependence of $k^{t,\mu}(s,\cdot,\cdot)$ on the initial $\mu$ and the time $s$.
\begin{proposition}\label{lem12}
	Let Assumptions (H1)-(H3) be satisfied and $\mu\in W^{1,2}(\brd)$. Then, there is a weak solution $k^{t,\mu}$ of equation \eqref{k}. For any $\varphi\in {L^{2}(\brd)}$, the function $\phi^{t,\mu}$ defined as
	\begin{align*}
		&\phi^{t,\mu}(s,y):=\int_{\brd} k^{t,\mu}(s,x,y)\varphi(x)dx,\quad (s,y)\in[t,T]\times\brd,
	\end{align*}
	satisfies the following estimate
	\begin{align}\label{lem12_1}
		&\sup_{s\in[t,T]}\|\phi^{t,\mu}(s,\cdot)\|_{L^{2}(\brd)}\le C\big(\gamma,L,T,\|\mu\|_{W^{1,2}(\brd)}\big)\|\vp\|_{L^2(\brd)}.
	\end{align}
	If moreover $\varphi\in { W^{1,2}(\brd)}$, then for $t\le s<s'\le T$ and $\mu,\mu'\in W^{1,2}(\brd)$, we have the following estimates
	\begin{align}
		&\sup_{t\le s\le T}\|\phi^{t,\mu'}(s,\cdot)-\phi^{t,\mu}(s,\cdot)\|_{L^{2}(\brd)}\label{lem12_2}\\
		&\qquad\le C(\gamma,L,T,\|\mu\|_{W^{1,2}(\brd)},\|\mu'\|_{W^{1,2}(\brd)})\|\varphi\|_{ W^{1,2}(\brd)}\|\mu'-\mu\|_{W^{1,2}(\brd)};\notag\\
		&\|\phi^{t,\mu}(s',\cdot)-\phi^{t,\mu}(s,\cdot)\|_{L^{2}(\brd)}\label{lem12_3}\\
		&\qquad\le C(\gamma,L,T,\|\mu\|_{W^{1,2}(\brd)},\|\varphi\|_{ W^{1,2}(\brd)})|s'-s|^{\frac{1}{2}}.\notag
	\end{align}
\end{proposition}

\begin{proof}
	The existence of $k^{t,\mu}$ and estimates \eqref{lem12_1} and \eqref{lem12_2} are direct consequences of Assumption (H3) and Lemma~\ref{lem9}. Now we prove \eqref{lem12_3}. We set
	\begin{align*}
		&u^{t,\mu}(s,y):=\int_{\brd} f^{t,\mu}(s,x,y)\varphi(x)dx,\quad (s,y)\in[t,T]\times\brd.
	\end{align*}
	Then, we have
	\begin{align*}
		\|\phi^{t,\mu}(s',\cdot)-\phi^{t,\mu}(s,\cdot)\|^2_{L^{2}(\brd)}\le &2\int_{\brd} |u^{t,\mu}(s',y)-u^{t,\mu}(s,y)|^2 dy \\
		&+2\int_{\brd}\Big|\int_{\brd} \varphi(x)[g^{t,\mu}(s',x,y)-g^{t,\mu}(s,x,y)]dx\Big|^2dy.
	\end{align*}
	From Cauchy's inequality and Assumption (H3), we have
	\begin{align*}
		&\int_{\brd} |u^{t,\mu}(s',y)-u^{t,\mu}(s,y)|^2 dy=\int_{\brd}\Big|\int_s^{s'}\frac{\dd u^{t,\mu}}{\dd r}(r,y)dr\Big|^2dy\\
		&\qquad\le |s'-s|\|u_s^{t,\mu}\|^2_{L^{2}([t,T];L^2(\brd)} \le |s'-s|C(\gamma,L,T)\|\varphi\|_{W^{1,2}(\brd)}^2.
	\end{align*}
	From Cauchy's inequality and Lemma~\ref{lem9}, we have
	\begin{align*}
		&\int_{\brd}\Big|\int_{\brd} \varphi(x)[g^{t,\mu}(s',x,y)-g^{t,\mu}(s,x,y)]dx\Big|^2dy\\
		&\qquad\le\|\varphi\|^2_{L^2(\brd)}|s'-s|\|g_r\|^2_{L^2([s,s'];L^2(\brd\times\brd))}\\
		&\qquad\le \|\varphi\|^2_{L^2(\brd)} |s'-s|C(\gamma,L,T,\|\mu\|_{W^{1,2}(\brd)}).
	\end{align*}
	The proof is complete.
\end{proof}

\begin{remark}
	Quite related interesting results can be found in \cite{TSE} on regularity of solutions of FP equations defined on the space of probability measures on a torus. However, they require that the diffusion coefficient is a positive constant, and our PF equation \eqref{pde_m} is not covered.
\end{remark}

\section{Derivatives of $X^{t,x,\mu}$}\label{SDE}
In this section, we consider the derivatives of $X^{t,x,\mu}$ with respect to $x$ and $\mu$. We always suppose that Assumptions (H1)-(H3) hold true. Let $m^{t,\mu}$ be the solution of the FP equation \eqref{pde_m}. From Assumption (H2) and standard arguments of SDEs, we know that for any $p\geq 2$, SDE \eqref{sde} has a unique solution such that
\begin{equation}\label{X}
	\e[\sup_{t\le s\le T}|X_s^{t,x,\mu}|^p]\le C(p,L,T,|x|).
\end{equation}
Moreover, from Proposition~\ref{lemma6}, for all $x,x'\in\brd$ and $\mu,\mu'\in W^{1,2}(\brd)$,
\begin{equation}\label{deltaX}
	\begin{split}
		&\e[\sup_{t\le s\le T}|X_s^{t,x',\mu'}-X_s^{t,x,\mu}|^p] \\
		&\qquad\le C(p,\gamma,L,T,\|\mu\|_{W^{1,2}(\brd)})(|x'-x|^p+\|\mu'-\mu\|^p_{L^2(\brd)}).
	\end{split}
\end{equation}

We begin with the derivatives of $X^{t,x,\mu}$ with respect to $x$. It is standard to show that the derivative $\dd _xX^{t,x,\mu}=(\dd_{x_i} X_j^{t,x,\mu})_{1\le i,j\le d}\in\mathcal{S}_{\f}^2([t,T];\br^{d\times d})$ is the unique solution of the following SDE
\begin{equation*}
	\begin{split}
		\dd _{x}X^{t,x,\mu}_{s}=I+\int_t^s b_x(X^{t,x,\mu}_r,m^{t,\mu}_r)\dd _{x}X^{t,x,\mu}_{r}dr+\int_t^s \sigma_x(X^{t,x,\mu}_r,m^{t,\mu}_r)\dd _{x}X^{t,x,\mu}_{r} dB_r,
	\end{split}
\end{equation*}
with $s\in[t,T]$, and the derivative $\dd _{x}^2X^{t,x,\mu}$ is the unique solution in the class $\mathcal{S}_{\f}^2([t,T];\br^{d\times d\times d})$ of the following SDE
\begin{equation*}
	\begin{split}
		\dd _x^2X^{t,x,\mu}_s=&\int_t^s b_x(X^{t,x,\mu}_r,m^{t,\mu}_r)\dd^2 _xX^{t,x,\mu}_r+(\dd _xX^{t,x,\mu}_r)^*b_{xx}(X^{t,x,\mu}_r,m^{t,\mu}_r)(\dd _xX^{t,x,\mu}_r)dr\\
		&+\int_t^s \sigma_x(X^{t,x,\mu}_r,m^{t,\mu}_r)\dd^2 _xX^{t,x,\mu}_r+(\dd _xX^{t,x,\mu}_r)^*\sigma_{xx}(X^{t,x,\mu}_r,m^{t,\mu}_r)(\dd _xX^{t,x,\mu}_r) dB_r,
	\end{split}
\end{equation*}
with $s\in[t,T]$. With standard arguments for SDEs, the following proposition is an immediate consequence of Assumption (H2), estimate \eqref{deltaX} and Proposition~\ref{lemma6}.
\begin{proposition}\label{lem2.0}
	Let Assumptions (H1) and (H2) be satisfied and $p\geq 2$. Then, for any $x,x'\in\brd$ and $\mu,\mu'\in W^{1,2}(\brd)$, we have
	\begin{align}
		&\e[\sup_{t\le s\le T}|(\dd_xX_s^{t,x,\mu},\dd_x^2X_s^{t,x,\mu})|^p]\le C(p,L,T);\label{lem2.0_1}\\
		&\e[\sup_{t\le s\le T}|\dd_xX_s^{t,x',\mu'}-\dd_xX_s^{t,x,\mu}|^p+\sup_{t\le s\le T}|\dd_x^2X_s^{t,x',\mu'}-\dd_x^2X_s^{t,x,\mu}|^p]\label{lem2.0_2}\\
		&\qquad \le C(p,\gamma,L,T,\|\mu\|_{W^{1,2}(\brd)})(|x'-x|^p+\|\mu'-\mu\|^p_{L^2(\brd)}).\notag
	\end{align}
\end{proposition}

In the rest of this section, we consider the derivative of $X^{t,x,\mu}$ with respect to $\mu$. For $h\in(0,1]$ and along the direction $\tilde{\mu}\in L^2(\brd)$, we define
\begin{align*}
	Y_s^{t,x,\mu}(\tilde{\mu},h):=\frac{1}{h}(X_s^{t,x,\mu+h\tilde{\mu}}-X_s^{t,x,\mu}),\quad s\in[t,T].
\end{align*}
Then, $Y^{t,x,\mu}(\tilde{\mu},h)$ satisfies the following SDE:
\begin{equation}\label{sde_Y}
	\begin{split}
		&Y^{t,x,\mu}_s(\tilde{\mu},h)=\\
		&\int_t^s\int_0^1\Big[b_x(X^{t,x,\mu}_r+\lambda hY^{t,x,\mu}_r(\tilde{\mu},h),m^{t,\mu+h\tilde{\mu}}(r,\cdot))\cdot Y^{t,x,\mu}_r(\tilde{\mu},h)\\
		&\ +\int_{\brd} \frac{\dd b}{\dd m}(X^{t,x,\mu}_r,m^{t,\mu}(r,\cdot)+\lambda h\tm_h^{t,\mu}(\tilde{\mu})(r,\cdot))(\xi)\tm_h^{t,\mu}(\tilde{\mu})(r,\xi) d\xi \Big] d\lambda dr\\
		&\ +\int_t^s\int_0^1\Big[\sigma_x(X^{t,x,\mu}_r+\lambda hY^{t,x,\mu}_r(\tilde{\mu},h),m^{t,\mu+h\tilde{\mu}}(r,\cdot))\cdot Y^{t,x,\mu}_r(\tilde{\mu},h)\\
		&\ +\int_{\brd} \frac{\dd\sigma}{\dd m}(X^{t,x,\mu}_r,m^{t,\mu}(r,\cdot)+\lambda h\tm_h^{t,\mu}(\tilde{\mu})(r,\cdot))(\xi)\tm_h^{t,\mu}(\tilde{\mu})(r,\xi) d\xi \Big] d\lambda dB_r,
	\end{split}
\end{equation}
with $s\in[t,T]$, where $\tilde{m}_h^{t,\mu}(\tilde{\mu})$ is defined as in \eqref{m_h:def}. From Assumption (H2) and estimate \eqref{lem3_1}, using standard arguments of SDEs, we have for $p\geq 2$,
\begin{equation}\label{estimate_Y}
	\begin{split}
		\e[\sup_{t\le s\le T}|Y^{t,x,\mu}_s(\tilde{\mu},h)|^p]&\le C(p,\gamma,L,T,\|\mu\|_{W^{1,2}(\brd)})\|\tilde{\mu}\|^p_{L^2(\brd)}.
	\end{split}
\end{equation}
Let $\tm^{t,\mu}(\tilde{\mu})$ be the unique solution of PDE \eqref{tm} (see Lemma~\ref{thm4}) and $Y^{t,x,\mu}(\tilde{\mu})$ be the unique solution in the space $\mathcal{S}_{\f}^2[t,T]$ of the following SDE:
\begin{equation}\label{sde_Y'}
	\begin{aligned}
		&Y^{t,x,\mu}_s(\tilde{\mu})=\int_t^s\Big[b_x(X^{t,x,\mu}_r,m^{t,\mu}(r,\cdot)) Y^{t,x,\mu}_r(\tilde{\mu})\\
		&\qquad\qquad\qquad +\int_{\brd} \frac{\dd b}{\dd m}(X^{t,x,\mu}_r,m^{t,\mu}(r,\cdot))(\xi)\tm^{t,\mu}(\tilde{\mu})(r,\xi) d\xi \Big]dr\\
		&\quad\qquad\qquad+\int_t^s\Big[\sigma_x(X^{t,x,\mu}_r,m^{t,\mu}(r,\cdot)) Y^{t,x,\mu}_r(\tilde{\mu})\\
		&\qquad\qquad\qquad +\int_{\brd} \frac{\dd\sigma}{\dd m}(X^{t,x,\mu}_r,m^{t,\mu}(r,\cdot))(\xi)\tm^{t,\mu}(\tilde{\mu})(r,\xi) d\xi \Big]dB_r,\  s\in[t,T].
	\end{aligned}
\end{equation}
From Assumption (H2) and Lemma~\ref{thm4},  using standard arguments of SDEs, we have for $p\geq 2$,
\begin{align}\label{estimate_Y'}
	\e[\sup_{t\le s\le T}|Y^{t,x,\mu}_s(\tilde{\mu})|^p]\le C(p,\gamma,L,T,\|\mu\|_{W^{1,2}(\brd)})\|\tilde{\mu}\|^p_{L^2(\brd)},
\end{align}
and (further in view of Proposition~\ref{lemma6} and estimates \eqref{deltaX} and \eqref{estimate_Y'},)
\begin{equation}\label{estimate_deY'}
	\begin{split}
		&\e[\sup_{t\le s\le T}|Y^{t,x',\mu'}_s(\tilde{\mu})-Y^{t,x,\mu}_s(\tilde{\mu})|^2]\\
		&\qquad\le C(\gamma,L,T,\|\mu\|_{W^{1,2}(\brd)},\|\mu'\|_{W^{1,2}(\brd)})\|\tilde{\mu}\|^2_{L^2(\brd)}\\
		&\quad\qquad\times(|x'-x|^2+\|\mu'-\mu\|^2_{L^2(\brd)}).
	\end{split}
\end{equation}

We have the following convergence of $Y^{t,x,\mu}(\tilde{\mu},h)$ to $Y^{t,x,\mu}(\tilde{\mu})$ in $\mathcal{S}^2_{\f}[t,T]$ as $h$ goes to $0$, whose proof is given in Appendix~\ref{pf_lem2.3}.

\begin{lemma}\label{lem2.3}
	Let Assumptions (H1) and (H2) be satisfied, $\mu\in W^{1,2}(\brd)$ and $\tilde{\mu}\in L^2(\brd)$. Then,
	\begin{equation*}
		\begin{split}
			\e[\sup_{t\le s\le T}|\frac{1}{h}(X_s^{t,x,\mu+h\tilde{\mu}}-X_s^{t,x,\mu})-Y^{t,x,\mu}_s(\tilde{\mu})|^2]\le C(\gamma,L,T,\|\mu\|_{W^{1,2}(\brd)},\|\tilde{\mu}\|_{L^2(\brd)})h^2.
		\end{split}
	\end{equation*}
	That is, $Y_s^{t,x,\mu}(\tilde{\mu})$ is the derivative of $X^{t,x,\mu}_s$ along the direction $\tilde{\mu}\in L^2(\brd)$.
\end{lemma}

Lemma~\ref{lem2.3} and estimates \eqref{estimate_Y'} and \eqref{estimate_deY'} show that the mapping
\begin{align*}
	L^2(\brd)\ni\mu\mapsto X^{t,x,\mu}_s\in L^2(\f_s;\br^d)
\end{align*}
is Fr\'echet differentiable for $\mu\in W^{1,2}(\brd)$. Its Fr\'echet derivative $D_\mu X^{t,x,\mu}_s$ is
\begin{align*}
	D_\mu X^{t,x,\mu}_s(\tilde{\mu})=Y_s^{t,x,\mu}(\tilde{\mu}),\qquad\tilde{\mu}\in L^2(\brd).
\end{align*}
To represent the Fr\'echet derivative $D_\mu X^{t,x,\mu}_s$, we define the process $U^{t,x,\mu}(y)$ as
\begin{equation}\label{sde_U}
	\begin{aligned}
		&U^{t,x,\mu}_s(y)=\int_t^s\Big[b_x(X^{t,x,\mu}_r,m^{t,\mu}(r,\cdot)) U^{t,x,\mu}_r(y)\\
		&\qquad\qquad\qquad +\int_{\brd} \frac{\dd b}{\dd m}(X^{t,x,\mu}_r,m^{t,\mu}(r,\cdot))(\xi)k^{t,\mu}(r,\xi,y) d\xi \Big]dr\\
		&\quad\qquad\qquad+\int_t^s\Big[\sigma_x(X^{t,x,\mu}_r,m^{t,\mu}(r,\cdot))U^{t,x,\mu}_r(y)\\
		&\qquad\qquad\qquad +\int_{\brd} \frac{\dd\sigma}{\dd m}(X^{t,x,\mu}_r,m^{t,\mu}(r,\cdot))(\xi)k^{t,\mu}(r,\xi,y) d\xi \Big]dB_r,\  s\in[t,T].
	\end{aligned}
\end{equation}
The following proposition is a direct consequence of Proposition~\ref{lem7} and the uniqueness of solutions of SDE \eqref{sde_Y'}.
\begin{proposition}\label{lem2.4}
	Let Assumptions (H1) and (H2) be satisfied and $\mu\in W^{1,2}(\brd)$. Let $U^{t,x,\mu}(\cdot)\in L^2(\brd;\mathcal{S}_\f^2[t,T])$ be the solution of equation \eqref{sde_U}. Then, the derivative $Y_s^{t,x,\mu}(\tilde{\mu})$ of $X_s^{t,x,\mu}$ along the direction $\tilde{\mu}\in L^2(\brd)$ satisfies
	\begin{align*}
		Y_s^{t,x,\mu}(\tilde{\mu})=\int_{\brd} U_s^{t,x,\mu}(y)\tilde{\mu}(y)dy,\quad s\in [t,T].
	\end{align*}
	That is, the Fr\'echet derivative $D_\mu X^{t,x,\mu}_s$ of $X^{t,x,\mu}_s$ satisfies
	\begin{align*}
		D_\mu X^{t,x,\mu}_s(\tilde{\mu})=\int_{\brd} U^{t,x,\mu}_s(y)\tilde{\mu}(y)dy,\qquad\tilde{\mu}\in L^2(\brd).
	\end{align*}
\end{proposition}

The following proposition shows the boundedness and the continuous dependence of $U^{t,x,\mu}_s(\cdot)$ with respect to $(x, \mu)$, which is proved in Appendix~\ref{pf_lem2.5}.

\begin{proposition}\label{lem2.5}
	Let Assumptions (H1)-(H3) be satisfied and $p\geq 2$. Then, for all $x,x'\in\brd$ and $\mu,\mu'\in W^{1,2}(\brd)$, we have the following estimates
	\begin{align}
		&\e\Big[\int_{\brd}\sup_{t\le s\le T}|U^{t,x,\mu}_s(y)|^2dy\Big]\le C(\gamma,L,T,\|\mu\|_{W^{1,2}(\brd)});\label{lem2.5_1}\\
		&\e\Big[\sup_{t\le s\le T}\Big|\int_{\brd}|U^{t,x,\mu}_s(y)|^2dy\Big|^p\Big]\le C(p,\gamma,L,T,\|\mu\|_{W^{1,2}(\brd)});\label{lem2.5_1'}\\
		&\e\Big[\int_{\brd}\sup_{t\le s\le T}|U^{t,x',\mu'}_s(y)-U^{t,x,\mu}_s(y)|^2dy\Big]\label{lem2.5_2}\\
		&\qquad \le C(\gamma,L,T,\|\mu\|_{W^{1,2}(\brd)},\|\mu'\|_{W^{1,2}(\brd)})(|x'-x|^2+\|\mu'-\mu\|^2_{W^{1,2}(\brd)}).\notag
	\end{align}
\end{proposition}

\section{Regularity of the value function}\label{regularity}
In this section, we study the regularity of the value function
\begin{equation}\label{v}
	V(t,x,\mu):=\e[\Phi(X_T^{t,x,\mu},m^{t,\mu}(T,\cdot))], \quad (t,x,\mu)\in[0,T]\times\brd\times W^{1,2}(\brd),
\end{equation}
where $m^{t,\mu}$ is the solution of the FP equation \eqref{pde_m} and $X^{t,x,\mu}$ is the solution of SDE \eqref{sde}. We need the following assumption on function $\Phi$.

\textbf{(H4)} The function $\Phi:\brd\times L^2(\brd)\to\br$ is twice differentiable in $x$ and differentiable in $m$, with the derivative $\frac{\dd \Phi}{\dd m}(x,m)(\cdot)\in W^{1,2}(\brd)$ for all $(x,m)\in \brd\times L^2(\brd)$. The function and the derivatives are uniformly bounded by $L$. That is, for all $(t,x,\mu)\in[0,T]\times\brd\times L^2(\brd)$ and $1\le i,j\le d$,
\begin{align*}
	|(\Phi,\frac{\dd \Phi}{\dd x_i},\frac{\dd^2 \Phi}{\dd x_i\dd x_j})(x,m)|+\|\frac{\dd \Phi}{\dd m}(x,m)(\cdot)\|_{W^{1,2}(\brd)}\le L.
\end{align*}
Moreover, the function and the derivatives are $L$-Lipschitz continuous for $1\le i,j\le d$. That is, for all $(x,m),(x',m')\in\brd\times L^2(\brd)$ and $1\le i,j\le d$,
\begin{equation*}
	\begin{split}
		&|(\Phi,\frac{\dd \Phi}{\dd x_i},\frac{\dd^2 \Phi}{\dd x_i\dd x_j})(x',m')-(\Phi,\frac{\dd \Phi}{\dd x_i},\frac{\dd^2 \Phi}{\dd x_i\dd x_j})(x,m)|\\
		&\qquad+\|\frac{\dd \Phi}{\dd m}(x',m')(\cdot)-\frac{\dd \Phi}{\dd m}(x,m)(\cdot)\|_{L^2(\brd)}\le L(|x'-x|+\|m'-m\|_{L^2(\brd)}).
	\end{split}
\end{equation*}

The following proposition is proved in Appendix~\ref{pf_thm3.1}.
\begin{proposition}\label{thm3.1}
	Let Assumptions (H1)-(H4) be satisfied. Then, for all $t\in[0,T]$, the function $V(t,\cdot,\cdot)$ is twice differentiable in $x$ and differentiable in $\mu$ for $(x,\mu)\in \brd\times W^{1,2}(\brd)$, and the derivatives are of the form
	\begin{align}
		&\frac{\dd V}{\dd x_i}(t,x,\mu)=\sum_{k=1}^{d}\e\Big[\frac{\dd \Phi}{\dd x_k}(X_T^{t,x,\mu},m^{t,\mu}(T,\cdot))\dd_{x_i}X_{T,k}^{t,x,\mu}\Big],\quad1\le i\le d,\label{V_x}\\
		&\frac{\dd^2 V}{\dd x_i\dd x_j}(t,x,\mu)=\sum_{k,l=1}^{d}\e\Big[\frac{\dd^2 \Phi}{\dd x_k\dd x_l}(X_T^{t,x,\mu},m^{t,\mu}(T,\cdot))\dd_{x_i}X_{T,k}^{t,x,\mu}\dd_{x_j}X_{T,l}^{t,x,\mu}\label{V_xx}\\
		&\quad\qquad\qquad\qquad\qquad+\frac{\dd \Phi}{\dd x_k}(X_T^{t,x,\mu},m^{t,\mu}(T,\cdot))\dd^2_{x_ix_j}X_{T,k}^{t,x,\mu}\Big],\quad1\le i,j\le d,\notag\\
		&\frac{\dd V}{\dd \mu}(t,x,\mu)(y)=\sum_{k=1}^d\e\Big[\frac{\dd \Phi}{\dd x_k}(X_T^{t,x,\mu},m^{t,\mu}(T,\cdot))U_{T,k}^{t,x,\mu}(y)\Big]\label{V_mu}\\
		&\quad\qquad\qquad\qquad+\e\Big[\int_{\brd}\frac{\dd \Phi}{\dd m}(X_T^{t,x,\mu},m^{t,\mu}(T,\cdot))(\xi)k^{t,\mu}(T,\xi,y)d\xi\Big].\notag
	\end{align}
	Moreover, for all $(t,x,\mu),(t',x',\mu')\in[0,T]\times\brd\times W^{1,2}(\brd)$ and $1\le i,j\le d$, we have the following estimates
	\begin{align*}
		&|(V,\frac{\dd V}{\dd x_i},\frac{\dd^2 V}{\dd x_i\dd x_j})(t,x,\mu)|+\|\frac{\dd V}{\dd \mu}(t,x,\mu)(\cdot)\|_{L^2(\brd)}\le C(\gamma,L,T,\|\mu\|_{W^{1,2}(\brd)});\\
		&|(V,\frac{\dd V}{\dd x_i},\frac{\dd^2 V}{\dd x_i\dd x_j})(t',x',\mu')-(V,\frac{\dd V}{\dd x_i},\frac{\dd^2 V}{\dd x_i\dd x_j})(t,x,\mu)|\\
		&+\|\frac{\dd V}{\dd \mu}(t',x',\mu')(\cdot)-\frac{\dd V}{\dd \mu}(t,x,\mu)(\cdot)\|_{L^2(\brd)}\\
		&\qquad\le C(\gamma,L,T,\|\mu\|_{W^{1,2}(\brd)},\|\mu'\|_{W^{1,2}(\brd)})(|t'-t|^{\frac{1}{2}}+|x'-x|+\|\mu'-\mu\|_{ W^{1,2}(\brd)}).
	\end{align*}
\end{proposition}

\section{PDE associated with a Mean-Field SDE}\label{PDE_V}
We first give the definition of the solution space.
\begin{definition}\label{def:f}
	We say that $f\in C^{1,2,1}_{b}$, if $f:[0,T]\times\brd\times L^2(\brd)\to\br$ is differentiable in $(t,\mu)$ and twice differentiable in $x$ for $(t,x,\mu)\in[0,T]\times\brd\times W^{{ 3},2}(\brd)$, and for all $(t,x,\mu),(t',x',\mu')\in[0,T]\times\brd\times W^{3,2}(\brd)$ and $1\le i,j\le d$,
	\begin{align*}
		&|(f,\frac{\dd f}{\dd t},\frac{\dd f}{\dd x_i},\frac{\dd^2 f}{\dd x_i\dd x_j})(t,x,\mu)|+\|\frac{\dd f}{\dd \mu}(t,x,\mu)(\cdot)\|_{L^2(\brd)}\le C(\|\mu\|_{W^{3,2}(\brd)});\\
		&|(f,\frac{\dd f}{\dd t},\frac{\dd f}{\dd x_i},\frac{\dd^2 f}{\dd x_i\dd x_j})(t',x',\mu')-(f,\frac{\dd f}{\dd t},\frac{\dd f}{\dd x_i},\frac{\dd^2 f}{\dd x_i\dd x_j})(t,x,\mu)|\\
		&+\|\frac{\dd f}{\dd \mu}(t',x',\mu')(\cdot)-\frac{\dd f}{\dd \mu}(t,x,\mu)(\cdot)\|_{L^2(\brd)}\\
		&\qquad\le C(\|\mu\|_{ W^{3,2}(\brd)},\|\mu'\|_{ W^{3,2}(\brd)})(|t'-t|^{\frac{1}{2}}+|x'-x|+\|\mu'-\mu\|_{W^{2,2}(\brd)}).
	\end{align*}
\end{definition}

We first establish the It\^o's formula for $f\in  C^{1,2,1}_b$, which is proved in Appendix~\ref{pf_forall}.
\begin{lemma}\label{forall}
	Let Assumptions (H1) and (H2') be satisfied. Then, for $f\in C^{1,2,1}_b$, $0\le t\le s\le T$, $x\in\brd$ and $\mu\in W^{3,2}(\brd)$, we have
	\begin{equation}\label{forall_1}
		\begin{split}		
			&f(s,X_s^{t,x,\mu},m^{t,\mu}(s,\cdot))-f(t,x,\mu)\\
			&=\int_t^s \Big[\frac{\dd f}{\dd r}(r,X_r^{t,x,\mu},m^{t,\mu}(r,\cdot))\\
			&\qquad+\sum_{i=1}^d \frac{\dd f}{\dd x_i}(r,X_r^{t,x,\mu},m^{t,\mu}(r,\cdot))b_i(X_r^{t,x,\mu},m^{t,\mu}(r,\cdot))\\
			&\qquad+ \sum_{i,j=1}^d\frac{\dd^2 f}{\dd x_i\dd x_j}(r,X_r^{t,x,\mu},m^{t,\mu}(r,\cdot))a_{ij}(X_r^{t,x,\mu},m^{t,\mu}(r,\cdot)\\
			&\qquad +\int_{\brd}\frac{\dd f}{\dd \mu}(r,X_r^{t,x,\mu},m^{t,\mu}(r,\cdot))(\xi)\frac{\dd m^{t,\mu}}{\dd r}(r,\xi)d\xi\Big]dr\\
			&\qquad + \sum_{i,j=1}^d\int_t^s\frac{\dd f}{\dd x_i}(r,X_r^{t,x,\mu},m^{t,\mu}(r,\cdot))\sigma_{ij}(X_r^{t,x,\mu},m^{t,\mu}(r,\cdot))dB_r^j.
		\end{split}
	\end{equation}
\end{lemma}

\begin{remark}\label{rk_1}
	From Assumption (H4) we know that $\Phi\in C_b^{1,2,1}$, so we can apply It\^o's formula in Lemma~\ref{forall} for $\Phi$ when $(x,\mu)\in\brd\times W^{3,2}(\brd)$. Actually, the conditions of $\Phi$ in  Assumption (H4) is stronger than the conditions of $f\in C_b^{1,2,1}$ in Definition~\ref{def:f}, so It\^o's formula holds true for $\Phi$ when Assumption (H2) is satisfied and $(x,\mu)\in\brd\times W^{1,2}(\brd)$. The proof is a simplified version of the proof of Lemma~\ref{forall}, so it is omitted here. It\^o's formula for $\Phi$ allows us to show that our value function $V(t,x,\mu)$ is continuously differentiable with respect to $t$.
\end{remark}

\begin{proposition}\label{4.2}
	Let Assumptions (H1), (H2) and (H4) be satisfied. Then, function $V(\cdot,x,\mu)$ is differentiable for $(x,\mu)\in\brd\times W^{1,2}(\brd)$, with the derivative of the form
	\begin{equation}\label{4.2_1}
		\begin{split}
			\frac{\dd V}{\dd t}(t,x,\mu)=-\e\Big[&\sum_{i=1}^db_i(X_T^{t,x,\mu},m^{t,\mu}(T,\cdot))\frac{\dd \Phi}{\dd x_i}(X_T^{t,x,\mu},m^{t,\mu}(T,\cdot))\\
			&+\sum_{i,j=1}^da_{ij}(X_T^{t,x,\mu},m^{t,\mu}(T,\cdot))\frac{\dd^2 \Phi}{\dd x_i\dd x_j}(X_T^{t,x,\mu},m^{t,\mu}(T,\cdot))\\
			&+\int_{\brd}\frac{\dd \Phi}{\dd m}(X_T^{t,x,\mu},m^{t,\mu}(T,\cdot))(\xi)\frac{\dd m^{t,\mu}}{\dd s}(T,\xi)d\xi\Big].
		\end{split}	
	\end{equation}
	If moreover Assumption (H2') holds true, then for all $(t,x,\mu),(t',x',\mu')\in[0,T]\times\brd\times W^{{3},2}(\brd)$, we have the following estimates
	\begin{align}
		&|\frac{\dd V}{\dd t}(t,x,\mu)|\le C(\gamma,L,T,\|\mu\|_{W^{2,2}(\brd)}),\label{4.2_1.5}\\
		&|\frac{\dd V}{\dd t}(t',x',\mu')-\frac{\dd V}{\dd t}(t,x,\mu)|\le C(\gamma,L,T,\|\mu\|_{W^{{3},2}(\brd)},\|\mu'\|_{W^{{3},2}(\brd)})\label{4.2_2}\\
		&\qquad\qquad\qquad\qquad\qquad\qquad\qquad \times(|t'-t|^{\frac{1}{2}}+|x'-x|+\|\mu'-\mu\|_{W^{2,2}(\brd)}).\notag
	\end{align}
\end{proposition}

\begin{proof}
	From Remark~\ref{rk_1} and the definition of $V$, we have
	\begin{align*}
		V(t,x,\mu)-V(T,x,\mu)&=\e\Big[\int_t^T\sum_{i=1}^d\frac{\dd \Phi}{\dd x_i}(X^{t,x,\mu}_s,m^{t,\mu}(s,\cdot))b_i(X^{t,x,\mu}_s,m^{t,\mu}(s,\cdot))\\
		&\qquad +\sum_{i,j=1}^{d}\frac{\dd^2 \Phi}{\dd x_i\dd x_j}(X^{t,x,\mu}_s,m^{t,\mu}(s,\cdot))a_{ij}(X^{t,x,\mu}_s,m^{t,\mu}(s,\cdot))\\
		&\qquad+\int_{\brd} \frac{\dd\Phi}{\dd m}(X^{t,x,\mu}_s,m^{t,\mu}(s,\cdot))(\xi)\frac{\dd m^{t,\mu}}{\dd s}(s,\xi) d\xi ds\Big],
	\end{align*}
    for $(t,x,\mu)\in [0,T]\times\brd\times W^{1,2}(\brd)$. Then it is evident that \eqref{4.2_1} holds true. The proof of \eqref{4.2_1.5} and \eqref{4.2_2} is the analogy to the proof of Proposition~\ref{thm3.1}, by using Lemma~\ref{thm1}, Propositions~\ref{lemma6} and \ref{thm1'} and estimates \eqref{deltaX} and \eqref{thm3.1_1}.
\end{proof}

From Assumptions (H1), (H2'), (H3) and (H4) and Propositions~\ref{thm3.1} and \ref{4.2}, we know that $V\in C_b^{1,2,1}$. From Lemma~\ref{forall}, we know that for $0\le t\le s\le T$ and $(x,\mu)\in\brd\times W^{3,2}(\brd)$,
\begin{equation}\label{remark_1}
	\begin{split}		
		&V(s,X_s^{t,x,\mu},m^{t,\mu}(s,\cdot))-V(t,x,\mu)\\
		&=\int_t^s \Big[\frac{\dd V}{\dd r}(r,X_r^{t,x,\mu},m^{t,\mu}(r,\cdot))\\
		&\quad+\sum_{i=1}^d \frac{\dd V}{\dd x_i}(r,X_r^{t,x,\mu},m^{t,\mu}(r,\cdot))b_i(X_r^{t,x,\mu},m^{t,\mu}(r,\cdot))\\
		&\quad+ \sum_{i,j=1}^d\frac{\dd^2 V}{\dd x_i\dd x_j}(r,X_r^{t,x,\mu},m^{t,\mu}(r,\cdot))a_{ij}(X_r^{t,x,\mu},m^{t,\mu}(r,\cdot)\\
		&\quad +\int_{\brd}\frac{\dd V}{\dd \mu}(r,X_r^{t,x,\mu},m^{t,\mu}(r,\cdot))(\xi)\frac{\dd m^{t,\mu}}{\dd r}(r,\xi)d\xi\Big]dr\\
		&\quad +\sum_{i,j=1}^d \int_t^s\frac{\dd V}{\dd x_i}(r,X_r^{t,x,\mu},m^{t,\mu}(r,\cdot))\sigma_{ij}(X_r^{t,x,\mu},m^{t,\mu}(r,\cdot))dB^j_r.
	\end{split}
\end{equation}
Our main result is stated as follows.

\begin{theorem}\label{main}
	Let Assumptions (H1), (H2'), (H3) and (H4) be satisfied. Then, the function $V(t,x,\mu)=\e[\Phi(X_T^{t,x,\mu},m^{t,\mu}(T,\cdot))],\ (t,x,\mu)\in[0,T]\times\brd\times W^{3,2}(\brd)$ is the unique solution in $C^{1,2,1}_b$ of the PDE
	\begin{equation}\label{main_1}
		\left\{
		\begin{aligned}
			&\frac{\dd V}{\dd t}(t,x,\mu)+\sum_{i=1}^d\frac{\dd V}{\dd x_i}(t,x,\mu)b_i(x,\mu)+\sum_{i,j=1}^d\frac{\dd^2V}{\dd x_i\dd x_j}(t,x,\mu)a_{ij}(x,\mu)\\
			&\qquad+\int_{\brd}\frac{\dd V}{\dd \mu}(t,x,\mu)(\xi)\Big\{\sum_{i,j=1}^d\frac{\dd^2}{\dd \xi_i\dd\xi_j}[a_{ij}(\xi,\mu(\cdot))\mu(\xi)]\\
			&\qquad-\sum_{i=1}^d\frac{\dd}{\dd\xi_i}[b_i(\xi,\mu(\cdot))\mu(\xi)]\Big\}d\xi=0,\\
			&\qquad(t,x,\mu)\in[0,T)\times\brd\times W^{3,2}(\brd);\\
			&V(T,x,\mu)=\Phi(x,\mu),\quad (x,\mu)\in\brd\times W^{3,2}(\brd).
		\end{aligned}
		\right.
	\end{equation}
\end{theorem}

\begin{proof}
	First note from Proposition~\ref{thm1'} that $m^{t,\mu}(s,\cdot)\in W^{3,2}(\brd)$ for $0\le t\le s\le T$ and $\mu\in W^{3,2}(\brd)$. From the flow property
	\begin{align*}
		(X^{s,X_s^{t,x,\mu},m^{t,\mu}(s,\cdot)}_r,m^{s,m^{t,\mu}(s,\cdot)}(r,\cdot))=(X_r^{t,x,\mu},m^{t,\mu}(r,\cdot)),\\
		\qquad0\le t\le s\le r\le T,\quad (x,\mu)\in\brd\times W^{3,2}(\brd),
	\end{align*}
	as well as
	\begin{align*}
		V(s,y,\mu')=\e[\Phi(X_T^{s,y,\mu'},m^{s,\mu'}(T,\cdot))]=\e[\Phi(X_T^{s,y,\mu'},m^{s,\mu'}(T,\cdot))|\f_s],\\
		\qquad (s,y,\mu')\in [0,T]\times\brd\times W^{3,2}(\brd),
	\end{align*}
	we deduce that
	\begin{align*}
		V(s,X_s^{t,x,\mu},m^{t,\mu}(s,\cdot))&=\e[\Phi(X_T^{s,y,\mu'},m^{s,\mu'}(T,\cdot))|\f_s]|_{(y,\mu')=(X_s^{t,x,\mu},m^{t,\mu}(s,\cdot))}\\
		&=\e[\Phi(X^{s,X_s^{t,x,\mu},m^{t,\mu}(s,\cdot)}_T,m^{s,m^{t,\mu}(s,\cdot)}(T,\cdot))|\f_s]\\
		&=\e[\Phi(X_T^{t,x,\mu},m^{t,\mu}(T,\cdot))|\f_s],\quad s\in[t,T].
	\end{align*}
	That is, $V(s,X_s^{t,x,\mu},m^{t,\mu}(s,\cdot)),\ s\in[t,T]$, is a martingale. On the other hand, due to \eqref{remark_1}, we have
	\begin{equation*}
		\begin{split}		
			&V(s,X_s^{t,x,\mu},m^{t,\mu}(s,\cdot))-V(t,x,\mu)\\
			&=\sum_{i,j=1}^d\int_t^s\frac{\dd V}{\dd x_i}(r,X_r^{t,x,\mu},m^{t,\mu}(r,\cdot))\sigma_{ij}(X_r^{t,x,\mu},m^{t,\mu}(r,\cdot))dB_r^j,\quad s\in[t,T]
		\end{split}
	\end{equation*}
	and
	\begin{equation*}
		\begin{split}		
			&0=\int_t^s \Big[\frac{\dd V}{\dd r}(r,X_r^{t,x,\mu},m^{t,\mu}(r,\cdot))+\sum_{i=1}^d \frac{\dd V}{\dd x_i}(r,X_r^{t,x,\mu},m^{t,\mu}(r,\cdot))b_i(X_r^{t,x,\mu},m^{t,\mu}(r,\cdot))\\
			&\qquad+ \sum_{i,j=1}^d\frac{\dd^2 V}{\dd x_i\dd x_j}(r,X_r^{t,x,\mu},m^{t,\mu}(r,\cdot))a_{ij}(X_r^{t,x,\mu},m^{t,\mu}(r,\cdot))\\
			&\qquad +\int_{\brd}\frac{\dd V}{\dd \mu}(r,X_r^{t,x,\mu},m^{t,\mu}(r,\cdot))(\xi)\frac{\dd m^{t,\mu}}{\dd r}(r,\xi)d\xi\Big]dr,\quad s\in[t,T],
		\end{split}
	\end{equation*}
	from which we obtain the desired PDE.
	
	It only remains to prove the uniqueness of the solution of the PDE in the class $C_b^{1,2,1}$. Let $V_*\in C^{1,2,1}_b$ be a solution of PDE \eqref{main_1}. Then, from the It\^o's formula in Lemma~\ref{forall}, we know that the process
	\begin{equation*}
		\begin{split}		
			&V_*(s,X_s^{t,x,\mu},m^{t,\mu}(s,\cdot))-V_*(t,x,\mu)\\
			&=\sum_{i,j=1}^d\int_t^s\frac{\dd V_*}{\dd x_i}(r,X_r^{t,x,\mu},m^{t,\mu}(r,\cdot))\sigma_{ij}(X_r^{t,x,\mu},m^{t,\mu}(r,\cdot))dB^j_r,\quad s\in[t,T],
		\end{split}
	\end{equation*}
	is a martingale. Thus, for all $(t,x,\mu)\in[0,T]\times\brd\times W^{3,2}(\brd)$,
	\begin{align*}
		V_*(t,x,\mu)&=\e[V_*(T,X_T^{t,x,\mu},m^{t,\mu}(T,\cdot))|\f_t]\\
		&=\e[\Phi(X_T^{t,x,\mu},m^{t,\mu}(T,\cdot))]=V(t,x,\mu).
	\end{align*}
	This proves that $V_*$ and $V$ coincide, that is, the solution is unique in $C_b^{1,2,1}$. The proof is complete.
\end{proof}

\begin{corollary}\label{cor}
	Let Assumptions (H1), (H2'), (H3) and (H4) be satisfied. Then, the function $V$ satisfies the first equality of  PDE \eqref{pde_v} at any triplet $(t,x,\mu)\in[0,T]\times\brd\times W^{2,2}(\brd)$.
\end{corollary}

\begin{proof}
	For $\mu\in W^{2,2}(\brd)$, we can find $\{\mu_n,\ n\geq 1\}\subset W^{3,2}(\brd)$ such that $\lim_{n\to +\infty}\mu_n=\mu$ in $W^{2,2}(\brd)$. From Proposition~\ref{thm3.1} and Assumption (H2), we know that
	\begin{align*}
		&\lim_{n\to+\infty}\Big\{\sum_{i=1}^d\frac{\dd V}{\dd x_i}(t,x,\mu_n)b_i(x,\mu_n)+\sum_{i,j=1}^d\frac{\dd^2V}{\dd x_i\dd x_j}(t,x,\mu_n)a_{ij}(x,\mu_n)\\
		&\qquad+\int_{\brd}\frac{\dd V}{\dd \mu}(t,x,\mu_n)(\xi)\Big[\sum_{i,j=1}^d\frac{\dd^2}{\dd \xi_i\dd\xi_j}[a_{ij}(\xi,\mu_n(\cdot))\mu_n(\xi)]\\
		&\qquad-\sum_{i=1}^d\frac{\dd}{\dd\xi_i}[b_i(\xi,\mu_n(\cdot))\mu_n(\xi)]\Big]d\xi\Big\}\\
		&=\sum_{i=1}^d\frac{\dd V}{\dd x_i}(t,x,\mu)b_i(x,\mu)+\sum_{i,j=1}^d\frac{\dd^2V}{\dd x_i\dd x_j}(t,x,\mu)a_{ij}(x,\mu)\\
		&\qquad+\int_{\brd}\frac{\dd V}{\dd \mu}(t,x,\mu)(\xi)\Big\{\sum_{i,j=1}^d\frac{\dd^2}{\dd \xi_i\dd\xi_j}[a_{ij}(\xi,\mu(\cdot))\mu(\xi)]-\sum_{i=1}^d\frac{\dd}{\dd\xi_i}[b_i(\xi,\mu(\cdot))\mu(\xi)]\Big\}d\xi,
	\end{align*}
	uniformly in $(t,x)\in [0,T)\times\brd$. From PDE \eqref{main_1} and the fact that $\{\mu_n,\ n\geq1\}\subset W^{3,2}(\brd)$, we know that the sequence $\frac{\dd V}{\dd t}(t,x,\mu_n)$ converge as $n\to+\infty$ uniformly in $(t,x)\in [0,T)\times\brd$. Therefore, we have for $(t,x)\in[0,T)\times\brd$,
	\begin{align*}
		&\frac{\dd V}{\dd t}(t,x,\mu)-\lim_{n\to+\infty}\frac{\dd V}{\dd t}(t,x,\mu_n)\\
		&=\lim_{h\to0}\frac{1}{h}[V(t+h,x,\mu)-V(t,x,\mu)]-\lim_{n\to+\infty}\lim_{h\to0}\frac{1}{h}[V(t+h,x,\mu_n)-V(t,x,\mu_n)]\\
		&=\lim_{h\to0}\lim_{n\to+\infty}\frac{1}{h}[({V(t+h,x,\mu)-V(t+h,x,\mu_n)})-(V(t,x,\mu)-V(t,x,\mu_n))]=0.
	\end{align*}
	That is, $V$ is a solution of PDE \eqref{pde_v} defined on $(t,x,\mu)\in[0,T]\times\brd\times W^{2,2}(\brd)$. The uniqueness result is a direct consequence of the continuity of $V$ and the uniqueness result in Theorem~\ref{main}.
\end{proof}

\section{The case of state-invariant diffusion}\label{sec:exp}
Consider the particular case that the coefficients $b$ and $\sigma$ do not depend  on the state $x$.
\begin{lemma}\label{exp1}
	Let Assumptions (H1) and (H2) be satisfied. If coefficients $(b,\sigma)$ are independent of $x$, then (H3) holds true.
\end{lemma}

\begin{proof}
	For $(t,\mu)\in[0,T]\times W^{1,2}(\brd)$, by standard arguments of the fundamental solution of parabolic equations, there is a weak solution $h^{t,\mu}$ of the following equation
	\begin{equation}\label{h}
		\left\{
		\begin{aligned}
			&\frac{\dd h^{t,\mu}}{\dd s}(s,z)-\sum_{i,j=1}^da_{ij}(m^{t,\mu}(s,\cdot))\frac{\dd^2h^{t,\mu}}{\dd z_i\dd z_j}(s,z)\\
			&\qquad+\sum_{i=1}^db_i(m^{t,\mu}(s,\cdot))\frac{\dd h^{t,\mu}}{\dd z_i}(s,z)=0,\quad(s,z)\in(t,T]\times\brd,\\
			&h^{t,\mu}(t,z)=\delta(z),\quad z\in\brd.
		\end{aligned}
		\right.
	\end{equation}
    Then, $f^{t,\mu}(s,x,y):=h^{t,\mu}(s,x-y)$ for $(s,x,y)\in [t,T]\times\brd\times\brd$ is a weak solution of equation \eqref{f}. For any $\varphi\in L^{2}(\brd)$, the function $\phi^{t,\mu}$ defined as \eqref{phi} satisfies for $(s,y)\in(t,T]\times\brd$,
    \begin{align*}
    	&\frac{\dd \phi^{t,\mu}}{\dd s}(s,y)=\int_{\brd}\frac{\dd f^{t,\mu}}{\dd s}(s,x,y)\varphi(x)dx=\int_{\brd}\frac{\dd h^{t,\mu}}{\dd s}(s,x-y)\varphi(x)dx\notag\\
    	&=\int_{\brd} \sum_{i,j=1}^da_{ij}(m^{t,\mu}(s,\cdot)) \frac{\dd^2 }{\dd x_i\dd x_j}[h^{t,\mu}(s,x-y)]\varphi(x)\notag\\
    	&\qquad-\sum_{i=1}^db_i(m^{t,\mu}(s,\cdot))\frac{\dd }{\dd x_i}[h^{t,\mu}(s,x-y)]\varphi(x)dx\notag\\
    	&=\int_{\brd} \sum_{i,j=1}^da_{ij}(m^{t,\mu}(s,\cdot))\frac{\dd^2}{\dd y_i\dd y_j}[h^{t,\mu}(s,x-y)]\varphi(x)\notag\\
    	&\qquad+\sum_{i=1}^db_i(m^{t,\mu}(s,\cdot))\frac{\dd}{\dd y_i}[h^{t,\mu}(s,x-y)]\varphi(x)dx\notag\\
    	&=\int_{\brd} \sum_{i,j=1}^da_{ij}(m^{t,\mu}(s,\cdot))\frac{\dd^2}{\dd y_i\dd y_j}[f^{t,\mu}(s,x,y)]\varphi(x)\notag\\
    	&\qquad+\sum_{i=1}^db_i(m^{t,\mu}(s,\cdot))\frac{\dd}{\dd y_i}[f^{t,\mu}(s,x,y)]\varphi(x)dx\notag\\
    	&=\sum_{i,j=1}^da_{ij}(m^{t,\mu}(s,\cdot))\frac{\dd^2 \phi^{t,\mu}}{\dd y_i\dd y_j}(s,y)+\sum_{i=1}^db_i(m^{t,\mu}(s,\cdot))\frac{\dd \phi^{t,\mu}}{\dd y_i}(s,y).\notag
    \end{align*}
    And from the definition of the function $\delta$, we have
    \begin{align*}
    	\phi^{t,\mu}(t,y)=\int_{\brd} \delta(x-y)\varphi(x)dx=\varphi(y),\quad y\in\brd.\label{rm1_0.2}
    \end{align*}
    Therefore, $\phi^{t,\mu}$ satisfies a second-order parabolic PDE. Similar to the proof of Lemma~\ref{thm1}, we have \eqref{lem8_1} and \eqref{lem8_1'}. Then, \eqref{lem10_1} is a consequence of \eqref{lem8_1} and Proposition~\ref{lemma6}.
\end{proof}

Here are the assumptions on coefficients $(b,\sigma)$ for our special cases. For notational convenience, we use the same constant $L>0$ for all the conditions below.

\textbf{(P1)} The function $\sigma$ is independent of $x$. And there exists $\gamma>0$, such that
\begin{align*}
	\sum_{i,j=1}^da_{ij}(m)\xi_i\xi_j\geq\gamma|\xi|^2,\quad \forall\xi\in\brd,\quad m\in L^2(\brd).
\end{align*}

\textbf{(P2)} The function $b$ is independent of $x$. The functionals $f_{ij}:=b_i,\sigma_{ij}$ have derivatives $\frac{\dd f_{ij}}{\dd m}(m)(\cdot)\in W^{1,2}(\brd)$ for all $m\in L^2(\brd)$ and $1\le i,j\le d$. The functionals and derivatives are bounded by $L$. The functionals $f_{ij}$ and the derivatives $\frac{\dd f_{ij}}{\dd m}(\cdot): L^2(\brd)\to L^2(\brd)$ are $L$-Lipschitz continuous. That is, for any $m,m'\in  L^2(\brd)$ and $1\le i,j\le d$,
\begin{align*}
	&|f_{ij}(m)|+\|\frac{\dd f_{ij}}{\dd m}(m)(\cdot)\|_{W^{1,2}(\brd)}\le L;\\
	&|f_{ij}(m')-f_{ij}(m)|+\|\frac{\dd f_{ij}}{\dd m}(m')(\cdot)-\frac{\dd f_{ij}}{\dd m}(m)(\cdot)\|_{L^2(\brd)}\le L\|m'-m\|_{L^2(\brd)}.
\end{align*}

Assumptions (P1) and (P2) imply (H1) and (H2') immediately, and (H3) (see  Lemma~\ref{exp1}). In view of Theorem~\ref{main}, if coefficients $(b,\sigma)$ satisfy Assumptions (P1) and (P2) and function $\Phi$ satisfies Assumption (H4), then, the value function $V(t,x,\mu)=\e[\Phi(X_T^{t,x,\mu},m^{t,\mu}(T,\cdot))]$, $(t,x,\mu)\in[0,T]\times\brd\times W^{3,2}(\brd)$ is the unique solution in $C^{1,2,1}_b$ of the PDE
\begin{equation*}
	\left\{
	\begin{aligned}
		&\frac{\dd V}{\dd t}(t,x,\mu)+\sum_{i=1}^d\frac{\dd V}{\dd x_i}(t,x,\mu)b_i(\mu)+\sum_{i,j=1}^d\frac{\dd^2V}{\dd x_i\dd x_j}(t,x,\mu)a_{ij}(\mu)\\
		&+\int_{\brd}\frac{\dd V}{\dd \mu}(t,x,\mu)(\xi)\Big\{\sum_{i,j=1}^da_{ij}(\mu(\cdot))\frac{\dd^2\mu}{\dd \xi_i\dd\xi_j}(\xi)-\sum_{i=1}^d b_i(\mu(\cdot))\frac{\dd \mu}{\dd\xi_i}(\xi)\Big\}d\xi=0,\\
		&\quad\qquad\qquad\qquad\qquad\qquad\qquad(t,x,\mu)\in[0,T)\times\brd\times W^{3,2}(\brd);\\
		&V(T,x,\mu)=\Phi(x,\mu),\quad (x,\mu)\in\brd\times W^{3,2}(\brd).
	\end{aligned}
	\right.
\end{equation*}

\begin{remark}
	Bensoussan et al. \cite{AB2,AB3} solve the mean-field games master equation for the linear quadratic problems under the condition that the volatility $\sigma$ is a constant. Our work cannot include their results because of our boundedness assumption. However, our work also go beyond their framework. We do not need the linear quadratic condition and allow the volatility $\sigma$ to depend on the density of the state. Our boundedness assumption is expected to be relaxed with appropriate approximations.
\end{remark}

\section{Discussion on the $L^1$ case}\label{l1}
Our choice of  the Hilbert space $L^2(\brd)$ (see also Bensoussan et al. \cite{AB0,AB2}) for the initial data $\mu$ of the FP equation allows  us to discuss always in Hilbert spaces both the flow $\mu\mapsto m^{t,\mu}(T,\cdot)$ and the derivative of $V(t,x,\mu)$ with respect to $\mu$, which greatly simplify our analysis. However, we could only show that $V$ satisfies the first equality of PDE \eqref{pde_v} for the density variable $\mu\in W^{2,2}(\brd)$ rather than the larger space $L^1(\brd)$. It would be interesting to show that the first equality of PDE \eqref{pde_v} still holds for $\mu\in L^1(\brd)$ by some limit procedure in the equation.

\subsection{Approximation method}
We first give the existence and uniqueness results of FP equation \eqref{pde_m} with initial $\mu\in L^1(\br^d)$ by an approximation method. We let coefficients $(b,\sigma)$ be defined on $\brd\times (W^{2,\infty}(\brd))'$, where $(W^{2,\infty}(\brd))'$ is the dual of the Sobolev space $W^{2,\infty}(\brd)$. For $n\geq 0$, we use the notation
\begin{align*}
	\|\cdot\|_{n,\infty}:=\|\cdot\|_{W^{n,\infty}(\brd)},\qquad \|\cdot\|_{(n,\infty)'}:=\|\cdot\|_{(W^{n,\infty}(\brd))'}.
\end{align*}
Actually, $L^1(\brd)\subset (W^{2,\infty}(\brd))'$ and $\|\cdot\|_{(2,\infty)'}\le \|\cdot\|_{L^1(\brd)}$. Besides Assumptions (H1) and (H2), we assume that the coefficients $(b,\sigma)$ are uniformly bounded and $L$-Lipschitz continuous in $(x,m)\in \brd\times (W^{2,\infty}(\brd))'$. That is, for $(x,m), (x',m') \in \brd\times (W^{2,\infty}(\brd))'$
\begin{equation}\label{L1:A1}
	\begin{split}
		&|(b,\sigma)(x,m)|\le L,\\
		&|(b,\sigma)(x',m')-(b,\sigma)(x,m)|\le L(|x'-x|+\|m'-m\|_{(2,\infty)'}).
	\end{split}
\end{equation}
We give an example.
\begin{example}
	Let $h\in W^{2,\infty}(\brd)$. Then, the function $f$ defined as
	\begin{align*}
		f(m):=\sin[m(h)], \quad m\in (W^{2,\infty}(\brd))',
	\end{align*}
	is uniformly bounded and Lipschitz continuous in $(W^{2,\infty}(\brd))'$.
\end{example}

\begin{lemma}\label{L1:lem1}
	Let Assumptions (H1), (H2) and \eqref{L1:A1} be satisfied. Let $\mu,\mu'\in L^{1}(\brd)$ be probability densities, and assume that $m$ and $m'$ are solutions of FP equation \eqref{pde_m} with initials $\mu$ and $\mu'$, respectively, such that for $s\in[t,T]$, $m(s,\cdot)$ and $m'(s,\cdot)$ are probability densities. Then,
	\begin{align}\label{L1:lem1_0}
		\sup_{t\le s\le T}\|(m'-m)(s,\cdot)\|_{(2,\infty)'}\le C(\gamma,L,T)\|\mu'-\mu\|_{({2,\infty})'}.
	\end{align}
\end{lemma}

\begin{proof}
	Without loss of generality, we restrict ourselves within the one-dimensional case $d=1$. We set $\Delta m:=m'-m$ and $\Delta\mu:=\mu'-\mu$. For $s\in(t,T]$ and $\xi\in W^{2,\infty}(\brd)$, we let $w$ be the solution to the following PDE
	\begin{equation}\label{L1:lem1_1}
		\left\{
		\begin{aligned}
			&\frac{\dd w}{\dd r}(r,x)-a(x,m'(r,\cdot))\frac{\dd^2 w}{\dd x^2}(r,x)\\
			&\qquad+b(x,m'(r,\cdot))\frac{\dd w}{\dd x}(s,x)=0,\qquad (r,x)\in(t,s]\times\brd,\\
			&w(t,x)=\xi(x),\quad x\in\brd.
		\end{aligned}
		\right.
	\end{equation}
	Since $m'$ is already known, PDE \eqref{L1:lem1_1} is a linear parabolic equation and is well studied in the literature. See \cite{PC1,LO} for the existence and uniqueness results and \cite{F.D,TSE} for some estimates of the solution. Similar to \cite[Proposition 2.2]{TSE}, we have
	\begin{equation}\label{L1:lem1_2}
		\sup_{t\le r\le s}\|w(r,\cdot)\|_{2,\infty}\le C(\gamma,L)\|\xi\|_{2,\infty}.
	\end{equation}
	By the definition of $\Delta m$ and $w$, we have
	\begin{align*}
		\int_\br \xi(x)\Delta m(s,x)dx=&\int_\br w(t,x)\Delta\mu(x)dx +\int_t^s \int_\br [(a(x,m'(r))-a(x,m(r)))w_{xx}(r,x)\\
		&+(b(x,m'(r))-b(x,m(r)))w_{x}(r,x)]m(r,x) dxds.
	\end{align*}
	From Assumption \eqref{L1:A1}, estimate \eqref{L1:lem1_2} and the fact that $m(r,\cdot)$ is a probability density, we have
	\begin{align*}
		|\int_\br \xi(x)\Delta m(s,x)dx|&\le \|w(t,\cdot)\|_{2,\infty}\|\Delta\mu\|_{(2,\infty)'}+L\int_t^s \|\Delta m(r,\cdot)\|_{(2,\infty)'}\|w(r,\cdot)\|_{2,\infty} dr\\
		&\le C(\gamma,L)\|\xi\|_{2,\infty}\big(\|\Delta\mu\|_{(2,\infty)'}+\int_t^s \|\Delta m(r,\cdot)\|_{(2,\infty)'}dr\big).
	\end{align*}
	The above eatimate holds true for any $\xi\in W^{2,\infty}(\brd)$, so we have
	\begin{align*}
		\|\Delta m(s,\cdot)\|_{(2,\infty)'}\le C(\gamma,L)\big(\|\Delta\mu\|_{(2,\infty)'}+\int_t^s \|\Delta m(r,\cdot)\|_{(2,\infty)'}dr\big).
	\end{align*}
    From Gronwall's inequality, we have \eqref{L1:lem1_0}.
\end{proof}

As a consequence of Lemmas~\ref{thm1} and \ref{L1:lem1}, we have the following existence and uniqueness results of FP equation \eqref{pde_m} with initial $\mu\in L^1(\brd)$ by using an approximation method.

\begin{theorem}\label{L1:thm1}
	Let Assumptions (H1), (H2) and \eqref{L1:A1} be satisfied and $\mu\in L^{1}(\brd)$ be a probability density. Then, FP equation \eqref{pde_m} has a unique solution $m\in L^\infty([t,T];(W^{2,\infty}(\brd))')$ such that for $s\in[t,T]$, $m(s,\cdot)\in L^1(\brd)$ is a probability density.
\end{theorem}

\begin{proof}
	We only consider the one-dimensional case $d=1$. For a probability density $\mu\in L^{1}(\brd)$, we can find a sequence of probability densities $\{\mu_n\}\subset L^2(\brd)$ such that $\lim_{n\to+\infty}\mu_n=\mu$ in $L^1(\brd)$. From Lemma~\ref{thm1},  FP equation \eqref{pde_m} has a unique solution $m_n\in L^{\infty}([t,T];L^{2}(\brd))$ with initial $\mu_n$. And since $\mu_n$ is a probability density, from the comparison theorem we know that for $s\in[t,T]$, $m_n(s,\cdot)$ is a probability density. From Lemma~\ref{L1:lem1} we know that
	\begin{align*}
		\sup_{t\le s\le T}\|(m_{n_1}-m_{n_2})(s,\cdot)\|_{(2,\infty)'}&\le C(\gamma,L,T)\|\mu_{n_1}-\mu_{n_2}\|_{(2,\infty)'}\\
		&\le C(\gamma,L,T)\|\mu_{n_1}-\mu_{n_2}\|_{L^1(\brd)}, \quad n_1,n_2\geq 1.
	\end{align*}
    Therefore, there exists $m\in L^\infty([t,T];(W^{2,\infty}(\brd))')$ such that
    \begin{align}\label{L1:thm1_1}
    	\lim_{n\to +\infty}\sup_{t\le s\le T}\|(m_{n}-m)(s,\cdot)\|_{(2,\infty)'}=0.
    \end{align}
	For any text function $\phi\in C_c^\infty([t,T)\times\brd)$, we have
	\begin{align*}
		&\int_\br \phi(t,x)\mu_n(x)dx+\int_t^T\int_\br [\phi_s(s,x)+a(x,m_n(s))\phi_{xx}(s,x)\\
		&\qquad+b(x,m_n(s))\phi_x(s,x)] m_n(s,x)dxds=0,\qquad n\geq 1.
	\end{align*}
	From \eqref{L1:thm1_1} and the fact that $\phi\in C_c^\infty([t,T)\times\brd)$, we have
	\begin{align*}
		&\lim_{n\to +\infty}\int_\br \phi(t,x)\mu_n(x)dx=\int_\br \phi(t,x)\mu(x)dx;\\
		&\lim_{n\to +\infty}\int_t^T\int_\br [\phi_s(s,x)+a(x,m_n(s))\phi_{xx}(s,x)+b(x,m_n(s))\phi_x(s,x)] m_n(s,x)dxds\\
		&=\int_t^T\int_\br [\phi_s(s,x)+a(x,m(s))\phi_{xx}(s,x)+b(x,m(s))\phi_x(s,x)] m(s,x)dxds.
	\end{align*}
	Therefore, m is a solution of FP equation \eqref{pde_m} with initial $\mu$. The uniqueness result is a direct consequence of Lemma~\ref{L1:lem1}. And since $\mu$ is a probability density, from the comparison theorem we know that for $s\in[t,T]$, $m(s,\cdot)$ is a probability density.
\end{proof}

Next we study the derivative of $m$ with respect to initial $\mu\in L^1(\brd)$. We first give the definition of differentiable functions on $L^1(\brd)$. We say a function $f:L^1(\brd)\to\br$ is differentiable, if there exists a function $L^1(\brd)\times\brd\ni(m,x)\mapsto \frac{\dd f}{\dd m}(m)(x)$ such that $\frac{\dd f}{\dd m}(m)(\cdot)\in L^\infty(\brd)$ for all $m\in L^1(\brd)$, and for any $\tilde{m}\in L^1(\brd)$,
\begin{equation*}
	\frac{d}{d \theta}f(m+\theta \tilde{m})\Big|_{\theta=0}=\int_{\brd} \frac{\dd f}{\dd m}(m)(x)\tilde{m}(x)dx.
\end{equation*}
We give an example.
\begin{example}
	Let $h\in L^{2}(\brd)\cap L^\infty(\brd)$. The function $f$ defined as
	\begin{align*}
		f(m):=\int_{\brd} h(x)m(x)dx, \quad m\in L^2(\brd)\cup L^1(\brd),
	\end{align*}
    is differentiable in both $L^2(\brd)$ and $L^1(\brd)$.
\end{example}
Besides Assumptions (H1), (H2) and \eqref{L1:A1}, we assume that the functionals $f(x,\cdot):=b_i,\sigma_{ij}(x,\cdot):L^1(\brd)\to \br$ are differentiable for $x\in\brd$ and $1\le i,j\le d$, and satisfy
\begin{align}\label{L1:A2}
	\|\frac{\dd f}{\dd m}(x,m)(\cdot)\|_{2,\infty}\le L,\qquad (x,m)\in \brd\times L^1(\brd).
\end{align}
With an approximation method similar to the proofs of Lemma~\ref{L1:lem1} and Theorem~\ref{L1:thm1}, we have the following existence and uniqueness results of equation \eqref{tm}.
\begin{theorem}\label{L1:thm2}
	Let Assumptions (H1), (H2), \eqref{L1:A1} and \eqref{L1:A2} be satisfied. Let $\mu\in L^{1}(\brd)$ be a probability density and $m^{t,\mu}$ be the solution of equation \eqref{pde_m} with initial $\mu$. Let $\tilde{\mu}\in L^1(\brd)$. Then, equation \eqref{tm} has a unique solution $\tm^{t,\mu}(\tilde{\mu})\in L^\infty([t,T];(W^{2,\infty}(\brd))')$, such that
	\begin{align*}
		\sup_{t\le s\le T}\|\tm^{t,\mu}(\tilde{\mu})(s,\cdot)\|_{(2,\infty)'}\le C(\gamma,L,T)\|\tilde{\mu}\|_{(2,\infty)'}.
	\end{align*}
\end{theorem}

Actually, if we further assume that the functionals $f:=b_i,\sigma_{ij},\ 1\le i,j\le d$, satisfy for $x\in\brd$ and $m,m'\in L^1(\brd)\subset (W^{2,\infty}(\brd))'$,
\begin{equation}\label{L1:A3}
	\begin{split}
		&\|\frac{\dd f}{\dd m}(x,m)(\cdot)\|_{4,\infty}\le L,\quad \|f(\cdot,m')-f(\cdot,m)\|_{2,\infty}\le L\|m'-m\|_{(2,\infty)'},\\
		&\|\frac{\dd f}{\dd m}(x,m')(\cdot)-\frac{\dd f}{\dd m}(x,m)(\cdot)\|_{2,\infty}\le L\|m'-m\|_{(2,\infty)'},
	\end{split}
\end{equation}
we have the following result, whose proof is similar to that of Lemma~\ref{L1:lem1} and is omitted here.

\begin{theorem}
	Let Assumptions (H1), (H2), \eqref{L1:A1}, \eqref{L1:A2} and \eqref{L1:A3} be satisfied. Let $\mu,\mu'\in L^{1}(\brd)$ be probability densities and $\tilde{\mu}:=\mu'-\mu$. Let $m^{t,\mu+h\tilde{\mu}}$ be the solution of equation \eqref{pde_m} with initial $\mu+h\tilde{\mu}$ for $h\in[0,1]$ and let $\tm^{t,\mu}(\tilde{\mu})$ be the solution of equation \eqref{tm}. Then,
	\begin{align*}
		\sup_{t\le s\le T}\|\frac{m^{t,\mu+h\tilde{\mu}}-m^{t,\mu}}{h}(s,\cdot)-\tm^{t,\mu}(\tilde{\mu})(s,\cdot)\|_{(4,\infty)'}\le C(\gamma,L,T)\|\tilde{\mu}\|^2_{(2,\infty)'}h.
	\end{align*}
    That is, $\tm^{t,\mu}(\tilde{\mu})(s,\cdot)$ is the derivative of $m^{t,\mu}(s,\cdot)$ along the direction $\tilde{\mu}$.
\end{theorem}

Following steps in Sections~\ref{SDE}, \ref{regularity} and \ref{PDE_V}, the function $V:[0,T]\times\brd\times L^1(\brd)\to\br$ defined by
\begin{equation*}
	V(t,x,\mu)=\e[\Phi(X^{t,x,\mu}_T,m^{t,\mu}(T,\cdot))]
\end{equation*}
is expected to satisfy the following PDE with an approximation method
\begin{equation*}
	\left\{
	\begin{aligned}
		&\frac{\dd V}{\dd t}(t,x,\mu)+\sum_{i=1}^d\frac{\dd V}{\dd x_i}(t,x,\mu)b_i(x,\mu)+\sum_{i,j=1}^d\frac{\dd^2V}{\dd x_i\dd x_j}(t,x,\mu)a_{ij}(x,\mu)\\
		&\qquad+\int_{\brd}\Big[\sum_{i,j=1}^d a_{ij}(\xi,\mu(\cdot))\frac{\dd^2}{\dd \xi_i\dd\xi_j}\frac{\dd V}{\dd \mu}(t,x,\mu)(\xi)\\
		&\qquad+\sum_{i=1}^d b_i(\xi,\mu(\cdot))\frac{\dd}{\dd\xi_i}\frac{\dd V}{\dd \mu}(t,x,\mu)(\xi)\Big]\mu(\xi)d\xi=0,\\
		&\qquad(t,x,\mu)\in[0,T)\times\brd\times L^{1}(\brd);\\
		&V(T,x,\mu)=\Phi(x,\mu),\quad (x,\mu)\in\brd\times L^{1}(\brd).
	\end{aligned}
	\right.
\end{equation*}
 A rigorous proof still requires the differentiability property of the map $\mu\mapsto V(t,x,\mu)$ and relies on a preliminary study of the derivative $\frac{\dd V}{\dd \mu}$. It is challenging to estimate the $W^{2,\infty}(\brd)$-norm of the function $\frac{\dd V}{\dd \mu}(t,x,\mu)(\cdot):\brd\to\br$ for $(t,x,\mu)\in[0,T]\times\brd\times L^1(\brd)$.

\subsection{FP equation with Nemytskii-type coefficients}
Barbu and R\"ockner \cite{BAR} show the existence and uniqueness of the solution $m^{t,\mu}\in C([0,T];L^1(\brd))$ of the FP equation \eqref{pde_m} with initial $\mu\in L^1(\brd)$ for the case of Nemytskii-type coefficients
\begin{align*}
	&b_i(x,u):=\bar{b}_i(x,u(x)),\quad a_{ij}(x,u):=\bar{a}_{ij}(x,u(x)),\quad\bar{b}_i,\bar{a}_{ij}:\brd\times\br\to\br,
\end{align*}
under some regularity assumptions on coefficients $(\bar{a},\bar{b})$ and the non-degenerate condition
\begin{align}\label{l1_3}
	\sum_{i,j=1}^d(\bar{a}_{ij}(x,u)+(\bar{a}_{ij})_u(x,u))u)\xi_i\xi_j\geq\gamma|\xi|^2,\quad \forall\xi\in\brd,\quad (x,u)\in\brd\times \br,
\end{align}
where $\gamma>0$. For any $\tilde{\mu}\in L^1(\brd)$, the derivative $\tm^{t,\mu}(\tilde{\mu})(s,\cdot)$ of the flow $\mu \mapsto m^{t,\mu}(s,\cdot)$ at the density $m^{t,\mu}(s,\cdot)$ along the direction $\tilde{\mu}$ satisfies
\begin{equation}\label{l1_4}
	\left\{
	\begin{aligned}
		&\frac{\dd \tilde{m}^{t,\mu}(\tilde{\mu})}{\dd s}(s,x)-\sum_{i,j=1}^d\frac{\dd^2}{\dd x_i\dd x_j}[A_{ij}(s,x)\tilde{m}^{t,\mu}(\tilde{\mu})(s,x)]\\
		&\qquad+\sum_{i=1}^d\frac{\dd}{\dd x_i}[B_i(s,x)\tilde{m}^{t,\mu}(\tilde{\mu})(s,x)]=0,\quad(s,x)\in(t,T]\times\brd,\\
		&\tilde{m}^{t,\mu}(\tilde{\mu})(t,x)=\tilde{\mu}(x),\quad x\in\brd,
	\end{aligned}
	\right.
\end{equation}
where for $1\le i,j\le d$,
\begin{align*}
	&A_{ij}(s,x):=\bar{a}_{ij}(x,m^{t,\mu}(s,x))+(\bar{a}_{ij})_u(x,m^{t,\mu}(s,x))m^{t,\mu}(s,x),\\
	&B_i(s,x):=\bar{b}_{i}(x,m^{t,\mu}(s,x))+(\bar{b}_{i})_u(x,m^{t,\mu}(s,x))m^{t,\mu}(s,x).
\end{align*}
 Xia et al. \cite[Theorem 3.1]{Zhang} give the existence and uniqueness results on $W^{2,p}$-locally integrable solutions of second-order parabolic equations with the second order coefficient being non-degenerate and uniformly continuous in $x$, and the first order coefficient being locally integrable. From the existence and uniqueness of $m^{t,\mu}\in C([0,T];L^1(\brd))$ and the non-degenerate condition \eqref{l1_3}, we see that PDE \eqref{l1_4} is a linear second-order parabolic equation with integrable coefficients. Therefore, PDE \eqref{l1_4} is expected to have a solution at least when $\bar{a}$ is further independent of $u$.


\appendix

\section{Proof of Statements in Section 2}\label{pf_2}

\subsection{Proof of Lemma~\ref{thm1}}\label{pf_thm1}
The existence and uniqueness of solutions of equation \eqref{pde_m} is a consequence of the existence and uniqueness of solutions of second-order parabolic equations, Assumptions (H1)-(H2) and the application of Schauder fixed point theorem. We refer to \cite{ARA,PC,ELC,KVN,LO}. Here we only prove estimate \eqref{estimate_m}. We only consider the case of dimension $d=1$; the case $d\geq 1$ can be obtained by an easy extension with standard arguments for second-order parabolic equations \cite{ELC}. For notational convenience, we drop the superscript $(t,\mu)$ in the proof. Without loss of generality, let us temporarily suppose that $m$ is smooth enough. Multiplying both sides of PDE \eqref{pde_m} with $m(s,x)$ and integrating them with respect to $x$, from Assumptions (H1)-(H2), we have
\begin{align*}
	\int_{\br} m_sm(s,x)+\gamma|m_x(s,x)|^2dx&\le\int_{\br} 2L|m_x(s,x)||m(s,x)|dx,\quad s\in(t,T].
\end{align*}
From the average inequality and Gronwall's inequality, we have
\begin{align}\label{thm1_s2_m}
	\sup_{t\le s\le T}\|m(s,\cdot)\|_{L^2(\br)}^2+\gamma\int_t^T\|m_x(s,\cdot)\|^2_{L^2(\br)}ds\le \exp(\frac{4L^2T}{\gamma})\|\mu\|_{L^2(\br)}^2.
\end{align}
Multiplying both sides of PDE \eqref{pde_m} with $a(x,m(s,\cdot))^{-1}m_s(s,x)$ and integrating them with respect to $x$, from Assumption (H2) we have for $s\in(t,T]$,
\begin{align*}
	&\int_{\br} a(x,m(s,\cdot))^{-1}|m_s(s,x)|^2+m_{sx}m_{x}(s,x)dx\\
	&\le \int_{\br} 3La(x,m(s,\cdot))^{-1} |m_x(s,x)||m_s(s,x)|+2La(x,m(s,\cdot))^{-1}|m(s,x)||m_s(s,x)|dx.
\end{align*}
In a similar way, from Assumptions (H1)-(H2), estimate \eqref{thm1_s2_m}, the average inequality and Gronwall's inequality, we have
\begin{equation}\label{thm1_s2_mx}
	\begin{split}
		&\sup_{t\le s\le T}\|m_x(s,\cdot)\|^2_{L^2(\br)}+\frac{1}{L}\int_t^T\|m_s(s,\cdot)\|^2_{L^2(\br)}ds\\
		&\qquad\le\frac{\gamma+8L^2T}{\gamma}\exp(\frac{22L^2T}{\gamma})\|\mu\|^2_{W^{1,2}(\br)}.
	\end{split}
\end{equation}
From equation \eqref{pde_m}, Assumptions (H1)-(H2) and estimates \eqref{thm1_s2_m}-\eqref{thm1_s2_mx}, we have
\begin{equation}\label{thm1_s2_mxx}
\begin{split}
	&\int_t^T \|m_{xx}(s,\cdot)\|^2_{L^2(\br)}ds\\
	&\qquad\le \frac{(1+9LT)(3\gamma L+24L^3T)}{\gamma^3}\exp(\frac{22L^2T}{\gamma})\|\mu\|^2_{W^{1,2}(\br)}.
\end{split}
\end{equation}
For $s\ne s'$, from Cauchy's inequality and Assumption (H2), we have
\begin{align*}
	&\|m(s',\cdot)-m(s,\cdot)\|^2_{L^2(\br)}=\int_{\br}\Big|\int_{s}^{s'}\frac{\partial m}{\partial r}(r,x)dr\Big|^2dx\\
	&\qquad\le |s'-s|\int_t^{T} \int_{\br}3L^2|m_{xx}(r,x)|^2+27L^2|m_{x}(r,x)|^2+12L^2|m(r,x)|^2dxdr.
\end{align*}
Plugging \eqref{thm1_s2_m}-\eqref{thm1_s2_mxx} into the last inequality, we have
\begin{equation*}\label{thm1_s2_s}
	\begin{split}
		&\sup_{s\neq s'}\frac{\|m(s,\cdot)-m(s',\cdot)\|^2_{L^2(\br)}}{|s-s'|}\\
		&\qquad\le \frac{3L^2+13\gamma^2LT}{\gamma^3}{(1+9LT)(3\gamma L+24L^3T)}\exp(\frac{22L^2T}{\gamma})\|\mu\|^2_{W^{1,2}(\br)}.
	\end{split}
\end{equation*}

\subsection{Proof of Proposition~\ref{lemma6}}\label{pf_lemma6}
We continue to restrict ourselves within the one-dimensional case $d=1$, and without loss of generality, we temporarily suppose that $m^{t,\mu}$ and $m^{t,\mu'}$ are smooth enough. We set $\Delta m:=m^{t,\mu'}-m^{t,\mu}$. Then, $\Delta m(t,x)=\Delta \mu(x):=\mu'(x)-\mu(x)$ for $x\in\br$, and for $(s,x)\in(t,T]\times\br$,
\begin{equation}\label{lem6_1}
	\begin{split}
		&\frac{\dd \Delta m}{\dd s}(s,x)-\frac{\dd^2}{\dd x^2}[a(x,{m}^{t,\mu'}(s,\cdot))\Delta m(s,x)]+\frac{\dd}{\dd x}[b(x,{m}^{t,\mu'}(s,\cdot))\Delta m(s,x)]\\
		&\qquad -\frac{\dd^2}{\dd x^2}[(a(x,{m}^{t,\mu'}(s,\cdot))-a(x,{m}^{t,\mu'}(s,\cdot)))m^{t,\mu}(s,x)]\\
		&\qquad +\frac{\dd}{\dd x}[(b(x,{m}^{t,\mu'}(s,\cdot))-b(x,{m}^{t,\mu'}(s,\cdot)))m^{t,\mu}(s,x)]=0.
	\end{split}
\end{equation}
Multiplying both sides of equation \eqref{lem6_1} with $\Delta m(s,x)$ and integrating them with respect to $x$, from Assumptions (H1) and (H2) and the average inequality, we have
\begin{align*}
	&\frac{d}{ds}\|\Delta m(s,\cdot)\|^2_{L^2(\br)}+\gamma \|\Delta m_x(s,\cdot)\|_{L^2(\br)}^2\\
	&\qquad\le  C(\gamma,L)(1+\sup_{t\le s\le T}\|m^{t,\mu}(s,\cdot)\|^2_{W^{1,2}(\br)})\|\Delta m(s,\cdot)\|^2_{L^2(\br)},\quad s\in(t,T].
\end{align*}
So from Lemma~\ref{thm1} and Gronwall's inequality, we have
\begin{equation}\label{lem6_2}
	\begin{split}
		&\sup_{t\le s\le T}\|\Delta m(s,\cdot)\|^2_{L^{2}(\br)}+\int_t^T\|\Delta m_{x}(s,\cdot)\|^2_{L^2(\br)}ds\\
		&\qquad\le C\big(\gamma,L,T,\|\mu\|_{W^{1,2}(\br)}\big)\|\Delta \mu\|^2_{L^{2}(\br)}.
	\end{split}
\end{equation}
Next, we multiply both sides of equation \eqref{lem6_1} with $a(x,m^{t,\mu'}(s,\cdot))^{-1}\Delta m_s(s,x)$ and integrate them with respect to $x$, from Assumptions (H1) and (H2), the average inequality and estimate \eqref{lem6_2}, we have for $s\in(t,T]$,
\begin{align*}
	&\frac{1}{L}\|\Delta m_s(s,\cdot)\|^2_{L^2(\br)}+ \frac{d}{ds}\|\Delta m_{x}(s,\cdot)\|_{L^2(\br)}^2\\
	&\le C(\gamma,L)\Big[\|\Delta m_{x}(s,\cdot)\|_{L^2(\br)}^2+(1+\|m^{t,\mu}(s,\cdot)\|^2_{W^{2,2}})\|\Delta m(s,\cdot)\|_{L^2(\br)}^2\Big]\\
	&\le C(\gamma,L)\|\Delta m_{x}(s,\cdot)\|_{L^2(\br)}^2+C\big(\gamma,L,T,\|\mu\|_{W^{1,2}(\br)}\big)(1+\|m^{t,\mu}(s,\cdot)\|^2_{W^{2,2}})\|\Delta \mu\|^2_{L^{2}(\br)}.
\end{align*}
From Lemma~\ref{thm1} and Gronwall's inequality, we have
\begin{equation}\label{lem6_3}
	\begin{split}
		&\sup_{t\le s\le T}\|\Delta m_x(s,\cdot)\|^2_{L^{2}(\br)}+\int_t^T\|\Delta m_{s}(s,\cdot)\|^2_{L^2(\br)}ds\\
		&\qquad\le C\big(\gamma,L,T,\|\mu\|_{W^{1,2}(\br)}\big)\|\Delta \mu\|^2_{W^{1,2}(\br)}.
	\end{split}
\end{equation}
Last, from equation \eqref{lem6_1}, Assumptions (H1) and (H2) and estimate \eqref{lem6_2}, we have for $(s,x)\in[t,T]\times\br$,
\begin{align*}
	&|\Delta m_{xx}(s,x)|\le C(\gamma,L)\Big[|\Delta m_s(s,x)|+|\Delta m_x(s,x)|+|\Delta m(s,x)|\\
	&\qquad\qquad\qquad\qquad\qquad+\|\Delta m(s,\cdot)\|_{L^2(\br)}(|m^{t,\mu}_{xx}(s,x)|+|m^{t,\mu}_{x}(s,x)|+|m^{t,\mu}(s,x)|)\Big]\\
	&\le C(\gamma,L)\Big[|\Delta m_s(s,x)|+|\Delta m_x(s,x)|+|\Delta m(s,x)|\Big]\\
	&\quad+C\big(\gamma,L,T,\|\mu\|_{W^{1,2}(\br)}\big)\|\Delta \mu\|^2_{L^{2}(\br)}\Big[|m^{t,\mu}_{xx}(s,x)|+|m^{t,\mu}_{x}(s,x)|+|m^{t,\mu}(s,x)|\Big].
\end{align*}
Therefore, from estimates \eqref{lem6_2}-\eqref{lem6_3} and Lemma~\ref{thm1}, we have
\begin{equation*}
	\begin{split}
		\int_t^T\|\Delta m_{xx}(s,\cdot)\|^2_{L^2(\br)}ds\le C\big(\gamma,L,T,\|\mu\|_{W^{1,2}(\br)}\big)\|\Delta \mu\|^2_{W^{1,2}(\br)}.
	\end{split}
\end{equation*}

\subsection{Proof of Proposition~\ref{thm1'}}\label{pf_thm1'}
We continue to restrict ourselves within the one-dimensional case $d=1$, and without loss of generality, we temporarily suppose that $m^{t,\mu}$ and $m^{t,\mu'}$ are smooth enough. For notational convenience, we write $m$ for $m^{t,\mu}$ when there is no confusion. We first prove \eqref{thm1'_1}. Taking the derivative of equation \eqref{pde_m} with respect to $s$, we have for $(s,x)\in(t,T]\times\br$,
\begin{equation}\label{prop1'_1}
	\begin{split}
		&\frac{\dd^2 m}{\dd s^2}(s,x)-\frac{\dd^2}{\dd x^2}[a(x,{m}(s,\cdot))\frac{\dd m}{\dd s}(s,x)]+\frac{\dd}{\dd x}[b(x,{m}(s,\cdot))\frac{\dd m}{\dd s}(s,x)]\\
		&\qquad-\frac{\dd^2}{\dd x^2}[\int_\br\frac{\dd a}{\dd m}(x,m(s,\cdot))(\xi)\frac{\dd m}{\dd s}(s,\xi)d\xi m(s,x)]\\
		&\qquad+\frac{\dd }{\dd x}[\int_\br\frac{\dd b}{\dd m}(x,m(s,\cdot))(\xi)\frac{\dd m}{\dd s}(s,\xi)d\xi m(s,x)]=0.
	\end{split}
\end{equation}
Multiplying both sides of equation \eqref{prop1'_1} with $m_s$ and integrating them with respect to $x$, from Assumptions (H1) and (H2), the average inequality and Lemma~\ref{thm1}, we have
\begin{align*}
	&\frac{d}{ds}\| m_s(s,\cdot)\|^2_{L^2(\br)}+\gamma\int_{\br} |m_{sx}(s,x)|^2dx\\
	&\qquad\le C(\gamma,L)(1+\sup_{t\le s\le T}\|m(s,\cdot)\|_{W^{1,2}(\br)}^2)\| m_s(s,\cdot)\|^2_{L^2(\br)}\\
	&\qquad\le C(\gamma,L,T,\|\mu\|_{W^{1,2}(\br)})\| m_s(s,\cdot)\|^2_{L^2(\br)},\quad s\in(t,T].
\end{align*}
Note from Assumption (H2) that
\begin{align*}
	\| m_s(t,\cdot)\|^2_{L^2(\br)}&=\int_{\br} \Big|\frac{\dd^2}{\dd x^2}[a(x,\mu(\cdot))\mu(x)]-\frac{\dd}{\dd x}[b(x,\mu(\cdot))\mu(x)]\Big|^2dx\le C(L)\|\mu\|_{W^{2,2}(\br)}^2.
\end{align*}
Therefore, by Gronwall's inequality, we have
\begin{equation}\label{prop1'_2}
	\begin{split}
		\sup_{t\le s\le T}\|m_s(s,\cdot)\|_{L^2(\br)}^2+\int_t^T\|m_{sx}(s,\cdot)\|^2_{L^2(\br)}ds\le C\big(\gamma,L,T,\|\mu\|_{W^{2,2}(\br)}\big).
	\end{split}
\end{equation}
From equation \eqref{pde_m} and Assumptions (H1) and (H2), we have
\begin{align*}
	|m_{xx}|\le \frac{1}{\gamma}|m_s|+\frac{3L}{\gamma}|m_x|+\frac{2L}{\gamma}|m|.
\end{align*}
So from \eqref{prop1'_2} and Lemma~\ref{thm1}, we have
\begin{equation}\label{prop1'_2.5}
	\begin{split}
		\sup_{t\le s\le T}\| m_{xx}(s,\cdot)\|_{L^2(\br)}
		\le C\big(\gamma,L,T,\|\mu\|_{W^{2,2}(\br)}\big).
	\end{split}
\end{equation}

We now prove \eqref{thm1'_2}. We set $\Delta m:=m^{t,\mu'}-m^{t,\mu}$. The equation of $\Delta m$ is \eqref{lem6_1}. Taking the derivative of equation \eqref{lem6_1} with respect to $s$, we have for $(s,x)\in(t,T]\times\br$,
\begin{align*}
	&\frac{\dd^2 \Delta m}{\dd s^2}(s,x)-\frac{\dd^2}{\dd x^2}\Big[a(x,{m}^{t,\mu'}(s,\cdot))\frac{\dd \Delta m}{\dd s}(s,x)\label{prop1'_3}\\
	&\quad+ \int_{\br} \frac{\dd a}{\dd m}(x,m^{t,\mu'}(s,\cdot))(\xi)\frac{\dd m^{t,\mu'}}{\dd s}(s,\xi)d\xi\Delta m(s,x)\notag\\
	&\quad+(a(x,{m}^{t,\mu'}(s,\cdot))-a(x,{m}^{t,\mu'}(s,\cdot)))\frac{\dd m^{t,\mu}}{\dd s}(s,x)\notag\\
	&\quad+ \int_{\br} \frac{\dd a}{\dd m}(x,m^{t,\mu'}(s,\cdot))(\xi)\frac{\dd \Delta m}{\dd s}(s,\xi)d\xi m^{t,\mu}(s,x)\notag\\
	&\quad+\int_{\br} \big(\frac{\dd a}{\dd m}(x,m^{t,\mu'}(s,\cdot))(\xi)-\frac{\dd a}{\dd m}(x,m^{t,\mu}(s,\cdot))(\xi)\big)\frac{\dd m^{t,\mu}}{\dd s}(s,\xi)d\xi m^{t,\mu}(s,x)\Big]\notag\\
	&\quad+\frac{\dd}{\dd x}\Big[b(x,{m}^{t,\mu'}(s,\cdot))\frac{\dd \Delta m}{\dd s}(s,x)\notag\\
	&\quad+ \int_{\br} \frac{\dd b}{\dd m}(x,m^{t,\mu'}(s,\cdot))(\xi)\frac{\dd m^{t,\mu'}}{\dd s}(s,\xi)d\xi\Delta m(s,x)\notag\\
	&\quad+(b(x,{m}^{t,\mu'}(s,\cdot))-b(x,{m}^{t,\mu'}(s,\cdot)))\frac{\dd m^{t,\mu}}{\dd s}(s,x)\notag\\
	&\quad+ \int_{\br} \frac{\dd b}{\dd m}(x,m^{t,\mu'}(s,\cdot))(\xi)\frac{\dd \Delta m}{\dd s}(s,\xi)d\xi m^{t,\mu}(s,x)]\notag\\
	&\quad+ \int_{\br} \big(\frac{\dd b}{\dd m}(x,m^{t,\mu'}(s,\cdot))(\xi)-\frac{\dd b}{\dd m}(x,m^{t,\mu}(s,\cdot))(\xi)\big)\frac{\dd m^{t,\mu}}{\dd s}(s,\xi)d\xi m^{t,\mu}(s,x)\Big]=0.\notag
\end{align*}
Multiplying both sides of the last equation with $\Delta m_s$ and integrating them with respect to $x$, from Assumptions (H1) and (H2), the average inequality, Lemma~\ref{thm1}, Proposition~\ref{lemma6},  and estimate \eqref{prop1'_2}, we have for $t<s\le T$,
\begin{align*}
	&\frac{d}{ds}\|\Delta m_s(s,\cdot)\|^2_{L^2(\br)}+\|\Delta m_{sx}(s,\cdot)\|^2_{L^2(\br)}\\
	&\le C(\gamma,L)\Big[\big(1+\|m^{t,\mu}(s,\cdot)\|^2_{W^{1,2}(\br)}+\|m^{t,\mu'}(s,\cdot)\|^2_{W^{1,2}(\br)}\big)\|\Delta m_s(s,\cdot)\|^2_{L^2(\br)}\\
	&\quad+\big(1+\|m^{t,\mu}(s,\cdot)\|^2_{W^{1,2}(\br)}+\|m^{t,\mu'}(s,\cdot)\|^2_{W^{1,2}(\br)}\big)\\
	&\quad\times\big(\|m_s^{t,\mu}(s,\cdot)\|^2_{L^2(\br)}+\|m_s^{t,\mu'}(s,\cdot)\|^2_{L^2(\br)}\big)\|\Delta m(s,\cdot)\|^2_{W^{1,2}(\br)}\\
	&\quad+\big(\|m_{sx}^{t,\mu}(s,\cdot)\|^2_{L^2(\br)}+\|m_{sx}^{t,\mu'}(s,\cdot)\|^2_{L^2(\br)}\big)\|\Delta m(s,\cdot)\|^2_{L^2(\br)}\Big]\\
	&\le C(\gamma,L,T,\|\mu\|_{W^{2,2}(\br)},\|\mu'\|_{W^{2,2}(\br)})\Big[\|\Delta m_s(s,\cdot)\|^2_{L^2(\br)}\\
	&\quad+\big(1+\|m_{sx}^{t,\mu}(s,\cdot)\|^2_{L^2(\br)}+\|m_{sx}^{t,\mu'}(s,\cdot)\|^2_{L^2(\br)}\big)\|\Delta\mu\|^2_{W^{1,2}(\br)}\Big]
\end{align*}
Note from Assumption (H2) that
\begin{align*}
	\|\Delta m_s(t,\cdot)\|^2_{L^2(\br)}&=\int_{\br} \Big|\frac{\dd^2}{\dd x^2}[a(x,\mu'(\cdot))\mu'(x)-a(x,\mu(\cdot))\mu(x)]\\
	&\qquad-\frac{\dd}{\dd x}[b(x,\mu'(\cdot))\mu'(x)-b(x,\mu(\cdot))\mu(x)]\Big|^2dx\\
	&\le C(L,\|\mu\|_{W^{2,2}(\br)})\|\Delta\mu\|^2_{W^{2,2}(\br)}.
\end{align*}
Therefore, by Gronwall's inequality and estimate \eqref{prop1'_2}, we have
\begin{equation}\label{prop1'_4}
	\begin{split}
		&\sup_{t\le s\le T}\|\Delta m_s(s,\cdot)\|_{L^2(\br)}^2+\int_t^T\|\Delta m_{sx}(s,\cdot)\|^2_{L^2(\br)}ds\\
		&\qquad\le C\big(\gamma,L,T,\|\mu\|_{W^{2,2}(\br)},\|\mu'\|_{W^{2,2}(\br)}\big)\|\Delta\mu\|^2_{W^{2,2}(\br)}.
	\end{split}
\end{equation}
Last, from equation \eqref{lem6_1} and Assumptions (H1) and (H2), we have
\begin{align*}
	|\Delta m_{xx}|\le C(\gamma,L)\Big[|\Delta m_s|+|\Delta m_x|+|\Delta m|+\|\Delta m\|_{L^2(\br)}(|m_{xx}^{t,\mu}|+|m_x^{t,\mu}|+|m^{t,\mu}|)\Big].
\end{align*}
So from Lemma~\ref{thm1}, Proposition~\ref{lemma6} and estimates \eqref{prop1'_2.5} and \eqref{prop1'_4}, we have
\begin{equation*}
	\begin{split}
		\sup_{t\le s\le T}\|\Delta m_{xx}(s,\cdot)\|^2_{L^2(\br)}\le C\big(\gamma,L,T,\|\mu\|_{W^{2,2}(\br)},\|\mu'\|_{W^{2,2}(\br)}\big)\|\Delta\mu\|^2_{W^{2,2}(\br)}.
	\end{split}
\end{equation*}

We now prove \eqref{thm1''_1}. Taking the derivative of equation \eqref{pde_m} with respect to $x$, we have for $(s,x)\in(t,T]\times\br$,
\begin{equation}\label{thm1''_2}
	\begin{split}
		\frac{\dd^2m}{\dd s\dd x}(s,x)-\frac{\dd^3}{\dd x^3}[a(x,m(s,\cdot))m(s,x)]+\frac{\dd^2}{\dd x^2}[b(x,m(s,\cdot))m(s,x)]=0.
	\end{split}
\end{equation}
From Assumptions (H1) and (H2'), we have
\begin{align}\label{thm1''_4}
	|m_{xxx}|\le C(\gamma,L)\Big[|m_{sx}|+|m_{xx}|+|m_{x}|+|m|\Big].
\end{align}
From Lemma~\ref{thm1} and estimate \eqref{thm1'_1}, we have
\begin{align*}
	&\|m\|_{L^2([t,T];W^{3,2}(\br))}\le C(\gamma,L,T,\|\mu\|_{W^{2,2}(\br)}).
\end{align*}
For $s,s'\in[t,T]$, from equation \eqref{thm1''_2}, Cauchy's inequality and estimate \eqref{thm1'_1}, we have
\begin{align*}
	&\|m_x(s',\cdot)-m_x(s,\cdot)\|^2_{L^2(\br)}\le |s'-s|\int_s^{s'}\int_{\br}\Big|m_{rx}(r,x)\Big|^2dxdr\\
	&\le |s'-s|\int_t^T \|m_{sx}(s,\cdot)\|^2_{L^2(\br)}ds\le C\big(\gamma,L,T,\|\mu\|_{W^{2,2}(\br)}\big) |s'-s|.
\end{align*}

We now prove \eqref{cor2_1}. Multiplying both sides of \eqref{prop1'_1} with $a(x,m(s,\cdot))^{-1}m_{ss}(s,x)$ and integrating them with respect to $x$, from Assumptions (H1) and (H2), the average inequality, Lemma~\ref{thm1} and estimate \eqref{thm1'_1}, we have
\begin{align*}
	&\|m_{ss}(s,\cdot)\|_{L^2(\br)}^2+\frac{d}{ds}\|m_{sx}(s,\cdot)\|^2_{L^2(\br)}\\
	&\qquad\le C(\gamma,L)\Big[\|m_{sx}(s,\cdot)\|^2+(1+\|m(s,\cdot)\|^2_{W^{2,2}(\br)})\|m_s(s,\cdot)\|^2\Big]\\
	&\qquad\le C(\gamma,L)\|m_{sx}(s,\cdot)\|^2+C(\gamma,L,T,\|\mu\|_{W^{2,2}(\br)}),\quad s\in(t,T].
\end{align*}
Note from equation \eqref{thm1''_2} and Assumption (H2') that
\begin{align*}
	\| m_{sx}(t,\cdot)\|^2_{L^2(\br)}&=\int_{\br} \Big|\frac{\dd^3}{\dd x^3}[a(x,\mu(\cdot))\mu(x)]-\frac{\dd^2}{\dd x^2}[b(x,\mu(\cdot))\mu(x)]\Big|^2dx\\
	&\le C(L)\|\mu\|_{W^{3,2}(\br)}^2.
\end{align*}
Therefore, using Gronwall's inequality, we have
\begin{equation}\label{cor2_2}
	\begin{split}
		&\sup_{t\le s\le T}\|m_{sx}(s,\cdot)\|_{L^2(\br)}^2+\int_t^T\|m_{ss}(s,\cdot)\|^2_{L^2(\br)}ds\le C\big(\gamma,L,T,\|\mu\|_{W^{3,2}(\br)}\big).
	\end{split}
\end{equation}
From Lemma~\ref{thm1} and estimates \eqref{thm1'_1}, \eqref{thm1''_4} and \eqref{cor2_2}, we have
\begin{equation*}
	\begin{split}
		&\sup_{t\le s\le T}\|m_{xxx}(s,\cdot)\|_{L^2(\br)}\le C\big(\gamma,L,T,\|\mu\|_{W^{3,2}(\br)}\big).
	\end{split}
\end{equation*}
From equation \eqref{prop1'_1}, Assumptions (H1) and (H2), we have
\begin{align*}
	|m_{sxx}|\le C(\gamma,L)\Big[ |m_{ss}|+|m_{sx}|+|m_{s}|+\|m_s\|_{L^2(\br)}(|m_{xx}|+|m_x|+|m|)\Big].
\end{align*}
Therefore, from Lemma~\ref{thm1} and estimates \eqref{thm1'_1} and \eqref{cor2_2}, we have
\begin{equation*}
	\begin{split}
		&\int_t^T\|m_{sxx}(s,\cdot)\|^2_{L^2(\br)}ds\le C\big(\gamma,L,T,\|\mu\|_{W^{3,2}(\br)}\big).
	\end{split}
\end{equation*}

We now prove \eqref{thm1'''_1}. From Cauchy's inequality and estimate \eqref{cor2_1}, we have for $s,s'\in(t,T]$,
\begin{align*}
	&\|m_s(s',\cdot)-m_s(s,\cdot)\|^2_{L^2(\br)}\le |s'-s|\int_s^{s'}\int_{\br}\Big|m_{rr}(r,x)\Big|^2dxdr\\
	&\qquad\le |s'-s|\int_t^T \|m_{ss}(s,\cdot)\|^2_{L^2(\br)}\le C\big(\gamma,L,T,\|\mu\|_{W^{3,2}(\br)}\big) |s'-s|;\\
	&\|m_{xx}(s',\cdot)-m_{xx}(s,\cdot)\|^2_{L^2(\br)}\le |s'-s|\int_s^{s'}\int_{\br}\Big|m_{rxx}(r,x)\Big|^2dxdr\\
	&\qquad\le |s'-s|\int_t^T \|m_{sxx}(s,\cdot)\|^2_{L^2(\br)}\le C\big(\gamma,L,T,\|\mu\|_{W^{3,2}(\br)}\big) |s'-s|.
\end{align*}
The proof is completed.

\section{Proof of Statements in Section 3}\label{pf_3}
\subsection{Proof of Lemma~\ref{thm4}}\label{pf_thm4}
We continue to restrict ourselves within the one-dimensional case $d=1$. For notational convenience, we write $\tm$ instead of $\tm^{t,\mu}(\tilde{\mu})$ in the proof when there is no ambiguity. We first prove the existence and uniqueness result. The proof relies on the use of Banach fixed point theorem. We define a map $\Phi:C^0([t,T];L^2(\br))\to C^0([t,T];L^2(\br))$ as: for any $\hat{m}\in C^0([t,T];L^2(\br))$, we set $\Phi(\hat{m})=\tm$ as the solution of the following PDE:
\begin{equation}\label{tm_1}
	\left\{
	\begin{aligned}
		&\frac{\dd \tilde{m}}{\dd s}(s,x)-\frac{\dd^2}{\dd x^2}\Big[a(x,m^{t,\mu}(s,\cdot))\tilde{m}(s,x)\\
		&\qquad+\int_{\br}\frac{\dd a}{\dd m}(x,m^{t,\mu}(s,\cdot))(\xi)\hm(s,\xi)d\xi \cdot m^{t,\mu}(s,x)\Big]\\
		&\qquad+\frac{\dd}{\dd x}\Big[b(x,m^{t,\mu}(s,\cdot))\tilde{m}(s,x)\\
		&\qquad+\int_{\br}\frac{\dd b}{\dd m}(x,m^{t,\mu}(s,\cdot))(\xi)\hm(s,\xi)d\xi \cdot m^{t,\mu}(s,x)\Big]=0, \\
		&\qquad\qquad\qquad\qquad\qquad\qquad\qquad\qquad (s,x)\in(t,T]\times\br,\\
		&\tilde{m}(t,x)=\tilde{\mu}(x),\quad x\in\br.
	\end{aligned}
	\right.
\end{equation}
Note from \eqref{estimate_m} that the function
\begin{align*}
	(s,x)\to&\frac{\dd^2}{\dd x^2}\Big[\int_{\br}\frac{\dd a}{\dd m}(x,m^{t,\mu}(s,\cdot))\hm(s,\xi)d\xi \cdot m^{t,\mu}(s,x)\Big]\\ &-\frac{\dd}{\dd x}\Big[\int_{\br}\frac{\dd b}{\dd m}(x,m^{t,\mu}(s,\cdot))\hm(s,\xi)d\xi \cdot m^{t,\mu}(s,x)\Big]
\end{align*}
belongs to $\in L^2([t,T]\times\br)$. So from \cite[Definition and Remark, p.374; Theorems 3 and 4, p.378]{ELC}, we know that equation \eqref{tm_1} has a unique solution $\tm\in C^0([t,T];L^2(\br))$. Therefore, $\Phi$ is well-defined. Let $\hat{m}^i\in C^0([t,T];L^2(\br))$ and $\tm^i=\Phi(\hat{m}^i)$ for $i=1,2$, we set $\Delta \hm=\hm^2-\hm^1$ and $\Delta \tm=\tm^2-\tm^1$. Then, $\Delta \tm$ satisfies
\begin{align*}
	&\frac{\dd \Delta \tilde{m}}{\dd s}(s,x)-\frac{\dd^2}{\dd x^2}\Big[a(x,m^{t,\mu}(s,\cdot))\Delta \tilde{m}(s,x)\\
	&\qquad+\int_{\br}\frac{\dd a}{\dd m}(x,m^{t,\mu}(s,\cdot))(\xi)\Delta \hm(s,\xi)d\xi \cdot m^{t,\mu}(s,x)\Big]\\
	&\qquad+\frac{\dd}{\dd x}\Big[b(x,m^{t,\mu}(s,\cdot))\Delta \tilde{m}(s,x)\\
	&\qquad+\int_{\br}\frac{\dd b}{\dd m}(x,m^{t,\mu}(s,\cdot))(\xi)\Delta \hm(s,\xi)d\xi \cdot m^{t,\mu}(s,x)\Big]=0, \quad (s,x)\in(t,T]\times\br,
\end{align*}
and the initial condition $\Delta \tilde{m}(t,x)=0$  for $x\in\br$. Without loss of generality, we temporarily suppose that the functions $(m,\tm^1,\tm^2)$ are smooth enough. Multiplying both sides of the last equation with $\Delta\tm$ and integrating with respect to $x$, from Assumptions (H1) and (H2), the average inequality and \eqref{estimate_m}, we have for $s\in(t,T]$,
\begin{align*}
	&\frac{d}{ds}\|\Delta\tilde{m}(s,\cdot)\|^2_{L^2(\br)}+\gamma\|\Delta\tilde{m}_x(s,\cdot)\|^2_{L^2(\br)}\\
	&\qquad\le \frac{12L^2}{\gamma}\|\Delta\tm(s,\cdot)\|^2_{L^2(\br)}+\frac{12L^2}{\gamma}\|m^{t,\mu}(s,\cdot)\|^2_{W^{1,2}(\br)}\|\Delta\hm(s,\cdot)\|^2_{L^2(\br)}\\
	&\qquad\le \frac{12L^2}{\gamma}\|\Delta\tm(s,\cdot)\|^2_{L^2(\br)}+ C(\gamma,L,T,\|\mu\|_{W^{1,2}(\br)})\|\Delta\hm(s,\cdot)\|^2_{L^2(\br)}.
\end{align*}
Note that $\|\Delta \tilde{m}(t,\cdot)\|^2_{L^2(\br)}=0$. Therefore, from Gronwall's inequality, we have
\begin{align*}
	\|\Delta\tm\|_{C^0([t,T];L^2(\br))}\le C(\gamma,L,T,\|\mu\|_{W^{1,2}(\br)})|T-t|^{\frac{1}{2}}\|\Delta\hm\|_{C^0([t,T];L^2(\br))}.
\end{align*}
So $\Phi$ is a contraction map for sufficiently small $(T-t)$. For general $(T-t)$, we can divide $(T-t)$ into several parts and on each part $\Phi$ is a contraction. From Banach fixed point theorem, we get the existence and uniqueness result.

We now prove \eqref{estimate_tm1}. We temporarily suppose that the functions $(m,\tm)$ are smooth enough. Multiplying both sides of equation \eqref{tm} with $\tm$ and integrating them with respect to $x$, from Assumptions (H1) and (H2), the average inequality and \eqref{estimate_m}, we have
\begin{align*}
	&\frac{d}{ds}\|\tm(s,\cdot)\|_{L^2(\br)}^2+\gamma\|\tm_x(s,\cdot)\|_{L^2(\br)}^2\le \frac{12L^2}{\gamma}(1+\|m^{t,\mu}(s,\cdot)\|_{W^{1,2}(\br)}^2)\|\tm(s,\cdot)\|_{L^2(\br)}^2\\
	&\qquad\le C(\gamma,L,T,\|\mu\|_{W^{1,2}(\br)})\|\tm(s,\cdot)\|_{L^2(\br)}^2,\quad s\in(t,T].
\end{align*}
Note that
\begin{align*}
	\|\tilde{m}(t,\cdot)\|^2_{L^2(\br)}=\|\tilde{\mu}\|_{L^2(\br)}^2.
\end{align*}
Therefore, from Gronwall's inequality, we have \eqref{estimate_tm1}.

We now prove \eqref{estimate_tm2}. We set for $(s,x)\in[t,T]\times\br$,
\begin{align*}
	&v:=\tm^{t,\mu'}(\tilde{\mu})-\tm^{t,\mu}(\tilde{\mu}),\quad\Delta m :=m^{t,\mu'}-m^{t,\mu},\quad \Delta \mu:=\mu'-\mu,\\
	&\Delta a(s,x):=a(x,m^{t,\mu'}(s,\cdot))-a(x,m^{t,\mu}(s,\cdot)),\\
	&\Delta b(s,x):=b(x,m^{t,\mu'}(s,\cdot))-b(x,m^{t,\mu}(s,\cdot)).
\end{align*}
Then, $v(t,x)=0$ for $x\in\br$, and for $(s,x)\in(t,T]\times\br,$
\begin{align*}
	&\frac{\dd v}{\dd s}(s,x)-\frac{\dd^2}{\dd x^2}\Big[a(x,m^{t,\mu'}(s,\cdot))v(s,x)+\Delta a(s,x)\tm^{t,\mu}(\tilde{\mu})(s,x)\notag\\
	&\qquad+\int_{\br}\frac{\dd a}{\dd m}(x,m^{t,\mu'}(s,\cdot))(\xi)v(s,\xi)d\xi \cdot m^{t,\mu}(s,x)\\
	&\qquad+\int_{\br}\frac{\dd \Delta a}{\dd m}(s,x)(\xi)\tm^{t,\mu}(\tilde{\mu})(s,\xi)d\xi \cdot m^{t,\mu}(s,x)\Big]\notag\\
	&\qquad+\int_{\br}\frac{\dd a}{\dd m}(x,m^{t,\mu'}(s,\cdot))(\xi)\tm^{t,\mu'}(\tilde{\mu})(s,\xi)d\xi \cdot \Delta m(s,x)\notag\\
	&\qquad+\frac{\dd}{\dd x}\Big[b(x,m^{t,\mu'}(s,\cdot))v(s,x)+\Delta b(s,x)\tm^{t,\mu}(\tilde{\mu})(s,x)\notag\\
	&\qquad+\int_{\br}\frac{\dd b}{\dd m}(x,m^{t,\mu'}(s,\cdot))(\xi)v(s,\xi)d\xi \cdot m^{t,\mu}(s,x)\\
	&\qquad+\int_{\br}\frac{\dd \Delta b}{\dd m}(s,x)(\xi)\tm^{t,\mu}(\tilde{\mu})(s,\xi)d\xi \cdot m^{t,\mu}(s,x)\Big]\notag\\
	&\qquad+\int_{\br}\frac{\dd b}{\dd m}(x,m^{t,\mu'}(s,\cdot))(\xi)\tm^{t,\mu'}(\tilde{\mu})(s,\xi)d\xi \cdot \Delta m(s,x)=0.
\end{align*}
We temporarily suppose that the functions $(m^{t,\mu},m^{t,\mu'},\tm^{t,\mu}(\tilde{\mu}),\tm^{t,\mu'}(\tilde{\mu}))$ are smooth enough. Multiplying both sides of the last equation with $v$ and integrating with respect to $x$, from Assumptions (H1) and (H2), the average inequality, Lemma~\ref{thm1}, Proposition~\ref{lemma6} and estimate \eqref{estimate_tm1}, we have for $s\in(t,T]$,
\begin{align*}
	&\frac{d}{ds}\|v(s,\cdot)\|^2_{L^2(\br)}+\|v_x(s,\cdot)\|^2_{L^2(\br)}\\
	&\le C(\gamma,L)\Big[\|v(s,\cdot)\|^2_{L^2(\br)}(1+\|m^{t,\mu}(s,\cdot)\|^2_{W^{1,2}(\br)})\\
	&\quad+\|\Delta m(s,\cdot)\|^2_{L^2(\br)}\|\tm^{t,\mu}(\tilde{\mu})(s,\cdot)\|^2_{W^{1,2}(\br)}+\|\Delta m(s,\cdot)\|^2_{W^{1,2}(\br)}\|\tm^{t,\mu'}(\tilde{\mu})(s,\cdot)\|^2_{L^{2}(\br)}\\
	&\quad+\|\Delta m(s,\cdot)\|^2_{L^{2}(\br)}\|\tm^{t,\mu}(\tilde{\mu})(s,\cdot)\|^2_{L^{2}(\br)}\|m^{t,\mu}(s,\cdot)\|^2_{W^{1,2}(\br)}\Big]\\
	&\le C(\gamma,L,T,\|\mu\|_{W^{1,2}(\br)},\|\mu'\|_{W^{1,2}(\br)})\Big[\|v(s,\cdot)\|^2_{L^2(\br)}+\|\Delta \mu\|^2_{L^2(\br)}\|\tm^{t,\mu}(\tilde{\mu})(s,\cdot)\|^2_{W^{1,2}(\br)}\\
	&\quad+\|\Delta m(s,\cdot)\|^2_{W^{1,2}(\br)}\|\tilde{\mu}\|^2_{L^{2}(\br)}+\|\Delta \mu\|^2_{L^{2}(\br)}\|\tilde{\mu}\|^2_{L^{2}(\br)}\Big].
\end{align*}
By using Gronwall's inequality, Proposition~\ref{lemma6} and estimate \eqref{estimate_tm1}, we have
\begin{align*}
	&\sup_{t\le s\le T}\|v(s,\cdot)\|_{L^2(\br)}\le C(\gamma,L,T,\|\mu\|_{W^{1,2}(\br)},\|\mu'\|_{W^{1,2}(\br)})\|\tilde{\mu}\|_{L^{2}(\br)}\|\Delta \mu\|_{L^2(\br)}.
\end{align*}

\subsection{Proof of Lemma~\ref{lem5}}\label{pf_lem5}
We continue to restrict ourselves within the one-dimensional case $d=1$. For notational convenience, we write $\tm_h$ instead of $\tm^{t,\mu}_h(\tilde{\mu})$ defined in \eqref{m_h:def} and write $\tm$ instead of $\tm^{t,\mu}(\tilde{\mu})$ in the proof. We set $u^h:=\tm_h-\tm$. Then, $u^h(t,x)=0$ for $x\in\br$, and for $(s,x)\in(t,T]\times\br$,
\begin{align*}
	&0=\frac{\dd u^h}{\dd s}(s,x)-\frac{\dd^2}{\dd x^2}\Big[a(x,m^{t,\mu_h}(s,\cdot))u^h(s,x)\\
	&\quad+[a(x,m^{t,\mu_h}(s,\cdot))-a(x,m^{t,\mu}(s,\cdot))]\tm(s,x)\notag\\
	&\quad+\int_0^1\int_{\br}\frac{\dd a}{\dd m}(x,m^{t,\mu}+\lambda h\tm_h(s,\cdot))(\xi)u^h(s,\xi)d\xi d\lambda \cdot m^{t,\mu}(s,x)\notag\\
	&\quad+\int_0^1\int_{\br}[\frac{\dd a}{\dd m}(x,m^{t,\mu}+\lambda h\tm_h(s,\cdot))-\frac{\dd a}{\dd m}(x,m^{t,\mu}(x,\cdot))]\tm(s,\xi)d\xi d\lambda \cdot m^{t,\mu}(s,x)\Big]\notag\\
	&\quad+\frac{\dd}{\dd x}\Big[b(x,m^{t,\mu_h}(s,\cdot))u^h(s,x)+[b(x,m^{t,\mu_h}(s,\cdot))-b(x,m^{t,\mu}(s,\cdot))]\tm(s,x)\notag\\
	&\quad+\int_0^1\int_{\br}\frac{\dd b}{\dd m}(x,m^{t,\mu}+\lambda h\tm_h(s,\cdot))(\xi)u^h(s,\xi)d\xi d\lambda \cdot m^{t,\mu}(s,x)\notag\\
	&\quad+\int_0^1\int_{\br}[\frac{\dd b}{\dd m}(x,m^{t,\mu}+\lambda h\tm_h(s,\cdot))-\frac{\dd b}{\dd m}(x,m^{t,\mu}(x,\cdot))]\tm(s,\xi)d\xi d\lambda \cdot m^{t,\mu}(s,x)\Big].
\end{align*}
Without loss of generality, we temporarily suppose that the functions $(u^h,\tm,m^{t,\mu})$ are smooth enough. Multiplying both sides of the last equation with $u^h$ and integrating with respect to $x$, from Assumptions (H1) and (H2), the average inequality, Lemma~\ref{thm1} and estimate \eqref{lem3_1}, we have for $s\in(t,T]$,
\begin{align*}
	&\frac{d}{ds}\|u^h(s,\cdot)\|_{L^2(\br)}^2+\|u^h_x(s,\cdot)\|^2_{L^2(\br)}\\
	&\le C(\gamma,L)\big(1+\|m^{t,\mu}(s,\cdot)\|^2_{W^{1,2}(\br)}\big)\Big[\|u^h(s,\cdot)\|^2_{L^2(\br)}\\
	&\qquad+h^2\|\tm_h(s,\cdot)\|^2_{L^2(\br)}\|\tm(s,\cdot)\|^2_{W^{1,2}(\br)}\Big]\\
	&\le C(\gamma,L,T,\|\mu\|_{W^{1,2}(\br)})\Big[\|u^h(s,\cdot)\|^2_{L^2(\br)}+h^2\|\tilde{\mu}\|^2_{L^2(\br)}\|\tm(s,\cdot)\|^2_{W^{1,2}(\br)}\Big].
\end{align*}
Therefore, from Gronwall's inequality and Lemma~\ref{thm4}, we have
\begin{align}\label{lem5_2}
	\sup_{t\le s\le T}\|u^h(s,\cdot)\|_{L^2(\br)}^2\le C\big(\gamma,L,T,\|\mu\|_{W^{1,2}(\br)}\big)\|\tilde{\mu}\|_{L^{2}(\br)}^4h^2.
\end{align}

\subsection{Proof of Proposition~\ref{lem7}}\label{pf_lem7}
From the uniqueness result in Lemma~\ref{thm4}, we only need to check that the function
\begin{align*}
	u^{t,\mu,\tilde{\mu}}(s,x):=\int_{\brd} k^{t,\mu}(s,x,y)\tilde{\mu}(y)dy,\quad (s,x)\in [t,T]\times\brd,
\end{align*}
satisfies equation \eqref{tm}. Actually, from \eqref{k}, we have for $(s,x)\in(t,T]\times\brd$,
\begin{align*}
	&\frac{\dd u^{t,\mu,\tilde{\mu}}}{\dd s}(s,x)=\int_{\brd} \frac{\dd k^{t,\mu}}{\dd s}(s,x,y)\tilde{\mu}(y)dy\\
	&=\int_{\brd} \sum_{i,j=1}^d\frac{\dd^2}{\dd x_i\dd x_j}\Big[a_{ij}(x,m^{t,\mu}(s,\cdot))k^{t,\mu}(s,x,y)\tilde{\mu}(y)\\
	&\qquad\qquad+\int_{\brd}\frac{\dd a_{ij}}{\dd m}(x,m^{t,\mu}(s,\cdot))(\xi)k^{t,\mu}(s,\xi,y)\tilde{\mu}(y)d\xi \cdot m^{t,\mu}(s,x)\Big] dy\\
	&\qquad-\sum_{i=1}^d\frac{\dd}{\dd x_i}\Big[b_i(x,m^{t,\mu}(s,\cdot))k^{t,\mu}(s,x,y)\tilde{\mu}(y)\\
	&\qquad\qquad+\int_{\brd}\frac{\dd b_i}{\dd m}(x,m^{t,\mu}(s,\cdot))(\xi)k^{t,\mu}(s,\xi,y)\tilde{\mu}(y)d\xi \cdot m^{t,\mu}(s,x)\Big] dy\\
	&=\sum_{i,j=1}^d\frac{\dd^2}{\dd x_i\dd x_j}\Big[a_{ij}(x,m^{t,\mu}(s,\cdot))\int_{\brd} k^{t,\mu}(s,x,y)\tilde{\mu}(y)dy\\
	&\qquad\qquad+\int_{\brd}\frac{\dd a_{ij}}{\dd m}(x,m^{t,\mu}(s,\cdot))(\xi)\int_{\brd} k^{t,\mu}(s,\xi,y)\tilde{\mu}(y)dyd\xi \cdot m^{t,\mu}(s,x)\Big]\\
	&\quad-\sum_{i=1}^d\frac{\dd}{\dd x_i}\Big[b_i(x,m^{t,\mu}(s,\cdot))\int_{\brd} k^{t,\mu}(s,x,y)\tilde{\mu}(y)dy\\
	&\qquad\qquad+\int_{\brd}\frac{\dd b_i}{\dd m}(x,m^{t,\mu}(s,\cdot))(\xi)\int_{\brd} k^{t,\mu}(s,\xi,y)\tilde{\mu}(y)dyd\xi \cdot m^{t,\mu}(s,x)\Big]\\
	&=\sum_{i,j=1}^d\frac{\dd^2}{\dd x_i\dd x_j}\Big[a_{ij}(x,m^{t,\mu}(s,\cdot))u^{t,\mu,\tilde{\mu}}(s,x)\\
	&\qquad\qquad+\int_{\brd}\frac{\dd a_{ij}}{\dd m}(x,m^{t,\mu}(s,\cdot))(\xi)u^{t,\mu,\tilde{\mu}}(s,\xi)d\xi \cdot m^{t,\mu}(s,x)\Big]\\
	&\quad-\sum_{i=1}^d\frac{\dd}{\dd x_i}\Big[b_i(x,m^{t,\mu}(s,\cdot))u^{t,\mu,\tilde{\mu}}(s,x)\\
	&\qquad\qquad+\int_{\brd}\frac{\dd b_i}{\dd m}(x,m^{t,\mu}(s,\cdot))(\xi)u^{t,\mu,\tilde{\mu}}(s,\xi)d\xi \cdot m^{t,\mu}(s,x)\Big].
\end{align*}
And from the definition of the function $\delta$, we have
\begin{align*}
	u^{t,\mu,\tilde{\mu}}(t,x)=\int_{\brd}\delta(x-y)\tilde{\mu}(y)dy=\tilde{\mu}(x),\quad x\in\brd.
\end{align*}
The proof is complete.

\subsection{Proof of Lemma~\ref{lem9}}\label{pf_lem9}
We continue to restrict ourselves within the one-dimensional case $d=1$. For notational convenience, we write $(f,g)$ instead of $(f^{t,\mu},g^{t,\mu})$ when there is no ambiguity. We first prove \eqref{lem9_1}. For notational convenience, we set for $(s,x,y)\in[t,T]\times\br\times\br$,
\begin{align*}
	l(s,x,y):=&\frac{\dd^2}{\dd x^2}\Big[\int_{\br}\frac{\dd a}{\dd m}(m^{t,\mu}(s,\cdot))(\xi)f(s,\xi,y)d\xi \cdot m^{t,\mu}(s,x)\Big]\\
	&-\frac{\dd^2}{\dd x^2}\Big[\int_{\br}\frac{\dd b}{\dd m}(m^{t,\mu}(s,\cdot))(\xi)f(s,\xi,y)d\xi \cdot m^{t,\mu}(s,x)\Big].
\end{align*}
We have already proved that
\begin{align}\label{lem9_1.5}
	\|l\|_{L^2([t,T];L^2(\br^2))}\le C(\gamma,L,T,\|\mu\|_{W^{1,2}(\br)}).
\end{align}
Without loss of generality, we temporarily suppose that $g$ is smooth enough. Multiplying both sides of equation \eqref{g} with $g$ and integrating with respect to $x$ and $y$, from Assumptions (H1) and (H2), the average inequality and Lemma~\ref{thm1}, we have for $s\in(t,T]$,
\begin{align*}
	&\frac{d}{ds}\|g(s,\cdot,\cdot)\|^2_{L^2(\br^2)}+\|g_x(s,\cdot,\cdot)\|^2_{L^2(\br^2)}\\
	&\qquad\le C(\gamma,L,T)\big[(1+\|m^{t,\mu}(s,\cdot)\|^2_{W^{1,2}(\br)})\|g(s,\cdot,\cdot)\|^2_{L^2(\br^2)}+\|l(s,\cdot,\cdot)\|^2_{L^2(\br^2)}\big]\\
	&\qquad\le C(\gamma,L,T,\|\mu\|_{W^{1,2}(\br)})\big[\|g(s,\cdot,\cdot)\|^2_{L^2(\br^2)}+\|l(s,\cdot,\cdot)\|^2_{L^2(\br^2)}\big].
\end{align*}
Therefore, from Gronwall's inequality and estimate \eqref{lem9_1.5}, we have
\begin{equation}\label{lem9_2}
	\begin{split}
		\sup_{s\in[t,T]}\|g(s,\cdot,\cdot)\|^2_{L^{2}(\br^2)}+\int_t^T\|g_{x}(s,\cdot,\cdot)\|^2_{L^2(\br^2)}ds\le C\big(\gamma,L,T,\|\mu\|_{W^{1,2}(\br)}\big).
	\end{split}
\end{equation}
Multiplying both sides of equation \eqref{g} with $a(s,m^{t,\mu}(s,\cdot))^{-1}g_s(s,x,y)$ and integrating with respect to $x$ and $y$, from Assumptions (H1) and (H2), the average inequality and estimate \eqref{lem9_2}, we have for $s\in(t,T]$,
\begin{align*}
	&\frac{d}{ds}\|g_x(s,\cdot,\cdot)\|^2_{L^2(\br^2)}+\|g_s(s,\cdot,\cdot)\|^2_{L^2(\br^2)}\\
	&\le C(\gamma,L,T)\big[\|g_x(s,\cdot,\cdot)\|^2_{L^2(\br^2)}+(1+\|m^{t,\mu}(s,\cdot)\|^2_{W^{2,2}(\br)})\|g(s,\cdot,\cdot)\|^2_{L^2(\br^2)}\\
	&\qquad+\|l(s,\cdot,\cdot)\|^2_{L^2(\br^2)}\big]\\
	&\le C(\gamma,L,T,\|\mu\|_{W^{1,2}(\br)})\big[\|g_x(s,\cdot,\cdot)\|^2_{L^2(\br^2)}+1+\|m^{t,\mu}(s,\cdot)\|^2_{W^{2,2}(\br)}\\
	&\qquad+\|l(s,\cdot,\cdot)\|^2_{L^2(\br^2)}\big].
\end{align*}
Therefore, from Gronwall's inequality, Lemma~\ref{thm1} and estimate \eqref{lem9_1.5}, we have
\begin{equation}\label{lem9_3}
	\begin{split}
		\sup_{s\in[t,T]}\|g_x(s,\cdot,\cdot)\|^2_{L^{2}(\br^2)}+\int_t^T\|g_{s}(s,\cdot,\cdot)\|^2_{L^2(\br^2)}ds\le C\big(\gamma,L,T,\|\mu\|_{W^{1,2}(\br)}\big).
	\end{split}
\end{equation}

We now prove \eqref{lem11_1}. We set for $(s,x,y)\in[t,T]\times\br\times\br$,
\begin{align*}
	&\Delta f:=f^{t,\mu'}-f^{t,\mu},\  \Delta g :=g^{t,\mu'}-g^{t,\mu},\quad \Delta m :=m^{t,\mu'}-m^{t,\mu},\  \Delta \mu:=\mu'-\mu,\\
	&\Delta a(s,x):=a(x,m^{t,\mu'}(s,\cdot))-a(x,m^{t,\mu}(s,\cdot)),\\
	&\Delta b(s,x):=b(x,m^{t,\mu'}(s,\cdot))-b(x,m^{t,\mu}(s,\cdot)).
\end{align*}
From equation \eqref{g}, we know that $\Delta g(t,x,y)=0$ for $(x,y)\in\br^2$, and for $(s,x,y)\in(t,T]\times\br\times\br$,
\begin{align}
		&\frac{\dd \Delta g}{\dd s}(s,x,y)-\frac{\dd^2}{\dd x^2}\Big[a(x,m^{t,\mu'}(s,\cdot))\Delta g(s,x,y)+\Delta a(s,x)g^{t,\mu}(s,x,y)\label{lem11_2}\\
		&\qquad+\int_{\br}\frac{\dd a}{\dd m}(x,m^{t,\mu'}(s,\cdot))(\xi)\Delta g(s,\xi,y)d\xi \cdot m^{t,\mu'}(s,x)\notag\\
		&\qquad+\int_{\br}\frac{\dd \Delta a}{\dd m} (s,x)(\xi)g^{t,\mu}(s,\xi,y)d\xi \cdot m^{t,\mu'}(s,x)\notag\\
		&\qquad+\int_{\br}\frac{\dd a}{\dd m}(x,m^{t,\mu}(s,\cdot))(\xi)g^{t,\mu}(s,\xi,y)d\xi \cdot \Delta m(s,x) \Big]\notag\\
		&\qquad+\frac{\dd}{\dd x}\Big[b(x,m^{t,\mu'}(s,\cdot))\Delta g(s,x,y)+\Delta b(s,x)g^{t,\mu}(s,x,y)\notag\\
		&\qquad+\int_{\br}\frac{\dd  b}{\dd m}(x,m^{t,\mu'}(s,\cdot))(\xi)\Delta g(s,\xi,y)d\xi \cdot m^{t,\mu'}(s,x)\notag\\
		&\qquad+\int_{\br}\frac{\dd \Delta b}{\dd m}(s,x)(\xi)g^{t,\mu}(s,\xi,y)d\xi \cdot m^{t,\mu'}(s,x)\notag\\
		&\qquad+\int_{\br}\frac{\dd b}{\dd m}(x,m^{t,\mu}(s,\cdot))(\xi)g^{t,\mu}(s,\xi,y)d\xi \cdot \Delta m(s,x) \Big]=\Delta l(s,x,y),\notag
\end{align}
where for $(s,x,y)\in(t,T]\times\br\times\br$,
\begin{align*}
	\Delta l(s,x,y):=&\frac{\dd^2}{\dd x^2}\Big[\int_{\br} \frac{\dd a}{\dd m}(x,m^{t,\mu'}(s,\cdot))(\xi)\Delta f(s,\xi,y) d\xi \cdot m^{t,\mu'}(s,x)\\
	&\qquad+\int_{\br} \frac{\dd \Delta a}{\dd m}(s,x)(\xi) f^{t,\mu}(s,\xi,y) d\xi\cdot m^{t,\mu'}(s,x)\\
	&\qquad+\int_{\br} \frac{\dd a}{\dd m}(x,m^{t,\mu}(s,\cdot))(\xi) f^{t,\mu}(s,\xi,y) d\xi \cdot \Delta m(s,x)\Big]\\
	&+\frac{\dd}{\dd x}\Big[\int_{\br} \frac{\dd b}{\dd m}(x,m^{t,\mu'}(s,\cdot))(\xi)\Delta f(s,\xi,y) d\xi \cdot m^{t,\mu'}(s,x)\\
	&\qquad+\int_{\br} \frac{\dd \Delta b}{\dd m}(s,x)(\xi) f^{t,\mu}(s,\xi,y) d\xi \cdot m^{t,\mu'}(s,x)\\
	&\qquad+\int_{\br} \frac{\dd b}{\dd m}(x,m^{t,\mu}(s,\cdot))(\xi) f^{t,\mu}(s,\xi,y) d\xi \cdot \Delta m(s,x)\Big].
\end{align*}
From Assumptions (H2) and (H3), Proposition~\ref{lemma6} and Lemma~\ref{thm1}, we have
\begin{equation}\label{lem11_5}
	\begin{split}
		&\|\Delta l\|^2_{L^2([t,T];L^2(\br^2))}\\
		&\le C(\gamma,L,T,\|\mu\|_{W^{1,2}(\br)})\int_t^T\Big[\|\Delta \mu\|^2_{L^2(\br)}+\|\Delta m(s,\cdot)\|^2_{L^2(\br)}\Big]\\
		&\qquad\times \|m^{t,\mu'}(s,\cdot)\|^2_{W^{2,2}(\br)}+\|\Delta m(s,\cdot)\|_{W^{2,2}(\br)}^2ds\\
		&\le C(\gamma,L,T,\|\mu\|_{W^{1,2}(\br)})\Big[\|\Delta \mu\|^2_{L^2(\br)}\int_t^T\|m^{t,\mu'}(s,\cdot)\|^2_{W^{2,2}(\br)}ds\\
		&\qquad+\int_t^T\|\Delta m(s,\cdot)\|_{W^{2,2}(\br)}^2ds\Big]\\
		&\le C(\gamma,L,T,\|\mu\|_{W^{1,2}(\br)},\|\mu'\|_{W^{1,2}(\br)})\|\Delta\mu\|^2_{W^{1,2}(\br)}.
	\end{split}
\end{equation}
We temporarily suppose that $\Delta g$ is smooth enough. Multiplying both sides of equation \eqref{lem11_2} with $\Delta g$ and integrating them with respect to $x$ and $y$, from Assumptions (H1) and (H2), the average inequality, Lemma~\ref{thm1}, Proposition~\ref{lemma6} and estimates \eqref{lem9_2} and \eqref{lem9_3}, we have for $t<s\le T$,
\begin{align*}
	&\frac{d}{ds}\|\Delta g(s,\cdot,\cdot)\|^2_{L^2(\br^2)}+\|\Delta g_x(s,\cdot,\cdot)\|^2_{L^2(\br^2)}\\
	&\le C(\gamma,L)\Big[ \|\Delta l(s,\cdot,\cdot)\|^2_{L^2(\br)}+\big(1+\|m^{t,\mu'}(s,\cdot)\|^2_{W^{1,2}(\br)}\big)\\
	&\qquad\times\big(\|\Delta g(s,\cdot,\cdot)\|^2_{L^2(\br^2)}+\|g^{t,\mu}(s,\cdot,\cdot)\|^2_{W^{1,2}(\br^2)}\|\Delta m(s,\cdot)\|^2_{W^{1,2}(\br)}\big)\Big]\\
	&\le C(\gamma,L,T,\|\mu\|_{W^{1,2}(\br)},\|\mu'\|_{W^{1,2}(\br)})\Big[\|\Delta\mu\|^2_{W^{1,2}(\br)}+\|\Delta l(s,\cdot,\cdot)\|^2_{L^2(\br)}\\
	&\qquad+\|\Delta g(s,\cdot,\cdot)\|^2_{L^2(\br^2)}\Big].
\end{align*}
Therefore, from Gronwall's inequality and estimate \eqref{lem11_5}, we have
\begin{align*}
	\sup_{t\le s\le T}\|\Delta g(s,\cdot,\cdot)\|^2_{L^2(\br^2)}&\le C\big(\gamma,L,T,\|\mu\|_{W^{1,2}(\br)},\|\mu'\|_{W^{1,2}(\br)}\big)\|\Delta\mu\|^2_{W^{1,2}(\br)}.
\end{align*}

\section{Proof of Statements in Section 4}\label{pf_4}
\subsection{Proof of Lemma~\ref{lem2.3}}\label{pf_lem2.3}
We continue to restrict ourselves within the one-dimensional case $d=1$. For $h\in(0,1]$ and $(s,x)\in[t,T]\times\br$, we set
\begin{align*}
	&\rho_s^{t,x,\mu}(\tilde{\mu},h):=Y^{t,x,\mu}_s(\tilde{\mu},h)-Y^{t,x,\mu}_s(\tilde{\mu});\\
	&u_h^{t,\mu}(\tilde{\mu})(s,x):=\tm_h^{t,\mu}(\tilde{\mu})(s,x)-\tm^{t,\mu}(\tilde{\mu})(s,x).
\end{align*}
From Lemma~\ref{lem5}, we have the estimate
\begin{equation}\label{lem2.3_2}
	\sup_{t\le s\le T}\|u_h^{t,\mu}(\tilde{\mu})(s,\cdot)\|_{L^2(\br)}\le C\big(\gamma,L,T,\|\mu\|_{W^{1,2}(\br)},\|\tilde{\mu}\|_{L^2(\br)}\big)h.
\end{equation}
From SDEs \eqref{sde_Y} and \eqref{sde_Y'}, we know that $\rho^{t,x,\mu}(\tilde{\mu},h)$ satisfies the following SDE:
\begin{align*}
	&\rho^{t,x,\mu}_s(\tilde{\mu},h)\\
	&=\int_t^s\int_0^1\Big[b_x(X^{t,x,\mu}_r+\lambda hY^{t,x,\mu}_r(\tilde{\mu},h),m^{t,\mu+h\tilde{\mu}}(r,\cdot))\cdot \rho^{t,x,\mu}_r(\tilde{\mu},h)\\
	&\quad +[b_x(X^{t,x,\mu}_r+\lambda hY^{t,x,\mu}_r(\tilde{\mu},h),m^{t,\mu+h\tilde{\mu}}(r,\cdot))-b_x(X^{t,x,\mu}_r,m^{t,\mu}(r,\cdot))]\cdot Y^{t,x,\mu}_r(\tilde{\mu}) \\
	&\quad +\int_{\br} \frac{\dd b}{\dd m}(X^{t,x,\mu}_r,m^{t,\mu}(r,\cdot)+\lambda h\tm_h^{t,\mu}(\tilde{\mu})(r,\cdot))(\xi)\cdot u_h^{t,\mu}(\tilde{\mu})(r,\xi) d\xi \\
	&\quad +\int_{\br}\big[ \frac{\dd b}{\dd m}(X^{t,x,\mu}_r,m^{t,\mu}(r,\cdot)+\lambda h\tm_h^{t,\mu}(\tilde{\mu})(r,\cdot))(\xi)\\
	&\qquad\qquad\quad- \frac{\dd b}{\dd m}(X^{t,x,\mu}_r,m^{t,\mu}(r,\cdot))(\xi)\big]\cdot \tm^{t,\mu}(\tilde{\mu})(r,\xi) d\xi\Big]d\lambda dr\\
	&\quad+\int_t^s\int_0^1\Big[\sigma_x(X^{t,x,\mu}_r+\lambda hY^{t,x,\mu}_r(\tilde{\mu},h),m^{t,\mu+h\tilde{\mu}}(r,\cdot))\cdot \rho^{t,x,\mu}_r(\tilde{\mu},h)\\
	&\quad + [\sigma_x(X^{t,x,\mu}_r+\lambda hY^{t,x,\mu}_r(\tilde{\mu},h),m^{t,\mu+h\tilde{\mu}}(r,\cdot))-\sigma_x(X^{t,x,\mu}_r,m^{t,\mu}(r,\cdot))]\cdot Y^{t,x,\mu}_r(\tilde{\mu}) \\
	&\quad +\int_{\br} \frac{\dd \sigma}{\dd m}(X^{t,x,\mu}_r,m^{t,\mu}(r,\cdot)+\lambda h\tm_h^{t,\mu}(\tilde{\mu})(r,\cdot))(\xi)\cdot u_h^{t,\mu}(\tilde{\mu})(r,\xi) d\xi \\
	&\quad +\int_{\br}\big[ \frac{\dd \sigma}{\dd m}(X^{t,x,\mu}_r,m^{t,\mu}(r,\cdot)+\lambda h\tm_h^{t,\mu}(\tilde{\mu})(r,\cdot))(\xi)\\
	&\qquad\qquad\quad- \frac{\dd \sigma}{\dd m}(X^{t,x,\mu}_r,m^{t,\mu}(r,\cdot))(\xi)\big]\cdot \tm^{t,\mu}(\tilde{\mu})(r,\xi) d\xi\Big]d\lambda dB_r,
\end{align*}
for $s\in[t,T]$. From standard arguments of SDEs, Cauchy's inequality, Assumption (H2), estimate \eqref{lem3_1}, Lemma~\ref{thm4}, and estimates \eqref{estimate_Y}, \eqref{estimate_Y'} and \eqref{lem2.3_2},  we have
\begin{align*}
	&\e[\sup_{t\le s\le T}|\rho^{t,x,\mu}_s(\tilde{\mu},h)|^2]\\
	&\le C(L,T)\Big[\e[\sup_{t\le s\le T}|Y^{t,x,\mu}_s(\tilde{\mu})|^4]^\frac{1}{2}\e[\sup_{t\le s\le T}|Y^{t,x,\mu}_s(\tilde{\mu},h)|^4]^\frac{1}{2}h^2\\
	&\qquad+\e[\sup_{t\le s\le T}|Y^{t,x,\mu}_s(\tilde{\mu})|^2]\sup_{t\le s\le T}\|\tm_h^{t,\mu}(s,\cdot)\|_{L^2(\br)}^2h^2\\
	&\qquad+\sup_{t\le s\le T}\|\tm_h^{t,\mu}(s,\cdot)\|_{L^2(\br)}^2\sup_{t\le s\le T}\|\tm^{t,\mu}(s,\cdot)\|_{L^2(\br)}^2h^2+\sup_{t\le s\le T}\|u_h^{t,\mu}(s,\cdot)\|_{L^2(\br)}\Big]\\
	&\le C(\gamma,L,T,\|\mu\|_{W^{1,2}(\br)},\|\tilde{\mu}\|_{L^2(\br)})h^2.
\end{align*}

\subsection{Proof of Proposition~\ref{lem2.5}}\label{pf_lem2.5}
We continue to restrict ourselves within the one-dimensional case $d=1$. With standard arguments of SDEs, we have for any $y\in\br$,
\begin{align*}
	\e[\sup_{t\le s\le T}|U^{t,x,\mu}_s(y)|^2]&\le C(L,T)\e\Big[\int_t^T\Big|\int_{\br} \frac{\dd b}{\dd m}(X^{t,x,\mu}_s,m^{t,\mu}(s,\cdot))(\xi)k^{t,\mu}(s,\xi,y) d\xi\Big|^2\\
	&\quad\qquad\qquad +\Big|\int_{\br} \frac{\dd \sigma}{\dd m}(X^{t,x,\mu}_s,m^{t,\mu}(s,\cdot))(\xi)k^{t,\mu}(s,\xi,y) d\xi\Big|^2ds\Big].
\end{align*}
So from Assumption (H2) and Proposition~\ref{lem12}, we have \eqref{lem2.5_1}. Applying It\^o's lemma for $\int_\br |U^{t,x,\mu}(y)|^2dy$, with standard arguments of SDEs, we have
\begin{align*}
	&\e\Big[\sup_{t\le s\le T}\Big|\int_{\br}|U^{t,x,\mu}_s(y)|^2dy\Big|^p\Big]\\
	&\le C(p,L,T)\e\Big[\int_t^T \Big(\int_{\br} |\int_{\br} \frac{\dd b}{\dd m}(X^{t,x,\mu}_s,m^{t,\mu}(s,\cdot))(\xi)k^{t,\mu}(s,\xi,y) d\xi|^2dy\\
	&\quad\qquad\qquad\qquad\qquad+\int_{\br}|\int_{\br} \frac{\dd \sigma}{\dd m}(X^{t,x,\mu}_s,m^{t,\mu}(s,\cdot))(\xi)k^{t,\mu}(s,\xi,y) d\xi|^2 dy\Big)^p ds\Big].
\end{align*}
So from Assumption (H2) and Proposition~\ref{lem12}, we have \eqref{lem2.5_1'}. Now we prove \eqref{lem2.5_2}. We set for $(s,x,y)\in[t,T]\times\br\times\br$,
\begin{align*}
	&\Delta x:=x'-x,\quad \Delta \mu:=\mu'-\mu,\quad \Delta m(s,x):=m^{t,\mu'}(s,x)-m^{t,\mu}(s,x);\\
	&\Delta X_s:=X_s^{t,x',\mu'}-X_s^{t,x,\mu},\quad \Delta U_s(y):=U^{t,x',\mu'}_s(y)-U^{t,x,\mu}_s(y);\\
	&\Delta B_s(\xi):=\frac{\dd b}{\dd m}(X^{t,x',\mu'}_s,m^{t,\mu'}(s,\cdot))(\xi)-\frac{\dd b}{\dd m}(X^{t,x,\mu}_s,m^{t,\mu}(s,\cdot))(\xi);\\
	&\Delta \Sigma_s(\xi):=\frac{\dd \sigma}{\dd m}(X^{t,x',\mu'}_s,m^{t,\mu'}(s,\cdot))(\xi)-\frac{\dd \sigma}{\dd m}(X^{t,x,\mu}_s,m^{t,\mu}(s,\cdot))(\xi).
\end{align*}
With standard arguments of SDEs and Assumption (H2), we have
\begin{equation}\label{lem2.5_3}
	\begin{split}
		&\e[\int_{\br}\sup_{t\le s\le T}|\Delta U_s(y)|^2dy]\\
		&\le C(L,T)\e\Big[\int_t^T (\int_{\br}|U_s^{t,x,\mu}|^2dy)(|\Delta X_s|^2+\|\Delta m(s,\cdot)\|^2_{L^2(\br)}) \\
		&\quad +\int_{\br} \Big|\int_{\br} \frac{\dd b}{\dd m}(X_s^{t,x',\mu'},m^{t,\mu'}(s,\cdot))(\xi)[k^{t,\mu'}(s,\xi,y)-k^{t,\mu}(s,\xi,y)] d\xi\Big|^2 dy\\
		&\quad +\int_{\br} \Big|\int_{\br} \frac{\dd \sigma}{\dd m}(X_s^{t,x',\mu'},m^{t,\mu'}(s,\cdot))(\xi)[k^{t,\mu'}(s,\xi,y)-k^{t,\mu}(s,\xi,y)] d\xi\Big|^2 dy\\
		&\quad + \int_{\br} \Big|\int_{\br} \Delta B_s(\xi) k^{t,\mu}(s,\xi,y) d\xi\Big|^2 dy \\
		&\quad+ \int_{\br} \Big|\int_{\br} \Delta\Sigma_s(\xi) k^{t,\mu}(s,\xi,y) d\xi\Big|^2 dyds\Big].
	\end{split}
\end{equation}
From Cauchy's inequality, Proposition~\ref{lemma6} and estimates \eqref{deltaX}, \eqref{lem2.5_1} and \eqref{lem2.5_1'}, we have
\begin{equation}\label{lem2.5_4}
	\begin{split}
		&\e\Big[\int_t^T (\int_{\br}|U_s^{t,x,\mu}|^2dy)(|\Delta X_s|^2+\|\Delta m(s,\cdot)\|^2_{L^2(\br)})ds\Big]\\
		&\le C(T)\Big(\e[\sup_{t\le s\le T}|\int_{\br}|U_s^{t,x,\mu}|^2dy|^{2}]^{\frac{1}{2}}\e[\sup_{t\le s\le T}|\Delta X_s|^{4}]^{\frac{1}{2}}\\
		&\qquad\qquad+\e[\sup_{t\le s\le T}\int_{\br}|U_s^{t,x,\mu}|^2dy]\sup_{t\le s\le T}\|\Delta m(s,\cdot)\|^2_{L^2(\br)}\Big)\\
		&\le C(\gamma,L,T,\|\mu\|_{W^{1,2}(\br)})(|\Delta x|^2+\|\Delta\mu\|^2_{L^2(\br)}).
	\end{split}
\end{equation}
From Assumption (H2) and Proposition~\ref{lem12}, we have for any $s\in[t,T]$ and $c=b,\sigma$,
\begin{equation}\label{lem2.5_5}
	\begin{split}
		&\int_{\br} \Big|\int_{\br} \frac{\dd c}{\dd m}(X_s^{t,x',\mu'},m^{t,\mu'}(s,\cdot))(\xi)[k^{t,\mu'}(s,\xi,y)-k^{t,\mu}(s,\xi,y)] d\xi\Big|^2 dy\\
		&\qquad\le C(\gamma,L,T,\|\mu\|_{W^{1,2}(\br)},\|\mu'\|_{W^{1,2}(\br)})\|\Delta\mu\|^2_{W^{1,2}(\br)}.
	\end{split}
\end{equation}
From Assumption (H2), estimate \eqref{deltaX} and Propositions~\ref{lemma6} and \ref{lem12}, we have for any $s\in[t,T]$,
\begin{equation}\label{lem2.5_6}
	\begin{split}
		&\e\Big[ \int_{\br} \Big|\int_{\br} \Delta B_s(\xi) k^{t,\mu}(s,\xi,y) d\xi\Big|^2 dy+ \int_{\br} \Big|\int_{\br} \Delta\Sigma_s(\xi) k^{t,\mu}(s,\xi,y) d\xi\Big|^2 dy\Big]\\
		&\qquad\le C(\gamma,L,T,\|\mu\|_{W^{1,2}(\br)})\e\Big[\int_{\br} |\Delta B_s(\xi)|^2+|\Delta \Sigma_s(\xi)|^2d\xi\Big]\\
		&\qquad\le C(\gamma,L,T,\|\mu\|_{W^{1,2}(\br)})(\e[\sup_{t\le s\le T}|\Delta X_s|^2]+\sup_{t\le s\le T}\|\Delta m(s,\cdot)\|^2_{L^2(\br)})\\
		&\qquad\le C(\gamma,L,T,\|\mu\|_{W^{1,2}(\br)})(|\Delta x|^2+\|\Delta\mu\|^2_{L^2(\br)}).
	\end{split}
\end{equation}
Plugging \eqref{lem2.5_4}-\eqref{lem2.5_6} into \eqref{lem2.5_3}, we have  \eqref{lem2.5_2}.

\section{Proof of Proposition~\ref{thm3.1}}\label{pf_thm3.1}
We continue to restrict ourselves within the one-dimensional case $d=1$. Equalities \eqref{V_x}-\eqref{V_mu} are direct consequences of Sections~\ref{PDE2} and \ref{SDE}. The boundedness of $(V,\frac{\dd V}{\dd x},\frac{\dd^2 V}{\dd x^2})$ follows from Assumption (H4) and Proposition~\ref{lem2.0}. From Assumption (H4) and Propositions~\ref{lem12} and \ref{lem2.5}, we have
\begin{align*}
	&\|\frac{\dd V}{\dd \mu}(t,x,\mu)(\cdot)\|_{L^2(\br)}^2\\
	&\qquad\le C(L)\e\Big[\int_{\br} |U_T^{t,x,\mu}(y)|^2dy+\int_{\br}\Big|\int_{\br}\frac{\dd \Phi}{\dd m}(X_T^{t,x,\mu},m_T^{t,\mu})(\xi)k^{t,\mu}(T,\xi,y)d\xi\Big|^2 dy\Big]\\
	&\qquad\le C(\gamma,L,T,\|\mu\|_{W^{1,2}(\br)}).
\end{align*}
The Lipschitz-continuity of $(V,\frac{\dd V}{\dd x},\frac{\dd^2 V}{\dd x^2})$ with respect to $(x,\mu)$ follows from Assumption (H4), Propositions~\ref{lemma6} and \ref{lem2.0} and estimate \eqref{deltaX}. And from Assumption (H4), Cauchy's inequality, estimate \eqref{deltaX} and Propositions~\ref{lemma6}, \ref{lem12} and \ref{lem2.5}, we have
\begin{align*}
	&\|\frac{\dd V}{\dd \mu}(t,x',\mu')(\cdot)-\frac{\dd V}{\dd \mu}(t,x,\mu)(\cdot)\|_{L^2(\br)}^2\\
	&\le C(L)\Big\{\e\Big[\int_{\br} |U_T^{t,x',\mu'}-U_T^{t,x,\mu}|^2 dy\Big]+\e[|X_T^{t,x',\mu'}-X_T^{t,x,\mu}|^4]^{\frac{1}{2}}\e\Big[\Big|\int_{\br} |U_T^{t,x,\mu}|^2 dy\Big|^2\Big]^{\frac{1}{2}}\\
	&\quad+\|m_T^{t,\mu'}-m_T^{t,\mu}\|^2\e\Big[\int_{\br} |U_T^{t,x,\mu}|^2 dy\Big]\\
	&\quad+\e\Big[\int_{\br} \Big|\int_{\br}\frac{\dd \Phi}{\dd m}(X_T^{t,x',\mu'},m_T^{t,\mu'})(\xi)[k^{t,\mu'}(T,\xi,y)-k^{t,\mu}(T,\xi,y)]d\xi\Big|^2 dy\\
	&\quad+\int_{\br} \Big|\int_{\br}[\frac{\dd \Phi}{\dd m}(X_T^{t,x',\mu'},m_T^{t,\mu'})(\xi)-\frac{\dd \Phi}{\dd m}(X_T^{t,x,\mu},m_T^{t,\mu})(\xi)]k^{t,\mu}(T,\xi,y)d\xi\Big|^2 dy\Big]\Big\}\\
	&\le C(\gamma,L,T,\|\mu\|_{W^{1,2}(\br)},\|\mu'\|_{W^{1,2}(\br)})(|x'-x|^2+\|\mu'-\mu\|^2_{W^{1,2}(\br)}).
\end{align*}

It remains to prove the $\frac{1}{2}$-H\"older-continuity with respect to $t$. Using the time-shifted Brownian motion $B_s^t:=B_{t+s}-B_t$, $s\geq 0$, we see that $X^{t,x,\mu}$ solve the following SDE
\begin{equation*}
	X_{s+t}^{t,x,\mu}=x+\int_0^s b(X_{r+t}^{t,x,\mu},m_{r+t}^{t,\mu})dr+\int_0^s \sigma(X_{r+t}^{t,x,\mu},m_{r+t}^{t,\mu})dB_r^t,\quad s\in[0,T-t].
\end{equation*}
Consequently, $(m^{t,\mu}_{\cdot+t},X^{t,x,\mu}_{\cdot+t})$ and $(m^{0,\mu},X^{0,x,\mu})$ are solutions of the same system of PDE-SDE, only driven by different Brownian motions, $B^t$ and $B$, respectively. It follows that the laws of $(m^{t,\mu}_{\cdot+t},X^{t,x,\mu}_{\cdot+t})$ and $(m^{0,\mu},X^{0,x,\mu})$ coincide, and hence,
\begin{equation}\label{thm3.1_4}
	V(t,x,\mu)=\e[\Phi(X_T^{t,x,\mu},m^{t,\mu}(T,\cdot))]=\e[\Phi(X_{T-t}^{0,x,\mu},m^{0,\mu}(T-t,\cdot))].
\end{equation}
For two different initial times $0\le t\le t' \le T$,
\begin{align*}
	X_{T-t'}^{0,x,\mu}-X_{T-t}^{0,x,\mu}=-\int_{T-t'}^{T-t}b(X_{s}^{0,x,\mu},m_{s}^{0,x,\mu})ds-\int_{T-t'}^{T-t}\sigma(X_{s}^{0,x,\mu},m_{s}^{0,x,\mu})dB_s.
\end{align*}
With Assumption (H2) and standard argument of SDEs, we have for any $p\geq 2$,
\begin{align}\label{thm3.1_1}
	\e[|X_{T-t'}^{0,x,\mu}-X_{T-t}^{0,x,\mu}|^p]\le C(p,L,T)|t'-t|^{\frac{p}{2}}.
\end{align}
From Assumption (H4), Lemma~\ref{thm1}, \eqref{thm3.1_4} and \eqref{thm3.1_1}, we prove the $\frac{1}{2}$-H\"older-continuity of $V$ with respect to $t$. The proof of the remaining estimate of the $\frac{1}{2}$-H\"older-continuity with respect to $t$ for the derivatives of $V$ is carried out by using the same kind of argument. From Proposition~\ref{lem2.0}, we have for any $p\geq 2$ and $0\le t\le t' \le T$,
\begin{align}\label{thm3.1_3}
	\e[|\dd_xX_T^{t',x,\mu}-\dd_xX_T^{t,x,\mu}|^p+|\dd_x^2X_T^{t',x,\mu}-\dd_x^2X_T^{t,x,\mu}|^p]\le C(L,T)|t'-t|^{\frac{p}{2}}.
\end{align}
The $\frac{1}{2}$-H\"older continuity of $(V,\frac{\dd V}{\dd x},\frac{\dd^2 V}{\dd x^2})$ with respect to $t$ then follows from Assumption (H4), estimates \eqref{thm3.1_1} and \eqref{thm3.1_3}, Lemma~\ref{thm1},  and Proposition~\ref{lem2.0}. From equation \eqref{sde_U} and Propositions~\ref{lem12} and \ref{lem2.5}, we have for $0\le t\le t' \le T$,
\begin{equation}\label{thm3.1_2}
	\begin{split}
		&\e\Big[\int_{\br}|U_T^{t',x,\mu}(y)-U_T^{t,x,\mu}(y)|^2dy\Big]\\
		&\le C(L,T)\e\Big[\int_{T-t'}^{T-t}\int_{\br} |U_s^{0,x,\mu}(y)|^2\\
		&\qquad+\Big|\int_{\br} \frac{\dd \sigma}{\dd m}(X_s^{0,x,\mu},m^{0,\mu}(s,\cdot))(\xi)k^{0,\mu}(s,\xi,y)d\xi\Big|^2\\
		&\qquad+\Big|\int_{\br} \frac{\dd b}{\dd m}(X_s^{0,x,\mu},m^{0,\mu}(s,\cdot))(\xi)k^{0,\mu}(s,\xi,y)d\xi\Big|^2 dyds\Big]\\
		&\le C(\gamma,L,T,\|\mu\|_{W^{1,2}(\br)})|t'-t|.
	\end{split}
\end{equation}
Therefore, from Assumption (H4), Cauchy's inequality, Lemma~\ref{thm1}, Propositions~\ref{lem12} and \ref{lem2.5},  and estimates \eqref{thm3.1_1} and \eqref{thm3.1_2}, we have
\begin{align*}
	&\|\frac{\dd V}{\dd \mu}(t',x,\mu)(\cdot)-\frac{\dd V}{\dd \mu}(t,x,\mu)(\cdot)\|_{L^2(\br)}^2\\
	&\le C(L)\Big\{\e\Big[\int_{\br} |U_T^{t',x,\mu}-U_T^{t,x,\mu}|^2 dy\Big]+\e[|X_T^{t',x,\mu}-X_T^{t,x,\mu}|^4]^{\frac{1}{2}}\e\Big[\Big|\int_{\br} |U_T^{t,x,\mu}|^2 dy\Big|^2\Big]^{\frac{1}{2}}\\
	&\quad+\|m_{T-t'}^{0,\mu}-m_{T-t}^{0,\mu}\|^2\e\Big[\int_\br |U_T^{t,x,\mu}|^2 dy\Big]\\
	&\quad+\e\Big[\int_{\br} \Big|\int_{\br}\frac{\dd \Phi}{\dd m}(X_{T-t'}^{0,x,\mu},m_{T-t'}^{0,\mu})(\xi)[k^{0,\mu}(T-t',\xi,y)-k^{0,\mu}(T-t,\xi,y)]d\xi\Big|^2 dy\\
	&\quad+\int_{\br} \Big|\int_{\br}[\frac{\dd \Phi}{\dd m}(X_T^{t',x,\mu},m_{T}^{t',\mu})(\xi)-\frac{\dd \Phi}{\dd m}(X_T^{t,x,\mu},m_{T}^{t,\mu})(\xi)]k^{t,\mu}(T,\xi,y)d\xi\Big|^2 dy\Big]\Big\}\\
	&\le C(\gamma,L,T,\|\mu\|_{W^{1,2}(\br)})|t'-t|.
\end{align*}

\section{Proof of Lemma~\ref{forall}}\label{pf_forall}
We continue to restrict ourselves within the one-dimensional case $d=1$. Let $0\le t\le s\le T$, $x\in\br$ and $\mu\in W^{3,2}(\br)$. Note from Proposition~\ref{thm1'} that $m^{t,\mu}(s,\cdot)\in W^{3,2}(\br)$. We first consider $f(s,x,m^{t,\mu}(s,\cdot))-f(t,x,\mu)$. We put $t_i^n:=t+i(s-t)2^{-n}$ for $0\le i\le 2^n$ and $n\geq 1$. Then
\begin{align*}
	f(s,x,m^{t,\mu}(s,\cdot))-f(t,x,\mu)=\sum_{i=0}^{2^n-1}[f(t_{i+1}^n,x,m^{t,\mu}(t_{i+1}^n,\cdot))-f(t_i^n,x,m^{t,\mu}(t_{i}^n,\cdot))].
\end{align*}
Actually,
\begin{align*}
	&f(t_{i+1}^n,x,m^{t,\mu}(t_{i+1}^n,\cdot))-f(t_i^n,x,m^{t,\mu}(t_{i}^n,\cdot))\\
	&=\int_{t_i^n}^{t_{i+1}^n}\frac{\dd f}{\dd r}(r,x,m^{t,\mu}(t^n_{i+1},\cdot))dr\\
	&\quad+\int_{\br} \frac{\dd f}{\dd \mu}(t_i^n,x,m^{t,\mu}(t_{i}^n,\cdot))(\xi)[m^{t,\mu}(t_{i+1}^n,\xi)-m^{t,\mu}(t_{i}^n,\xi)]d\xi\\
	&\quad +\int_{\br} \int_0^1\Big[\frac{\dd f}{\dd \mu}(t_i^n,x,m^{t,\mu}(t_{i}^n,\cdot)+\lambda[m^{t,\mu}(t_{i+1}^n,\cdot)-m^{t,\mu}(t_{i}^n,\cdot)])(\xi)\\
	&\qquad\qquad\quad-\frac{\dd f}{\dd \mu}(t^n_i,x,m^{t,\mu}(t_{i}^n,\cdot))(\xi)\Big]\cdot[m^{t,\mu}(t_{i+1}^n,\xi)-m^{t,\mu}(t_{i}^n,\xi)]d\lambda d\xi\\
	&=\int_{t_i^n}^{t_{i+1}^n}\frac{\dd f}{\dd r}(r,x,m^{t,\mu}(t^n_{i+1},\cdot))dr+\int_{t_i^n}^{t_{i+1}^n}\int_{\br} \frac{\dd f}{\dd \mu}(t^n_i,x,m^{t,\mu}(t_{i}^n,\cdot))(\xi)\frac{\dd m^{t,\mu}}{\dd r}(r,\xi)d\xi dr\\
	&\quad +\int_{t_i^n}^{t_{i+1}^n}\int_{\br}\int_0^1 \Big[\frac{\dd f}{\dd \mu}(t^n_i,x,m^{t,\mu}(t_{i}^n,\cdot)+\lambda[m^{t,\mu}(t_{i+1}^n,\cdot)-m^{t,\mu}(t_{i}^n,\cdot)])(\xi)\\
	&\qquad\qquad\qquad\qquad-\frac{\dd f}{\dd \mu}(t_i^n,x,m^{t,\mu}(t_{i}^n,\cdot))(\xi)\Big]\cdot\frac{\dd m^{t,\mu}}{\dd r}(r,\xi)d\lambda d\xi dr.
\end{align*}
From Lemma~\ref{thm1}, Proposition~\ref{thm1'} and the fact that $f\in C_b^{1,2,1}$, we have
\begin{align*}
	&\Big|\int_{t_i^n}^{t_{i+1}^n}\frac{\dd f}{\dd r}(r,x,m^{t,\mu}(t^n_{i+1},\cdot))-\frac{\dd f}{\dd r}(r,x,m^{t,\mu}(r,\cdot))dr\Big|\\
	&\qquad\le C(\sup_{t^n_i\le r\le t_{i+1}^n}\|m^{t,\mu}(r,\cdot)\|_{ W^{3,2}(\br)})\int_{t_i^n}^{t_{i+1}^n}\|m^{t,\mu}(t^n_{i+1},\cdot)-m^{t,\mu}(r,\cdot)\|_{W^{2,2}(\br)}dr\\
	&\qquad\le C(\gamma,L,T,\|\mu\|_{W^{3,2}(\br)})\int_{t_i^n}^{t_{i+1}^n} |t_{i+1}^n-r|^{\frac{1}{2}} dr\le C(\gamma,L,T,\|\mu\|_{W^{3,2}(\br)})2^{-\frac{3n}{2}}.
\end{align*}
So we have
\begin{align*}
	&\Big|\sum_{i=0}^{2^n-1}\int_{t_i^n}^{t_{i+1}^n}\frac{\dd f}{\dd r}(r,x,m^{t,\mu}(t^n_{i+1},\cdot))dr-\int_{t}^{s}\frac{\dd f}{\dd r}(r,x,m^{t,\mu}(r,\cdot))dr\Big|\\
	&\qquad\le C(\gamma,L,T,\|\mu\|_{W^{3,2}(\br)})2^{-\frac{n}{2}}\stackrel{n\to+\infty}{\longrightarrow}0.
\end{align*}
From Lemma~\ref{thm1}, Proposition~\ref{thm1'} and the fact that $f\in C_b^{1,2,1}$, we have
\begin{align*}
	&\Big|\int_{t_i^n}^{t_{i+1}^n}\int_{\br} \Big[\frac{\dd f}{\dd \mu}(t_i^n,x,m^{t,\mu}(t_i^n,\cdot))(\xi) - \frac{\dd f}{\dd \mu}(r,x,m^{t,\mu}(r,\cdot))(\xi)\Big]\frac{\dd m^{t,\mu}}{\dd r}(r,\xi)  d\xi dr\Big|\\
	&\le \int_{t_i^n}^{t_{i+1}^n}\Big(\int_{\br}\Big|\frac{\dd f}{\dd \mu}(t_i^n,x,m^{t,\mu}(t_i^n,\cdot))(\xi)-\frac{\dd f}{\dd \mu}(r,x,m^{t,\mu}(r,\cdot))(\xi)\Big|^2 d\xi\Big)^{\frac{1}{2}}\\
	&\qquad\times\|\frac{\dd m^{t,\mu}}{\dd r}(r,\cdot)\|_{L^2(\br)} dr\\
	&\le  C(\sup_{t_i^n\le r\le t_{i+1}^n}\|m^{t,\mu}(r,\cdot)\|_{ W^{3,2}(\br)})\sup_{t_i^n\le r\le t_{i+1}^n}\|\frac{\dd m^{t,\mu}}{\dd r}(r,\cdot)\|_{L^2(\br)}\\
	&\qquad \times \int_{t^n_i}^{t^n_{i+1}}|r-t_i^n|^{\frac{1}{2}}+\|m^{t,\mu}(r,\cdot)-m^{t,\mu}(t_i^n,\cdot)\|_{W^{2,2}(\br)} dr\\
	&\le C(\gamma,L,T,\|\mu\|_{W^{3,2}(\br)})\int_{t^n_i}^{t^n_{i+1}}|r-t^n_i|^\frac{1}{2} dr\le C(\gamma,L,T,\|\mu\|_{W^{3,2}(\br)})2^{-\frac{3n}{2}}.
\end{align*}
So we have
\begin{align*}
	&\Big|\sum_{i=0}^{2^n-1}\int_{t_i^n}^{t_{i+1}^n}\int_{\br}  \frac{\dd f}{\dd \mu}(t_i^n,x,m^{t,\mu}(t_i^n,\cdot))(\xi)\frac{\dd m^{t,\mu}}{\dd r}(r,\xi)  d\xi dr\\
	&\qquad-\int_{t}^{s}\int_{\br}  \frac{\dd f}{\dd \mu}(r,x,m^{t,\mu}(r,\cdot))(\xi)\frac{\dd m^{t,\mu}}{\dd r}(r,\xi)  d\xi dr\Big|\\
	&\le C(\gamma,L,T,\|\mu\|_{W^{3,2}(\br)})2^{-\frac{n}{2}}\stackrel{n\to+\infty}{\longrightarrow}0.
\end{align*}
From Lemma~\ref{thm1}, Propositions~\ref{thm1'} and the fact that $f\in C_b^{1,2,1}$, we have
\begin{align*}
	&\Big|\int_{t_i^n}^{t_{i+1}^n}\int_{\br}\int_{\br} \Big[\frac{\dd f}{\dd \mu}(t_i^n,x,m^{t,\mu}(t_{i}^n,\cdot)+\lambda[m^{t,\mu}(t_{i+1}^n,\cdot)-m^{t,\mu}(t_{i}^n,\cdot)])(\xi)\\
	&\qquad-\frac{\dd f}{\dd \mu}(t_i^n,x,m^{t,\mu}(t_{i}^n,\cdot))(\xi)\Big]\cdot\frac{\dd m^{t,\mu}}{\dd r}(r,\xi)d\lambda d\xi dr\Big|\\
	&\le C(\sup_{t_i^n\le r\le t_{i+1}^n}\|m^{t,\mu}(r,\cdot)\|_{ W^{3,2}(\br)})\|m^{t,\mu}(t^n_{i+1},\cdot)-m^{t,\mu}(t^n_i,\cdot)\|_{W^{2,2}(\br)}\\
	&\qquad\times\int_{t_i^n}^{t_{i+1}^n}\|\frac{\dd m^{t,\mu}}{\dd r}(r,\cdot)\|_{L^2(\br)} dr\\
	&\le C(\gamma,L,T,\|\mu\|_{W^{3,2}(\br)})2^{-\frac{n}{2}}\int_{t^n_i}^{t^n_{i+1}}\|\frac{\dd m^{t,\mu}}{\dd r}(r,\cdot)\|_{L^2(\br)} dr.
\end{align*}
So from Lemma~\ref{thm1}, we have
\begin{align*}
	&\Big|\sum_{i=0}^{2^n-1}\int_{t_i^n}^{t_{i+1}^n}\int_{\br}\int_{\br} \Big[\frac{\dd f}{\dd \mu}(t_i^n,x,m^{t,\mu}(t_{i}^n,\cdot)+\lambda[m^{t,\mu}(t_{i+1}^n,\cdot)-m^{t,\mu}(t_{i}^n,\cdot)])(\xi)\\
	&\qquad-\frac{\dd f}{\dd \mu}(t_i^n,x,m^{t,\mu}(t_{i}^n,\cdot))(\xi)\Big]\cdot\frac{\dd m^{t,\mu}}{\dd r}(r,\xi)d\lambda d\xi dr\Big|\\
	&\le C(\gamma,L,T,\|\mu\|_{W^{3,2}(\br)})2^{-\frac{n}{2}}\int_{t}^{s}\|\frac{\dd m^{t,\mu}}{\dd r}(r,\cdot)\|_{L^2(\br)} dr\\
	&\le  C(\gamma,L,T,\|\mu\|_{W^{3,2}(\br)})2^{-\frac{n}{2}}\stackrel{n\to+\infty}{\longrightarrow}0.
\end{align*}
Up to now, we have
\begin{equation}\label{forall_2}
	\begin{split}
		&f(s,x,m^{t,\mu}(s,\cdot))-f(t,x,\mu)\\
		&\qquad=\int_t^s \Big[\frac{\dd f}{\dd r}(r,x,m^{t,\mu}(r,\cdot))+\int_{\br}\frac{\dd f}{\dd \mu}(r,x,m^{t,\mu}(r,\cdot))(\xi)\frac{\dd m^{t,\mu}}{\dd r}(r,\xi)d\xi \Big]ds.
	\end{split}
\end{equation}
We now set $\phi (s,x):=f(s,x,m^{t,\mu}(s,\cdot))$. Then,
\begin{align*}
	\phi(s,X_s^{t,x,\mu})=f(s,X_s^{t,x,\mu},m^{t,\mu}(s,\cdot)),\quad \phi(t,x)=f(t,x,\mu).
\end{align*}
From \eqref{forall_2} we know that
\begin{align*}
	\frac{\dd\phi}{\dd s}(s,x)=\frac{\dd f}{\dd s}(s,x,m^{t,\mu}(s,\cdot))+\int_{\br}\frac{\dd f}{\dd \mu}(s,x,m^{t,\mu}(s,\cdot))(\xi)\frac{\dd m^{t,\mu}}{\dd s}(s,\xi)d\xi.
\end{align*}
Therefore, by applying classical It\^o's formula for $\phi(s,X_s^{t,x,\mu})$, we have \eqref{forall_1}.

\bibliographystyle{amsplain}

\end{document}